\documentclass[a4paper,10pt]{article}

\usepackage{mathrsfs}
\usepackage{showlabels}
\usepackage{amsmath,amsthm,amssymb,esint}
\usepackage{mathtools}
\usepackage{graphicx}
\usepackage{bibentry}
\usepackage{enumitem}
\usepackage{xcolor}
\usepackage[hidelinks]{hyperref}
\usepackage[a4paper,width=160mm,top=30mm,bottom=30mm]{geometry}
\usepackage{subfigure}
\usepackage{todonotes}

\title{Efficient uncertainty quantification for mechanical properties of randomly perturbed elastic rods}

\author{Patrick Dondl\footnote{pwd@math.uni-freiburg.de, Abteilung für angewandte Mathematik, Universität Freiburg, 79104 Freiburg, Germany}, Yongming Luo\footnote{luo.yongming@smbu.edu.cn, Faculty of Computational Mathematics and Cybernetics, Shenzhen MSU-BIT University, China}
, Stefan Neukamm\footnote{stefan.neukamm@tu-dresden.de, Fakultät Mathematik, Technische Universität Dresden, 01062 Dresden, Germany}, Steve Wolff-Vorbeck\footnote{steve.wolff-vorbeck@mathematik.uni-freiburg.de, Abteilung für angewandte Mathematik, Universität Freiburg, 79104 Freiburg, Germany}}
\date{}

\numberwithin{equation}{section}

\newtheorem{theorem}{Theorem}[section]
\newtheorem{lemma}[theorem]{Lemma}
\newtheorem{definition}[theorem]{Definition}
\newtheorem{proposition}[theorem]{Proposition}
\newtheorem{corollary}[theorem]{Corollary}
\newtheorem{assumption}[theorem]{Assumption}
\newtheorem{remark}[theorem]{Remark}
\newtheorem{example}[theorem]{Example}

\makeatletter
\newsavebox\myboxA
\newsavebox\myboxB
\newlength\mylenA

\newcommand*\yoverline[2][0.75]{%
    \sbox{\myboxA}{$\m@th#2$}%
    \setbox\myboxB\null
    \ht\myboxB=\ht\myboxA%
    \dp\myboxB=\dp\myboxA%
    \wd\myboxB=#1\wd\myboxA
    \sbox\myboxB{$\m@th\overline{\copy\myboxB}$}
    \setlength\mylenA{\the\wd\myboxA}
    \addtolength\mylenA{-\the\wd\myboxB}%
    \ifdim\wd\myboxB<\wd\myboxA%
       \rlap{\hskip -0.3\mylenA\usebox\myboxB}{\usebox\myboxA}%
    \else
        \hskip -0.5\mylenA\rlap{\usebox\myboxA}{\hskip 0.5\mylenA\usebox\myboxB}%
    \fi}
\makeatother

\makeatletter

\newcommand*\xoverline[2][0.5]{%
    \sbox{\myboxA}{$\m@th#2$}%
    \setbox\myboxB\null
    \ht\myboxB=\ht\myboxA%
    \dp\myboxB=\dp\myboxA%
    \wd\myboxB=#1\wd\myboxA
    \sbox\myboxB{$\m@th\overline{\copy\myboxB}$}
    \setlength\mylenA{\the\wd\myboxA}
    \addtolength\mylenA{-\the\wd\myboxB}%
    \ifdim\wd\myboxB<\wd\myboxA%
       \rlap{\hskip 0.5\mylenA\usebox\myboxB}{\usebox\myboxA}%
    \else
        \hskip -0.5\mylenA\rlap{\usebox\myboxA}{\hskip 0.5\mylenA\usebox\myboxB}%
    \fi}
\makeatother

\makeatletter

\newcommand*\zoverline[2][0.5]{%
    \sbox{\myboxA}{$\m@th#2$}%
    \setbox\myboxB\null
    \ht\myboxB=\ht\myboxA%
    \dp\myboxB=\dp\myboxA%
    \wd\myboxB=#1\wd\myboxA
    \sbox\myboxB{$\m@th\overline{\copy\myboxB}$}
    \setlength\mylenA{\the\wd\myboxA}
    \addtolength\mylenA{-\the\wd\myboxB}%
    \ifdim\wd\myboxB<\wd\myboxA%
       \rlap{\hskip 0.3\mylenA\usebox\myboxB}{\usebox\myboxA}%
    \else
        \hskip -0.5\mylenA\rlap{\usebox\myboxA}{\hskip 0.5\mylenA\usebox\myboxB}%
    \fi}
\makeatother

\begin{document}

\maketitle
\begin{abstract}

Motivated by an application involving additively manufactured bioresorbable polymer scaffolds supporting bone tissue regeneration, we investigate the impact of uncertain geometry perturbations on the effective mechanical properties of elastic rods. To be more precise, we consider elastic rods modeled as three-dimensional linearly elastic bodies occupying randomly perturbed domains. Our focus is on a model where the cross-section of the rod is shifted along the longitudinal axis with stationary increments. To efficiently obtain accurate estimates on the resulting uncertainty of the effective elastic moduli, we use a combination of analytical and numerical methods. Specifically, we rigorously derive a one-dimensional surrogate model by analyzing the slender-rod $\Gamma$-limit. Additionally, we establish qualitative and quantitative stochastic homogenization results for the one-dimensional surrogate model. To compare the fluctuations of the surrogate with the original three-dimensional model, we perform numerical simulations by means of finite element analysis and Monte Carlo methods.\\

  \textbf{Keywords:} Stochastic homogenization, elastic rods, uncertainty quantification.
  \vspace{1mm}
  
  \textbf{MSC2020:} 74Q05, 74B05,  65C05, 65N30
  
\end{abstract}



\def\Ab{\mathbf{A}}
\def\Bb{\mathbf{B}}
\def\Phib{\boldsymbol{\Phi}}
\def\Psib{\boldsymbol{\Psi}}

\newcommand{\step}[1]{
  \smallskip

  \noindent
  \textbf{Step~#1:}}
\newcommand{\sD}{{\mathscr D}}
\newcommand{\sA}{{\mathscr A}}
\newcommand{\diver}{\operatorname{div}}
\newcommand{\lin}{\operatorname{Lin}}
\newcommand{\curl}{\operatorname{curl}}
\newcommand{\ran}{\operatorname{Ran}}
\newcommand{\kernel}{\operatorname{Ker}}
\newcommand{\la}{\langle}
\newcommand{\ra}{\rangle}
\newcommand{\N}{\mathbb{N}}
\newcommand{\R}{\mathbb{R}}
\newcommand{\C}{\mathbb{C}}
\newcommand{\ld}{\lambda}
\newcommand{\fai}{\varphi}
\newcommand{\n}{\mathbf{n}}
\newcommand{\uu}{{\boldsymbol{\mathrm{u}}}}
\newcommand{\UU}{{\boldsymbol{\mathrm{U}}}}
\newcommand{\buu}{\bar{{\boldsymbol{\mathrm{u}}}}}
\newcommand{\huu}{\hat{{\boldsymbol{\mathrm{u}}}}}
\newcommand{\hro}{\hat{{\boldsymbol{\mathrm{r}}}}}
\newcommand{\ten}{\\[4pt]}
\newcommand{\six}{\\[4pt]}
\newcommand{\nb}{\nonumber}
\newcommand{\hgamma}{H_{\Gamma}^1(\OO)}
\newcommand{\opert}{O_{\varepsilon,h}}
\newcommand{\barx}{\bar{x}}
\newcommand{\barf}{\bar{f}}
\newcommand{\hatf}{\hat{f}}
\newcommand{\xoneeps}{x_1^{\varepsilon}}
\newcommand{\xh}{x_h}
\newcommand{\scaled}{\nabla_{h}}
\newcommand{\scaledb}{\widehat{\nabla}_{1,\gamma}}
\newcommand{\vare}{\varepsilon}
\newcommand{\A}{{\bf{A}}}
\newcommand{\boldb}{{\bf{b}}}
\newcommand{\RR}{{\bf{R}}}
\newcommand{\B}{{\bf{B}}}
\newcommand{\CC}{{\bf{C}}}
\newcommand{\D}{{\bf{D}}}
\newcommand{\K}{{\bf{K}}}
\newcommand{\oo}{{\bf{o}}}
\newcommand{\id}{{\bf{Id}}}
\newcommand{\E}{\mathcal{E}}
\newcommand{\ii}{\mathcal{I}}
\newcommand{\sym}{\mathrm{sym\,}}
\newcommand{\lt}{\left}
\newcommand{\rt}{\right}
\newcommand{\ro}{{\bf{r}}}
\newcommand{\so}{{\bf{s}}}
\newcommand{\e}{{\bf{e}}}
\newcommand{\ww}{{\boldsymbol{\mathrm{w}}}}
\newcommand{\vv}{{\boldsymbol{\mathrm{v}}}}
\newcommand{\xxi}{{\boldsymbol{\xi}}}
\newcommand{\zz}{{\boldsymbol{\mathrm{z}}}}
\newcommand{\U}{{\boldsymbol{\mathrm{U}}}}
\newcommand{\G}{{\boldsymbol{\mathrm{G}}}}
\newcommand{\VV}{{\boldsymbol{\mathrm{V}}}}
\newcommand{\WW}{{\boldsymbol{\mathrm{W}}}}
\newcommand{\T}{{\boldsymbol{\mathrm{T}}}}
\newcommand{\II}{{\boldsymbol{\mathrm{I}}}}
\newcommand{\ZZ}{{\boldsymbol{\mathrm{Z}}}}
\newcommand{\hKK}{{{\bf{K}}}}
\newcommand{\f}{{\bf{f}}}
\newcommand{\g}{{\bf{g}}}
\newcommand{\lkk}{{\bf{k}}}
\newcommand{\tkk}{{{\bf{K}}}}
\newcommand{\W}{{\boldsymbol{\mathrm{W}}}}
\newcommand{\hW}{\widehat{{\boldsymbol{\mathrm{W}}}}}
\newcommand{\Y}{{\boldsymbol{\mathrm{Y}}}}
\newcommand{\EE}{{\boldsymbol{\mathrm{E}}}}
\newcommand{\F}{{\bf{F}}}
\newcommand{\spacev}{\mathcal{V}}
\newcommand{\spacevg}{\mathcal{W}^{\gamma}}
\newcommand{\spacevb}{\bar{\mathcal{V}}^{\gamma}(\Omega\times S)}
\newcommand{\spaces}{\mathcal{G}}
\newcommand{\spacesg}{\mathcal{G}^{\gamma}}
\newcommand{\spacesb}{\bar{\mathcal{S}}^{\gamma}(\Omega\times S)}
\newcommand{\skews}{H^1_{\barx,\mathrm{skew}}}
\newcommand{\kk}{\mathcal{K}}
\newcommand{\OO}{O}
\newcommand{\bhe}{{\bf{B}}_{\vare,h}}
\newcommand{\pp}{{\mathbb{P}}}
\newcommand{\ff}{{\mathcal{F}}}
\newcommand{\mWk}{{\mathcal{W}}^{k,2}(\Omega)}
\newcommand{\mWa}{{\mathcal{W}}^{1,2}(\Omega)}
\newcommand{\mWb}{{\mathcal{W}}^{2,2}(\Omega)}
\newcommand{\twos}{\xrightharpoonup{2}}
\newcommand{\twoss}{\xrightarrow{2}}
\newcommand{\bw}{\bar{w}}
\newcommand{\br}{\bar{{\bf{r}}}}
\newcommand{\bz}{\bar{{\bf{z}}}}
\newcommand{\tw}{{W}}
\newcommand{\tr}{{{\bf{R}}}}
\newcommand{\tz}{{{\bf{Z}}}}
\newcommand{\lo}{{{\bf{o}}}}
\newcommand{\hoo}{H^1_{00}(0,L)}
\newcommand{\ho}{H^1_{0}(0,L)}
\newcommand{\hotwo}{H^1_{0}(0,L;\R^2)}
\newcommand{\hooo}{H^1_{00}(0,L;\R^2)}
\newcommand{\hhooo}{H^1_{00}(0,1;\R^2)}
\newcommand{\dsp}{d_{S}^{\bot}(\barx)}
\newcommand{\LB}{{\bf{\Lambda}}}
\newcommand{\LL}{\mathbb{L}}
\newcommand{\mL}{\mathcal{L}}
\newcommand{\mhL}{\widehat{\mathcal{L}}}
\newcommand{\loc}{\mathrm{loc}}
\newcommand{\tqq}{\mathcal{Q}^{*}}
\newcommand{\tii}{\mathcal{I}^{*}}
\newcommand{\Mts}{\mathbb{M}}
\newcommand{\pot}{\mathrm{pot}}
\newcommand{\tU}{{\widehat{\bf{U}}}}
\newcommand{\tVV}{{\widehat{\bf{V}}}}
\newcommand{\pt}{\partial}
\newcommand{\bg}{\Big}
\newcommand{\hA}{\widehat{{\bf{A}}}}
\newcommand{\hB}{\widehat{{\bf{B}}}}
\newcommand{\hCC}{\widehat{{\bf{C}}}}
\newcommand{\hD}{\widehat{{\bf{D}}}}
\newcommand{\fder}{\partial^{\mathrm{MD}}}
\newcommand{\Var}{\mathrm{Var}}
\newcommand{\pta}{\partial^{0\bot}}
\newcommand{\ptaj}{(\partial^{0\bot})^*}
\newcommand{\ptb}{\partial^{1\bot}}
\newcommand{\ptbj}{(\partial^{1\bot})^*}
\newcommand{\geg}{\Lambda_\vare}
\newcommand{\tpta}{\tilde{\partial}^{0\bot}}
\newcommand{\tptb}{\tilde{\partial}^{1\bot}}
\newcommand{\vphi}{{\xi}}
\newcommand{\hmac}{\mathbb{H}_{\mathrm{macro}}}
\newcommand{\hcor}{\mathbb{H}_{\mathrm{cor}}}
\newcommand{\hbb}{\mathbb{H}}
\newcommand{\pbb}{\mathbb{P}}
\newcommand{\pbba}{\mathbb{P}_0}
\newcommand{\pbbc}{\mathbb{P}^\bot_0}
\newcommand{\lc}{{\boldsymbol{\mathrm{b}}_{1234}}}
\newcommand{\trace}{\mathrm{tr\,}}
\newcommand{\summ}{\mathrm{sum}}
\newcommand{\energy}{\mathcal{E}}
\newcommand{\mB}{\mathcal{B}}
\newcommand{\mT}{\mathcal{T}}
\newcommand{\mW}{\mathcal{W}}
\newcommand{\mE}{\widehat{\mathcal{E}}}
\newcommand{\Qel}{Q_{\mathrm{el}}}
\newcommand{\Qres}{Q_{\mathrm{res}}}
\newcommand{\hcorg}{\mathbb{H}_{\gamma,\mathrm{cor}}}
\newcommand{\Qelg}{Q_{\gamma,\mathrm{el}}}
\newcommand{\Qresg}{Q_{\gamma,\mathrm{res}}}

\tableofcontents

\section{Introduction}  
For a period of time now, additive manufacturing  has been a widespread method in a huge variety of engineering applications. Overall, its nearly limitless design freedom has enabled the practice of trial and error approaches in many fields like bio-engineering or civil-engineering, see for instance~\cite{teo2015novel,goh2015novel,schuckert2009mandibular,schantz2006cranioplasty} or~\cite{paolini2019additive}. In this context, an important field of application is the construction of rod-shaped elastic solids where one considers three-dimensional design structures $O_{h} = (0,L) \times hS \subset \mathbb{R}^3$. Here, $0<h \ll L$, denotes the thickness and $S\subset \mathbb{R}^2$ the cross-section of the rod. In general, additively manufactured rods need to maintain the structural integrity under mechanical loading conditions subject to further constraints on the shape and porosity of the structure. This leads to competing optimization goals that can be addressed by finite element analysis using computer-aided design models as input. However, the reliability of the finite element analysis can be impeded by the anomalies introduced during the fabrication process leading to marked differences between the mechanical properties of the optimal design $O_{h}$ and the printed object.

In~\cite{valainis2019integrated}, for instance, the authors established a workflow for melt-extrusion based 3D-printing to produce personalized (rod-shaped) bone scaffolds with triply periodic minimal surface architecture. Moreover, they conducted numerical and mechanical experiments that showed a significant variability in the effective Young's modulus of the printed objects leading to uncertainty in the slope of the stress-strain curve in a tensile test. In the context of additive manufacturing, small-scale variations of the material properties (e.g. density fluctuations as experimentally observed in~\cite{craft2018material}) and mesoscopic geometric deviations as, for instance, observed in~\cite{poh2018optimizing} and~\cite{khanzadeh2018quantifying} can be considered as main sources of uncertainty. These errors introduced during the printing process lead to a marked demand of uncertainty quantification in the pre-production process~\cite{moller2000fuzzy,owhadi2013optimal}.

In the present paper we consider a rod with geometric uncertainties. Our goal is to quantify the impact of those uncertainties on effective mechanical properties of the rod by combining analytical and numerical methods.
In our analysis we model the rod $O^h$ as a linearly elastic body and thus consider elastic energy functionals of the form
\begin{equation}\label{general_average_energy}
  \E(\mathbf{v};O^h):=\frac{1}{h^2}\frac1L\int_{O^h}Q(\nabla \mathbf{v}),
\end{equation}
where $\mathbf{v}:O^h\to\R^3$ denotes the displacement and $Q(\mathbf{F})=\mathbb L\mathbf{F}:\mathbf{F}$ the elastic energy density with $\mathbb L$ the elasticity tensor. In particular, we are interested in effective mechanical properties that are determined by minimizing the elastic energy subject to constraints on the displacement $\mathbf{v}$. A prototypical example of such a property is the effective Young's modulus which can be determined by minimizing $\E(\mathbf{v};O^h)$ over displacements with compression boundary conditions, i.e.
\begin{equation}\label{3D_min_prob}
  E(O^h) =\{\inf_{\mathbf{v}} \E(\mathbf{v};O^h): \ \mathbf{v}\in H^1(O^h;\R^3),\mathbf{v}(0,\cdot)=0, \mathbf{v}(L,\cdot)=\mathbf{e}_1\},
\end{equation}
which being a decisive measure in the analysis of mechanical stability of the rod. Moreover, motivated by the layerwise additive manufacturing~\cite{valainis2019integrated}, we focus on a statistical model for geometry perturbations that are caused by shifts of the $x_1$-layers $O^h\cap\{x_1\}\times\R^2$ by a translation of the form
\begin{equation}\label{eq:randomshift}
  \bg(0,\frac{h}{L}\int_0^{x_1}\Phi^\vare(t)\,dt\bg),
\end{equation}
where the translation-increment $\Phi^\vare=\Phi^\vare(\omega,t)$ is a stationary and ergodic, bounded, $\R^2$-valued random field on $\R$ that rapidly decorrelates on lengths larger than $\vare$ and that is defined on a probability space $(\Omega,\mathcal F,\mathbb P)$; we refer to Section~\ref{three_dimensional_model}, where we explain the model for geometry perturbations in detail. In this context, we also denote the randomly perturbed domain by $O^{\vare,h}(\omega)$ with $\omega$ a random sample from $(\Omega,\mathcal F,\mathbb P)$. One of our goals is to obtain sufficiently accurate estimates on the threshold probabilities
\begin{align}\label{eq:uq_est_intr}
    \mathbb{P}\left(|E(O_{h})-E(O^{\vare,h}(\omega))| \geq a \right)
\end{align}
for the effective Young's modulus $E$ with given threshold error $a>0$. The quantification of a threshold probability~\eqref{eq:uq_est_intr} is crucial in several applications, such as in the usage of additively manufactured porous bone tissue scaffolds made from bioresorbable polymer (e.g. polycarprolactone)~\cite{teo2015novel,goh2015novel,schuckert2009mandibular,schantz2006cranioplasty,valainis2019integrated}.

For solving the minimization problem that underlies the definition of $E$, established numerical techniques based on finite element methods for 3D solids involving the tetrahedralization of $O^{\vare,h}(\omega)$ combined with Monte Carlo simulations can be used, see for example~\cite{krumscheid2020quantifying,cliffe2011multilevel}. Unfortunately, such methods usually exhibit extremely high computational costs as they involve discretizations of three-dimensional structures and, therefore, often turn out to be impractical in practice. This high computational effort is a special limiting factor, as besides uncertainties of aleatoric type (following a well-characterized probability distribution) some errors introduced in the printing process are more difficult to quantify and thus fall in the epistemic category, where for example only bounds on probabilities can be established~\cite{mahnken2017variational,DRIESCHNER2020107106,FREITAG202081,kastian2020two}. In our case, epistemic uncertainties indicate a lack of knowledge on the specific distribution of the translation-increment $\Phi^\vare$. This may be due to fuzziness in the parameters of the distribution of $\Phi^\vare$ leading to fuzzy probability based random variables~\cite{puri1993fuzzy,moller2000fuzzy,fina2020polymorphic,mahnken2017variational}.

Thus there is a great demand for highly simplified surrogate models, i.e., an approximation of the problem in~\eqref{3D_min_prob} by a reduced problem that can be solved in a numerically efficient and less time consuming way. In Section~\ref{eff_uq}, we therefore derive a one-dimensional approximation of~\eqref{general_average_energy} by the usage of dimension reduction. More precisely, as an application of our main analytical results stated in Section~\ref{section:main results}, we prove that the elastic energy functional $\E(O^{\vare,h})$ in the limit of vanishing thickness $h\to 0$ $\Gamma$-converges to a one-dimensional energy functional of the form
\begin{equation*}
  \E^{\vare}(\omega,\bar{u},\ro)=\fint_0^L
  Q^{{\rm rod}}\Big(\begin{pmatrix}
    \pt_1\bar u+\frac{\ro_3\Phi^\vare_{1}-\ro_2\Phi^\vare_{2}}{L}\\
    \pt_1\ro
  \end{pmatrix}
  \Big)\,dx_1,
\end{equation*}
where $Q^{\rm rod}$ is a quadratic form determined from $Q$, and the two components of the translational increment $\Phi^{\vare}$ from \eqref{eq:randomshift}. The kinematic variables are given by $\bar u:(0,L)\to\R$ that describes the longitudinal extension or compression,
and by $\ro=(\ro_1,\ro_2,\ro_3):(0,L)\to\R^3$, which describes the in-plane torsional $\ro_1$ and flexional displacements $(\ro_2,\ro_3)$, respectively. For more details we refer to Sections~\ref{eff_uq} and~\ref{section:main results}.

Based on this, we obtain a one-dimensional surrogate model for the effective Young's modulus $E(O^{\vare,h}(\omega))$ in form of the minimization problem
\begin{equation}\label{1D_surr_SN}
  \begin{aligned}
    E^{\vare}(\omega)=\min\bg\{ \E^\vare(\omega,\bar u,\ro): \bar u \in \frac{x_1}{L}+H_0^1(0,L), \ro\in H_0^1(0,L;\R^3), \fint_0^L(\ro_2,\ro_3)\,dx_1=0\bg\}.
\end{aligned}
\end{equation}
Solving~\eqref{1D_surr_SN} requires only the solution of an ordinary differential system instead of a partial differential equation in three space dimensions and thus leads to a marked reduction of computational effort.
The practicability of the surrogate model is demonstrated in Section~\ref{1D_surrogate_pract}, where an explicit system of ordinary differential equations is depicted and a comparison of the computational effort for simulating the 3D model and its corresponding 1D surrogate model is performed. We further numerically study the convergence $E(O^{\vare,h}(\omega))\to E^\vare(\omega)$ as $h\to 0$ and deduce that the 1D surrogate substitute quantifies the fluctuation of $E(O^{\vare,h}(\omega))$ around its mean. From this, we infer a multi-fidelity approach for the computation of  $E^{\vare}(\omega)$ that combines the accuracy of the 3D model with the efficiency of the 1D surrogate model leading to a decent approximation of the effective Young's modulus for modest computational effort. Comparable surrogate models for the computation of effective mechanical properties of randomly perturbed slender elastic rods have not existed to that date and thus the present study noticeably contributes to the field of efficient uncertainty quantification.

The setting described so far is only a special example of a more general model class that we consider in the paper.
In particular, in Sections~\ref{eff_uq} and \ref{section:main results} we consider more general boundary conditions including also torsion and flexion of the rod and material properties that randomly oscillate in $x_1$-direction. In this setting we prove the $\Gamma$-convergence of the functionals $\E(O^{\vare,h}(\omega))$ as $\varepsilon,h\to 0$, where the limits $h\to0$ and $\varepsilon\to 0$ correspond to dimension reduction and homogenization respectively. Rigorous derivations of effective models for dimension reduction \cite{FJM2002,AcerbiButtazzoPercivale1991,MoraMueller2003,MoraMueller2004,FJM2006,Scardia2006} and homogenization \cite{Tartar1977,AlainLionsPapa1978,Nguetseng1989,Allaire1992,PapaVaradhan,Kozlov1979,CioranescuBook,JikovKozlovOleinik,ArmstrongBook,ShenBook,Neukamm2018lecturenotes} problems have been nowadays intensively studied. Following the theories introduced in the previously mentioned works, we consider in this paper both sequential limits
$$\lim_{\varepsilon\to 0}\lim_{h\to 0}\E(O^{\vare,h}(\omega))\quad\text{and}\quad\lim_{h\to 0}\lim_{\varepsilon\to 0}\E(O^{\vare,h}(\omega)),$$
which correspond to homogenization after dimension reduction and vice versa.

Moreover, apart from the qualitative convergence results we also establish quantitative convergence results for the 1D surrogate substitute $E^{\vare}(\omega)$ of $E(O^{\vare,h}(\omega))$ to a deterministic proxy $E^0$ as $\vare\to 0$ based on the spectral gap assumption (Assumption \ref{sg assumption}). It is worth noting that a major obstacle for deriving quantitative convergence results in stochastic homogenization lies in the fact that the stochastic correctors generally possess much wilder behavior than the periodic ones, hence applications of certain advanced and technical large-scale regularity theories (\cite{GNOInvent,GloriaNeukammOttoMIlan}) often become necessary. Interestingly, thanks to the one-dimensional nature of the model we are able to give a simplified proof for the quantitative results based on the precise form of the correctors and no large-scale regularity theories are needed. For further details, we refer to Section \ref{sec:quant} below.

In practical applications, one often encounters the interesting situations where homogenization and dimension reduction take place simultaneously. To be mathematically more precise, we may also consider the $\Gamma$-limits of $\E(O^{\vare,h}(\omega))$ in the case where $\vare=\vare(h)$ is a parameter of $h$ satisfying
\begin{align*}
    \lim_{h\to 0}\vare(h)=0,\quad\lim_{h\to 0}\frac{h}{\vare(h)}=\gamma\in[0,\infty].
\end{align*}
We address this in a forthcoming paper.

\subsubsection*{Outline of the paper}
The paper is organized as follows: In Section~\ref{eff_uq} we introduce a general, linear elastic three-dimensional model for the computation of effective mechanical properties of a thin elastic rod and demonstrate the feasibility of the surrogate model showing the marked reduction in computational effort. In Section \ref{section:main results} we then formulate the main analytical results of the paper. The proofs of the main analytical results are finally given in Section~\ref{section:proof of the main resuls}. For the reader's convenience, a self-contained introduction on the probabilistic framework invoked in this paper will be given in Appendix~\ref{appendix: two scale conv}.

\section{Efficient uncertainty modelling}\label{eff_uq}
\subsection{The three-dimensional model}\label{three_dimensional_model}
We introduce our three-dimensional model of a rod with randomly perturbed geometry. The unperturbed reference domain (with upscaled cross-section) is denoted by
\begin{equation*}
  O:=(0,L)\times S,
\end{equation*}
where $L>0$ denotes the length of the rod and $S\subset\R^2$ denotes the cross-section of the rod. We use the notation $x=(x_1,x_2,x_3)=(x_1,\bar x)$ for the components of $x$. Throughout the paper we suppose that
\begin{equation}\label{cancelation}
  S\subset\R^2\text{ is a bounded Lipschitz domain with }
  \int_S x_2\,d\barx=\int_S x_3\,d\barx=\int_S x_2 x_3\,d\barx =0.
\end{equation}
Note that this symmetry property is not a restriction, since it can always be achieved by rotating and translating  $S$. We consider linearly elastic (possibly heterogeneous) materials:
\begin{definition}[Material class]
  For $0<\alpha_1\leq\alpha_2$ we denote by $\mathcal Q(\alpha_1,\alpha_2)$ the set of all quadratic forms $Q:\R^{3\times 3}\to\R$ such that
  \begin{align}
    \alpha_1|\sym \F|^2\leq Q(\F)\leq \alpha_2|\sym \F|^2\label{assump on Q}
  \end{align}
  for all $\F\in\R^{3\times 3}$. The unique symmetric fourth order tensor $\mathbb L$ satisfying $Q(\F)=\mathbb L\F:\F$ for all $\F\in\R^{3\times 3}$
  is called the elasticity tensor associated with $Q$. We call a fourth order tensor $\mathbb L$ an elasticity tensor of class $\mathcal Q(\alpha_1,\alpha_2)$, if it is associated to some $Q\in\mathcal Q(\alpha_1,\alpha_2)$.
\end{definition}
We consider material heterogeneities and perturbations of $O$ that randomly oscillate in the $x_1$-direction in a stationary and ergodic way. For the precise definition we introduce the following functional analytic setting:

\begin{assumption}\label{assumption of prob}
  Let $(\Omega,\mathcal{F},\mathbb{P})$ be a separable probability space, and $\tau:\Omega\times\R\to\Omega$ be a one-dimensional shift group satisfying:
  \begin{itemize}
\item[(P1)] \textit{Group property}: $\tau_0=\mathrm{id}$ and $\tau_{x+y}=\tau_x\circ\tau_y$ for all $x,y\in\R$.

\item[(P2)]\textit{Measure preservation}: $\mathbb{P}(\tau_x F)=\mathbb{P}(F)$ for all $F\in\mathcal{F}$ and all $x\in\R$.

\item[(P3)]\textit{Measurability}: $(\omega,x)\mapsto \tau_x\omega$ is $(\mathcal{F}\otimes \R,\mathcal{F})$-measurable.

\item[(P4)]\textit{Ergodicity}: For all $F\in \mathcal{F}$ satisfying $\tau_x F=F$ for all $x\in\R$ we have $\mathbb{P}(F)\in\{0,1\}$.
\end{itemize}
\end{assumption}
\begin{remark}\label{two expectation}
\normalfont
For a random variable $f:\Omega\to\R$, we use both the notations $\la f\ra$ and $\int_{\Omega}f(\omega)\,d\mathbb{P}(\omega)$ to interpret the expectation of $f$. As we shall see, the former one will be a more convenient choice for formulating our numerical results (it is much shorter), while the latter one is a more practical notation for formulating results which make use of the two-scale convergence theories (spatial domain and probability space are distinguished in a clearer way).
\end{remark}

\begin{remark}[Stationary random field and stationary extension]\label{D:stationary}
  \normalfont
  Let $(\Omega,\mathcal{F},\mathbb{P})$ satisfy Assumption \ref{assumption of prob} and let $\varphi:\Omega\to\R$ be a random variable. Then $S\varphi:\Omega\times\R\to\R$, $S\varphi(\omega,x):=\varphi(\tau_x\omega)$ defines a random field (i.e., a measurable function on $\Omega\times \R$). We call it the \textit{stationary extension} of $\varphi$, see Lemma~\ref{L:stat} for details. A \textit{stationary random field} is a random field $\psi:\Omega\times\R\to\R$ that can be represented in the form $\psi=S\varphi$ for some random variable $\varphi$. In our setting, geometry perturbations and materials properties are described with help of stationary random fields.
\end{remark}

We particularly use the structure introduced in Assumption~\ref{assumption of prob} to model randomly perturbed reference configurations $O^{\vare,h}(\omega)$, $\omega\in\Omega$; here, $h>0$ and $\vare>0$ stand for the thickness of the rod and a scaling factor for the correlation length, respectively. The randomly perturbed domain is obtained from the unperturbed domain $O^h=(0,L)\times hS$ by randomly shifting the layers $\{x_1\}\times hS$, $x_1\in(0,L)$.
Since the mechanical properties do not change when globally translating the domain, we may assume w.l.o.g. that the first layer, i.e.,  $\{0\}\times hS$, is unchanged.
Furthermore we assume that the increments form layer to layer are stationary.
This leads to a model where the layer at position $x_1$ is translated by a vector of the from $\Big(0,\tfrac{h}{L}\int_0^{x_1}\Phi(\tau_{\frac{t}{\vare}}\omega)\,dt\Big)$ where $\Phi$ denotes a $\R^2$-valued random variable.

In summary, our assumptions on the material heterogeneity and the geometry are the following:

\begin{assumption}[Material law and perturbation]\label{assumption of model}\mbox{}
  Let Assumption~\ref{assumption of prob} be satisfied and assume that the quadratic form $Q:\Omega\times\R^{d\times d}\to\R$, and the randomly perturbed domain $O^{h,\vare}$ satisfy the following properties:
  \begin{itemize}
  \item [(A1)] $Q$ is measurable and there exist $0<\alpha_1\leq \alpha_2$ such that $Q(\omega,\cdot)\in\mathcal Q(\alpha_1,\alpha_2)$ for $\mathbb P$-a.a.~$\omega\in\Omega$.
  \item[(A2)] There exists a bounded random variable $\Phi:\Omega\to\R^2$ such that
    \begin{align*}
      O^{\vare,h}(\omega)=\Psi^{\vare,h}\Big(\omega,O\Big)\qquad\text{for }\mathbb P\text{-a.a.~}\omega\in\Omega,
    \end{align*}
    where
    \begin{align}\label{random translation}
      \Psi^{\vare,h}(\omega,x)=\Big(x_1,h\big(\barx+\frac{1}{L}\int_0^{x_1}\Phi(\tau_{\frac{t}{\vare}}\omega)\,dt\big)\Big).
    \end{align}
\end{itemize}
\end{assumption}
Let us anticipate that in addition to Assumption~\ref{assumption of model} we shall assume that the random perturbation is small in the sense that
\begin{equation}\label{eq:phic}
  \|\Phi\|_{L^\infty(\Omega)}\leq c_S,
\end{equation}
where $c_S$ denotes a constant that we can choose only depending on the cross-section $S$, see Proposition~\ref{prop compact}.
\medskip

In the following, let $Q$ and $O^{h,\vare}$ be as in Assumption~\ref{assumption of model}.
The scaled elastic energy of the perturbed rod subject to a (scaled) displacement $\vv\in H^1(O^{\vare,h}(\omega);\R^3)$ is given by
\begin{align}\label{def of energy2a}
 \E^{\vare,h}(\omega,\vv)\coloneqq \frac{1}{h^{2}}\frac1L\int_{O^{\vare,h}(\omega)}Q(\tau_{\frac{x_1}{\vare}}\omega,\nabla \vv(x))\,dx.
\end{align}
We are interested in the minimization problem
\begin{align}\label{intr:eq2a}
E^{\vare,h}(\omega):=\inf_{\vv \in H_{\mathcal{B}}^1(O^{\vare,h}(\omega);\R^3)} \E^{\vare,h}(\omega,\vv),
\end{align}
which models a mechanical test of the perturbed rod. The test is specified with help of the space
$H_{\mathcal{B}}^1(O^{\vare,h}(\omega);\R^3)$, which is a closed subspace of $H^1(O^{\vare,h}(\omega);\R^3)$ and defined via a set of boundary conditions imposed on the bottom face of perturbed domain, i.e., the set
\begin{align*}
S_0^{\vare,h}(\omega):=\overline{O^{\varepsilon,h}}(\omega)\cap(\{0\}\times\R^2),
\end{align*}
and on the top face,
\begin{align*}
S_L^{\vare,h}(\omega):=\overline{O^{\varepsilon,h}}(\omega)\cap(\{L\}\times\R^2).
\end{align*}

\begin{definition}[Boundary conditions]\label{D bc 3D}
  Let $\mathbf{t}\in\R^3$ , $\A_0,\A_L\in\R^{2\times 2}$ and $\K_0,\K_L\in \R_{\rm skw}^{2\times 2}$ be given. We denote by $H_{\mathcal{B}}^1(O^{\vare,h}(\omega);\R^3)$ the space of displacements $\vv\in H^1(O^{\vare,h}(\omega);\R^3)$ satisfying
\begin{subequations}\label{bc}
\begin{alignat}{3}
\vv(x)&=\bg(0,(\A_0+\frac1h\K_0)\bar{x}\bg)^T&&\quad\text{on $S_0^{\vare,h}(\omega)$},\label{bc0}\\
\vv(x)&=\mathbf{t} +\bg(0,(\A_L+\frac1h\K_L)\bar{x}\bg)^T&&\quad\text{on $S_L^{\vare,h}(\omega)$}\label{bcL}.
\end{alignat}
\end{subequations}

\end{definition}

Mechanically, the vector ${\bf{t}}$ represents a stretch or compression while $(\A_0+\frac1h\K_0)\bar{x}$ and $(\A_L+\frac1h\K_L)\bar{x}$ describe the linearized dilation and twist at the bottom and the top face of $O^{\vare,h}(\omega)$, respectively. Note that this corresponds to a dilation of order $O(h)$ and a twist of order $O(1)$. 

A schematic description of the boundary conditions is depicted in Fig. \ref{fig:schematic_descr}.

\begin{figure}[ht]
\centering
\includegraphics[width=0.55\textwidth]{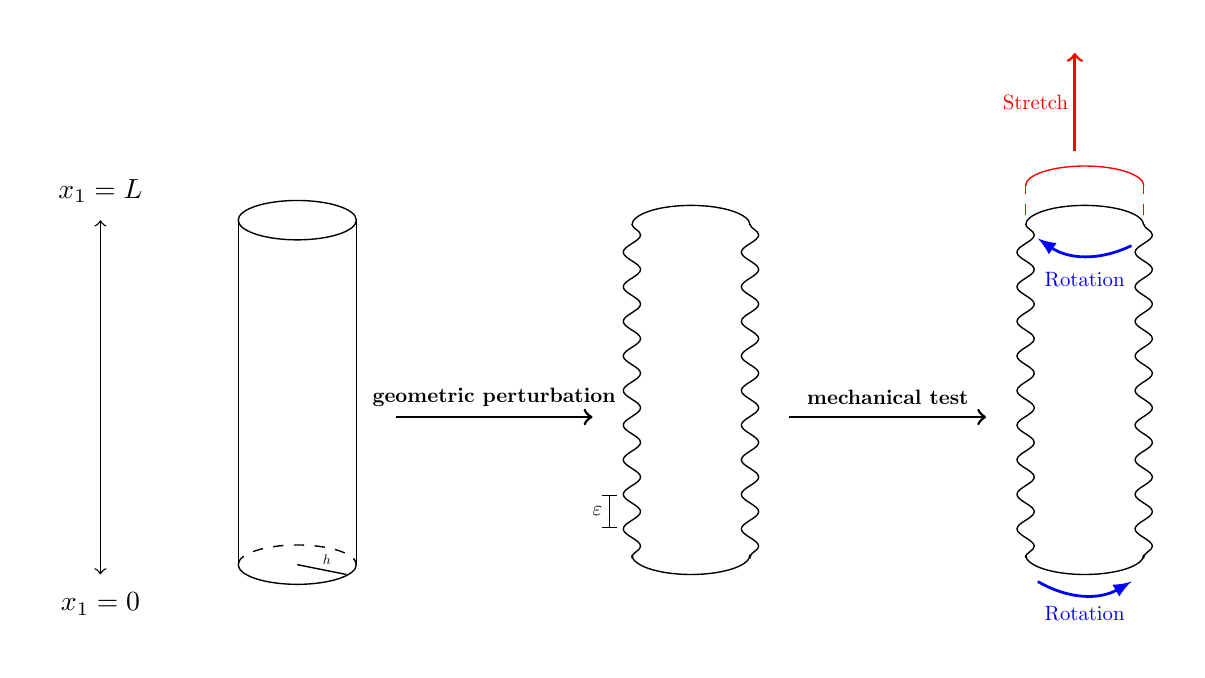}
 \caption[]{Stretching/compression and rotation at the bottom and top of a randomly perturbed cylinder. Perturbation is caused by layer shifting.}
 \label{fig:schematic_descr}
 \end{figure}

 \begin{example}
 \begin{itemize}
     \item[(a)] Tensile test: taking ${\bf{t}}=(1,0,0)$ and $\A_0=\K_0=\A_L=\K_L =0$ for the boundary conditions \eqref{bc} we describe the situation of pure tension of the rod. The tension is then proportional to the effective Young's modulus determined by \eqref{3D_min_prob}.
     \item[(b)] Coupling tension and twist: taking ${\bf{t}}=(1,0,0)$ and $\A_0=\K_0=\A_L=0$, and $\K_L=
       \left(\begin{smallmatrix}
         0&-0.5\\
         0.5&0
       \end{smallmatrix}\right)$ tension is coupled with a rotation of the top face $S_L^{\vare,h}$ of $O^{\vare,h}$, see Fig.~\ref{fig:schematic_descr}.
 \end{itemize}
 \end{example}

\subsection{The one-dimensional surrogate model}\label{1D_surrogate_pract}
We shall see and rigorously prove that in the limit $h\to 0$ the elastic energy converges to a one-dimensional linear rod model, where the configuration of the rod is described by a pair $(\bar{u},\ro)\in H^1(0,L)\times H^1(0,L;\R^3)$. Here, $\bar{u}$ describes the longitudinal displacement and $\ro=(\ro_1,\ro_2,\ro_3)$ the in-plane torsional ($\ro_1$) and flexional ($(\ro_{2},\ro_{3})$) displacements, respectively. The one-dimensional effective elastic energy is given by the following functional
\begin{align}\label{energy functional 1D}
  \E^{\vare}(\omega,\bar{u},\ro)&:=\fint_0^L
                                    Q^{\rm rod}\Big(\tau_{\frac{x_1}{\vare}}\omega,
                                    \begin{pmatrix}
                                      \pt_1\bar u+\tfrac1{L}(\ro_3\Phi_1(\tau_{\frac{x_1}{\vare}}\omega)-\ro_2\Phi_2(\tau_{\frac{x_1}{\vare}}\omega))\\
                                      \pt_1\ro
                                    \end{pmatrix}
  \Big)\,dx_1
\end{align}
where $Q^{\rm rod}(\omega,\cdot):\R^4\to\R$ is defined by the relaxtion formula
\begin{equation}\label{eq:def of Qel}
  Q^{\rm rod}(\omega,\xxi):=\inf_{\varphi\in H^1(S;\R^3)}\int_SQ\Big(\omega,
(\xxi_1\e_1+\xxi_{234}\wedge(0,\bar x))\otimes\e_1+(0,\nabla_{\bar x}\varphi)\Big)\,d\bar x
\end{equation}
and ``$\wedge$'' denotes the vector product in $\R^3$. We note that $Q^{\rm rod}(\omega,\cdot)$ is a positive definite quadratic form (see Lemma~\ref{L:quadr}) and models the effective elastic properties of the rod.  In the special case of an isotropic material the following explicit representation holds:

\begin{proposition}\label{precise q1D}
  Let $Q$ be as in  Assumptions \ref{assumption of model} and assume isotropicity, i.e.,
  \begin{equation}
    Q(\omega,\A)=2\mu(\omega)|\A|^2+\ld(\omega) (\trace \A)^2\label{2.15}
  \end{equation}
  for $\mathbb P$-a.a.~$\omega\in\Omega$, all $\A\in\R^{3\times 3}_\sym$, and Lam\'e parameters $\mu,\lambda\in L^\infty(\Omega)$.
Then $Q^{\rm rod}$ defined in \eqref{eq:def of Qel} takes the form
\begin{align}
Q^{\rm rod}(\omega,\xxi)&=\frac{\mu(\omega)(3\ld(\omega)+2\mu(\omega))}{\ld(\omega)+\mu(\omega)}\bg(|S|\xxi_1^2+\bg(\int_S x_3^2\,d\barx\bg)\xxi_3^2+\bg(\int_S x_2^2\,d\barx\bg)\xxi_4^2\bg)
\nonumber\\&+\mu(\omega)\bg(\int_S\bg((x_3-\pt_2\varphi_{\rm aff}(\barx))^2+(x_2+\pt_3\varphi_{\rm aff})^2(\barx)\bg)\,d\barx\bg)\xxi_2^2\label{form of q1D isotropic}
\end{align}
for $\xxi\in\R^4$, where $\varphi_{\rm aff}\in H^1(S)$ is a solution of
\begin{equation}\label{non-normalized}
\left\{
\begin{array}{ll}
-\Delta_{S}\varphi_{\rm aff}=0&\text{in $S$},\\
\\
(\pt_2\varphi_{\rm aff},\pt_3\varphi_{\rm aff})\cdot\nu=(x_3,-x_2)\cdot\nu&\text{on $\pt S$},
\end{array}
\right.
\end{equation}
Here, $\nu$ denotes the outer unit normal on $\partial S$. Furthermore, if $S$ is a disc, then $\varphi_{\rm aff}$ can be chosen equal zero.
\end{proposition}
See Section \ref{sec precise q1D} for the proof.
\medskip

Next, we discuss the asymptotics of $E^{\vare,h}(\omega)$ (see \eqref{intr:eq2a}) which is the main quantity of interest in our paper. For $h\to 0$, $\Gamma$-convergence of $\E^{\vare,h}(\omega,\vv)$ to the one dimensional energy $\E^{\vare}(\omega,\bar{u},\ro)$ is obtained in Section~\ref{section:main results} as a particular case of Theorem~\ref{main theorem qual}. From this we deduce the convergence of $E^{\vare,h}(\omega)$ to
\begin{align}\label{1D_surr}
    E^{\vare}(\omega):= \inf_{(\bar{u},\ro)}\E^{\vare}(\omega,(\bar{u},\ro))
\end{align}
as $h\to 0$, where the infimum runs over all $(\bar u,\ro)$ with
\begin{equation}\label{defW}
  (\bar u,\ro)\in (\bar u_{\rm aff},\ro_{\rm aff})+H^1_0(0,L)\times H^1_{00}(0,L;\R^3).
\end{equation}
The set above is an affine space that encodes the effective 1D boundary conditions that emerge from the 3D boundary conditions of Definition~\ref{D bc 3D}. It invokes the space
\begin{align}
  H_{00}^1(0,L)&:=\{\ro\in H_0^1(0,L;\R^3):\fint_0^L (\ro_2,\ro_3)(x_1)\,dx_1=0\},\label{def of Sobolev subspace*}
\end{align}
and the boundary conditions $\bar{u}_{\mathrm{aff}}$ and $\ro_{\mathrm{aff}}$ given by
 \begin{align}\label{affine function def}
\bar u_{\mathrm{aff}}:=\frac{{\bf{t}}_1 x_1}{L},\quad\ro_{\mathrm{aff}}:=(k_0+\frac{(k_L-k_0)x_1}{L},0,0),
 \end{align}
 where $\mathbf{t}_1\in\R$ is the first component of the vector $\mathbf{t}$ from~\eqref{bc0}, and $k_0,k_L\in\R$ are determined by the skew symmetric matrices $\K_0$ and $\K_L$ in~\eqref{bc0} and~\eqref{bcL} via
\begin{align}\label{def of K0 KL}
\K_0=\begin{pmatrix}
0&-k_0\\
k_0&0
\end{pmatrix},
\qquad
\K_L=\begin{pmatrix}
0&-k_L\\
k_L&0
\end{pmatrix}.
\end{align}
Solving~\eqref{1D_surr} now requires only the solution of an ordinary differential system instead of a partial differential equation in three space dimensions. As will be demonstrated in the following, this leads to a marked reduction of computational effort in approximating the effective elastic energy in~\eqref{intr:eq2a} numerically.

\subsubsection{Numerical analysis of the surrogate model for isotropic material}\label{2.2.1}
Using the explicit representation of the one-dimensional isotropic elastic energy in~\eqref{form of q1D isotropic} we derive an efficient method to compute the effective elastic energy in~\eqref{intr:eq2a} numerically.
In the following we focus on the special case of an isotropic material: Let the quadratic form $Q$, and the Lam\'e parameters $\mu,\ld$ be as in Proposition~\ref{precise q1D}; moreover, let $\Phi=(\Phi_1,\Phi_2)$ be the random field of Assumption~\ref{assumption of model}, cf.~\eqref{random translation}.

We aim to approximate the effective elastic energy $E^{\vare,h}(\omega)$ by the one-dimensional surrogate model~\eqref{infimum1}, where Proposition \ref{precise q1D} is taken into account. To that end set
\begin{align*}
  a(\omega):=\frac{\mu(\omega)(3\ld(\omega)+2\mu(\omega))}{\ld(\omega)+\mu(\omega)}.
\end{align*}
In the following we shall use the shorthand notation $f^{\vare}(\omega,x_1):=f(\tau_{x_1/\vare}\,\omega)$ for a function $f:\Omega\to\R$ and a sample $\omega\in\Omega$. For convenience, we also drop the dependence on $\omega$ in our notation. For instance, we simply write $f^\vare(x_1)$, $\Phi^\vare$, or $\E^\vare(\bar u,\ro)$ instead of $f^\vare(\omega,x_1)$, $\Phi^\vare(\omega,x_1)$, or $\E^\vare(\omega,\bar u,\ro)$. In view of Proposition~\ref{precise q1D}, $\E^{\vare}(\bar u,\ro)$ takes the explicit form
\begin{align*}
  \E^{\vare}(\bar u,\ro)&= \fint_0^L a^\vare
 \bg(|S|(\partial_1\bar{u}+\tfrac{1}{L}(\ro_3\Phi_{1}^\vare-\ro_2\Phi_{2}^\vare))^2+\bg(\int_S x_3^2d\barx\bg) (\partial_1\ro_2)^2+\bg(\int_S x_3^2d\barx\bg)(\partial_1\ro_3)^2\bg)\,dx_1 \\
 &+ \fint_0^L\mu^\vare\bg(\int_S\bg((x_3-\pt_2\varphi_{\rm aff}(\barx))^2+(x_2+\pt_3\varphi_{\rm aff})^2(\barx)\bg)\,d\barx\bg)
 (\partial_1 \ro_1)^{2}\,dx_1
\end{align*}
with $\varphi_{\rm aff}\in H^1(S)$ is a solution of \eqref{non-normalized}. As a consequence, the minimization of the 3D-energy $\E^{\vare,h}$ is reduced to the  minimization of the (still probabilistic) 1D-energy $\E^{\vare}(\bar u,\ro)$, where minimizers $(\bar u, \ro)$ in the affine space \eqref{defW} can be determined by solving the following (weak) ordinary differential system:
\begin{subequations}
\begin{align}\label{firstvar1}
  \fint_0^L a^\vare \left( \partial_1 \bar u +\tfrac{1}{L}(\Phi_{1}^\vare \ro_3-\Phi_{2}^\vare \ro_2) \right) \partial_1\bar{v}~\mathrm{d}x_{1} =&-\frac{\mathbf{t}_1}{L}\fint_0^La^\vare \partial_1 \bar{v}~\mathrm{d}x_1 ,   \\
  \fint_0^L \mu^\vare \partial_1 \mathbf{r}_1 \partial_1 \mathbf{s}_1~\mathrm{d}x_1=& -\fint_0^L\mu^\vare\frac{k_L-k_0}{L}\partial_1 \mathbf{s}_1~\mathrm{d}x_1,
\end{align}
\begin{align}
  &\big(\int_S x_2^2\,d\bar{x}\big) \fint_0^La^\vare \partial_1 \mathbf{r}_{2} \partial_1 \mathbf{s}_{2}~\mathrm{d}x_{1} - |S| \fint_0^La^\vare \tfrac{1}{L}\left(\Phi_{1}^\vare \ro_3 -\Phi_{2}^\vare\ro_2) +\partial_1 \bar{u}\right)\cdot \mathbf{s}_{2}\tfrac{1}{L}\Phi_{2}^\vare~\mathrm{d}x_{1}\\\nonumber
&\qquad= \frac{\mathbf{t}_1}{L}|S|\fint_0^L a^\vare \mathbf{s}_{2} \tfrac{1}{L}\Phi_{2}^\vare~\mathrm{d}x_{1} ,  \\
  &\big(\int_S x_3^2\,d\bar{x}\big) \fint_0^La^\vare \partial_1 \mathbf{r}_{3} \partial_1\mathbf{s}_{3}~\mathrm{d}x_{1} + |S| \fint_0^La^\vare\tfrac1L\left(\Phi_{1}^\vare \ro_3 -\Phi_{2}^\vare\ro_2)  +\partial_1 \bar{u}\right)\cdot \mathbf{s}_{3}\tfrac1L \Phi_{1}^\vare~\mathrm{d}x_{1}\\\nonumber
&\qquad= -\frac{\mathbf{t}_1}{L}|S|\fint_0^L a^\vare \mathbf{s}_{3} \tfrac1L\Phi_{1}^\vare~\mathrm{d}x_{1}
\end{align}
\end{subequations}
for all test functions $(\bar{v}, \mathbf{s}) \in H_0^1(0,L)\times H_{00}^1(0,L;\R^3)$. From~\eqref{firstvar1} we can directly deduce that for deterministic (fixed) constants $\mu$ and $\lambda$ the solution $\ro_1$ is deterministic as well and given by the affine displacement, i.e.,
\begin{align*}
    \ro_1= k_0+\frac{(k_L-k_0)x_1}{L}.
\end{align*}
Thus for $\mu$ and $\lambda$ fixed it remains to solve a system of three ordinary differential equations for $(\bar{u},\ro_2,\ro_3)$ where the boundary condition for $\bar{u}$ enters linearly and therefore the energy scales quadratically in the derivative $\partial_1 \bar{u}_{\mathrm{aff}}= \frac{{\bf{t}}_1}{L}$ of the affine displacement, see Fig.~\ref{fig:Fig01}.

Furthermore, in Theorem~\ref{thm conv rate} we derive quantitative results  for convergence of the 1D surrogate substitute $E^{\vare}(\omega)$ to a deterministic proxy $E^{0}$ as $\vare \rightarrow 0$: For the error $\|E^\vare(\cdot)-E^0\|_{L^2(\Omega)}$ we find sublinear decrease ($\sim \sqrt{\varepsilon}$) for probabilistic material constants $\mu_{\vare}^\omega(x_1),\lambda_{\vare}^\omega(x_1)$ and linear decrease ($\sim\varepsilon$) for deterministic material constants $\mu,\lambda$, cf. Theorem~\ref{thm conv rate} and Fig.~\ref{fig:Fig01}. The quantity $E^0$ is given by
\begin{align*}
E^0\coloneqq \inf\{\E^0(\bar{u},\ro)\,:\,(\bar{u},\ro)\text{ satisfying }\eqref{defW} \}
\end{align*}
with
\begin{align*}
\E^0(\bar{u},\ro)&:=\fint_0^L
                          Q^0\Big(
                          \begin{pmatrix}
                            \pt_1\bar u+\tfrac1L(\ro_3\la\Phi_{1}\ra-\ro_2\la\Phi_2\ra)\\
                            \pt_1\ro
                          \end{pmatrix}
\Big)
\,dx_1,
\end{align*}
where the function $Q^0$ is defined in \eqref{def:Q0} below. Let us anticipate that in the isotropic case we have explicitly for all $\xxi\in \R^4$,
\begin{align*}
  Q^0(\xxi)= &a^{\rm hom} \bg(|S|\xxi_1^2+\bg(\int_S x_3^2\,d\barx\bg)\xxi_3^2+\bg(\int_S x_2^2\,d\barx\bg)\xxi_4^2\bg)
\nonumber\\&+\mu^{\rm hom}\bg(\int_S\bg((x_3-\pt_2\varphi_{\rm aff}(\barx))^2+(x_2+\pt_3\varphi_{\rm aff})^2(\barx)\bg)\,d\barx\bg)\xxi_2^2,
\end{align*}
where
\begin{equation*}
  a^{\rm hom}=\langle a^{-1}\rangle^{-1},\qquad \mu^{\rm hom}=\langle\mu^{-1}\rangle^{-1}.
\end{equation*}
The value of $E^0$ can be determined by solving the deterministic system of equations
\begin{align*}
    \fint_0^L\left(\partial_1 \bar u +\tfrac1L(\langle \Phi_1 \rangle \ro_3 -\langle \Phi_2 \rangle \ro_2)\right) \partial_1 \bar v~\mathrm{d}x_{1} &=-\fint_0^L\frac{t_1}{L}\partial_1 \bar v~\mathrm{d}x_ ,  \\
     \fint_0^L\partial_1 \ro_1 \partial_1 \mathbf{s}_1   &= \fint_0^L\frac{k_L-k_0}{L}\partial_1 \mathbf{s}_1~\mathrm{d}x_1,   \\
    \bg(\fint_S x_2^2\,d\barx\bg) \fint_0^L\partial_1 \ro_{2} \partial_1 \mathbf{s}_{2}~\mathrm{d}x_{1} -  \fint_0^L\left(\tfrac1L(\langle \Phi_1 \rangle \ro_3 -\langle \Phi_2 \rangle \ro_2) +\partial_1 \bar u \right)\cdot \mathbf{s}_{2}\tfrac1L\langle \Phi_{2}\rangle~\mathrm{d}x_{1}
    &= \fint_0^L \mathbf{s}_{2} \tfrac1L\langle \Phi_{2}
       \rangle~\mathrm{d}x_{1} ,  \\
     \bg(\fint_S x_3^2\,d\barx\bg) \fint_0^L\partial_1\ro_{3} \partial_1 \mathbf{s}_{3}~\mathrm{d}x_{1} +  \fint_0^L\left(\tfrac1L(\langle \Phi_1 \rangle \ro_3 -\langle \Phi_2 \rangle \ro_2)  + \partial_1\bar u\right)\cdot \mathbf{s}_{3} \tfrac1L\langle \Phi_{1}\rangle~\mathrm{d}x_{1}
     &= -\fint_0^L  \mathbf{s}_{3} \tfrac1L\langle \Phi_{1} \rangle~\mathrm{d}x_{1}
\end{align*}
for all test functions $(\bar v, \mathbf{s}) \in H_0^1(0,L)\times H_{00}^1(0,L;\R^3)$.

\subsubsection{Numerical approximation of random fields}
In the following let $\kappa\in\{1,2,\mu,\ld\}$. Describing the stochastic geometric imperfections and material density variation numerically we consider $\Phi_{1}^\vare, \Phi_{2}^\vare$ and $\mu^{\vare},\lambda^{\vare}$ to be stationary and correlated random fields with covariance structure given by symmetric and positive semi-definite operators $C_{\vare}^{\kappa}(s,t)\colon (0,L)\times (0,L) \rightarrow \mathbb{R}^+$. Usually $C_{\vare}^{\kappa}$ is determined by
\[
C_\vare^{\kappa}(s,t)=\sigma_{\kappa}^2 C^{\kappa}\left(\frac{|s-t|}{\varepsilon}\right),
\]
for standard deviation $\sigma_{\kappa}$ and some positive function $C^{\kappa}$ with
\[
 C^{\kappa}(|t|) \rightarrow 0 \ \  \text{for $|t| \rightarrow \infty$},
\]
meaning that $\Phi_{1}^\vare(x_1),\Phi_{2}^\vare(x_1)$ and $\mu^{\vare}(x_1),\lambda^{\vare}(x_1)$ decorrelate on large distances. Modelling these random fields numerically on a discretization $(t^{m}_j)_{j=1,...,N}$ of $(0,L)$ we consider the covariance matrices
\begin{align*}
    \mathbf{C}^{\kappa}= \sigma_{\kappa}^2 \begin{pmatrix}
      C^{\kappa}_\vare(t_1^{m},t_1^{m})& \cdots & C^{\kappa}_\vare(t_1^{m},t_N^{m}) \\
                \vdots & \ddots  & \vdots \\
                C^{\kappa}_\vare(t_N^{m},t_1^{m}) & \cdots & C^{\kappa}_\vare(t_N^{m},t_N^{m})
     \end{pmatrix},
\end{align*}
where the positive semi-definite operators $C^\kappa_\vare$ are determined by
 \[
C^{\kappa}_\vare(t_i^{m},t_j^{m})=\exp\left(-\frac{|t_{i}^{m}-t_{j}^{m}|}{\varepsilon}\right).
\]
Moreover, we assume in the following that for any fixed $x_1\in (0,L)$ the random variables $\Phi_{1}^\vare(x_1)$ and $\Phi_{2}^\vare(x_1)$ are standard Gaussian whereas $\mu^{\vare}(x_1),\lambda^{\vare}(x_1)$ are log-normally distributed. Here, the log-normal distribution is chosen to model the situation of varying material parameters such that with respect to the original parameters $\mu$ and $\lambda$ we have
\[
\mu^{\vare}(x_1)\leq \mu, \quad \lambda^{\vare}(x_1) \leq \lambda
\]
and small deviations occur more frequently.

Discrete representations of $\Phi_{1}^\omega(x_1), \Phi_{2}^\omega(x_1)$ and $\mu^{\vare}(x_1),\lambda^{\vare}(x_1)$ are thus given by
\begin{align*}
    \Phi^{1}_N=L_{1}\cdot V^{1},\ \ \Phi^{2}_N=L_{2}\cdot V^{2}, \ \  \mu_{N}=L_\mu\cdot V^\mu, \ \ \lambda_{N}= L_\lambda \cdot V^\lambda
\end{align*}
for the Cholesky decompositions $\mathbf{C}^{\kappa}=L_{\kappa}L_{\kappa}^{T}$ and random samples $ V^{\kappa}$ with $V_j^{1}, V_j^{2}  \sim \mathcal{N}(0,1)$ and $V^\mu_j,V^\lambda_j \sim \mathcal{LN}(0,1)$ for $j=1,...,N$. Here, the Cholesky decomposition is preferred over the Karhunen-Lo\'{e}ve transformation as we are predominantly dealing with small correlation lengths $\varepsilon$ and thus the covariance matrices are sparse matrices.

\begin{figure}[htp!]
\centering
  \qquad
   \subfigure[\textbf{Pure tension}: $E^0=E^{\rm opt}=125.756$. ] {\includegraphics[width=0.43\textwidth]{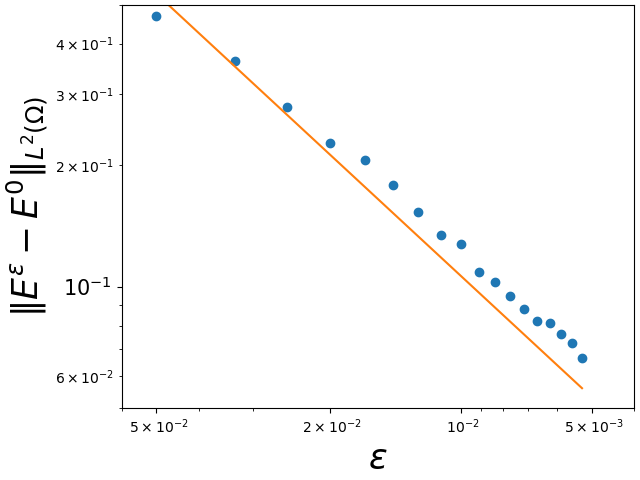}} \hspace{3em}
   \subfigure[\textbf{Pure tension}: $E^0=E^{\rm opt}=503.024$.] {\includegraphics[width=0.43\textwidth]{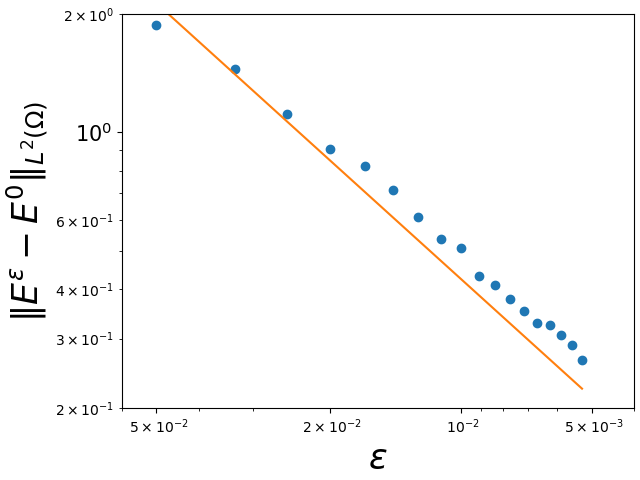}} \\
   \subfigure[\textbf{Tension and twist}: $\mu_N = \mu_{\rm opt}(0.978-L_{\mu}V^{\mu})$, $\lambda^{N}=2.16\mu^{N}$. $E^0=124.158$, $E^{\rm opt}=131.803$.] {\includegraphics[width=0.43\textwidth]{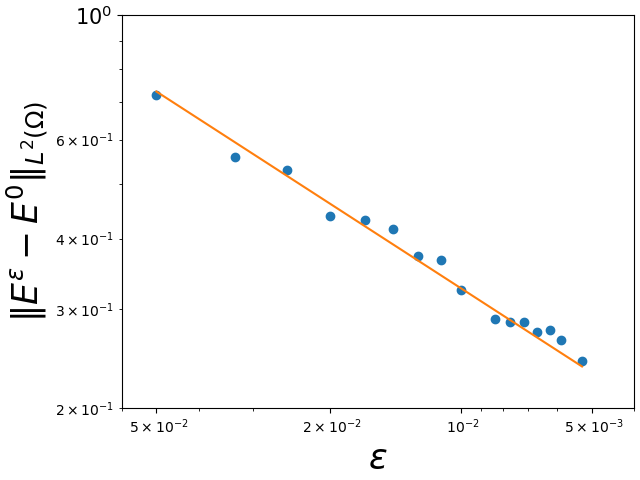}}
 \caption[]{\textbf{Trends of deviations}. Numerical evaluation of the $L^2$-error $ \|E^\vare(\omega)-E^0\|_{L^2(\Omega)}$
 by finite elements and Monte Carlo simulations for circular cross-section $S$ and different correlation lengths $\vare > 0$. We set $\sigma_{1}=\sigma_{2}=0.3$, $\sigma_\mu=0.00717$ and $\sigma_\lambda=0.0155$ where the effective elastic energy with respect to the unperturbed rod $O^h$ is denoted by $E^{\rm opt}$ and the parameters $\mu$ and $\lambda$ are fixed (deterministic) in (\textbf{a}) and (\textbf{b}) with $\mu=\mu_{\rm opt}=30.8$ and $\lambda=\lambda_{\rm opt}=66.6$ and given by random fields in (\textbf{c}). The compression is given by $\textbf{t}_1=L$ in (\textbf{a}) and (\textbf{c}) and $\textbf{t}_1=2L$ in (\textbf{b}). The remaining boundary conditions are determined by $k_0=k_L=0$ in (\textbf{a}) and (\textbf{b}), $k_0=0$, $k_l=0.5$ in (\textbf{c}). In correspondence to Theorem~\ref{thm conv rate} the interpolations (orange line) indicate linear decrease in ($\sim \varepsilon$) (\textbf{a})-(\textbf{b}) and sublinear ($\sim \sqrt{\varepsilon}$) decrease in (\textbf{c}).}
 \label{fig:Fig01}
 \end{figure}

\subsubsection{Numerical Experiments}\label{sec rate dimension reduction}
In order to demonstrate the marked reduction in computational effort using the surrogate model $E^{\vare}(\omega)$, we perform several numerical experiments and show its practicability. To do so we consider a circular cross-section $S$ and $L=1$ for different radii $h$ and compression boundary conditions at the top, i.e., we set $\mathbf{t}_1=0.5L$ and $k_0=k_L=0$. Further, we use deterministic Lam\'{e} parameters $\mu=30.8$ and $\ld=66.6$ and initially fix the correlation length to $\vare=0.05$.

Using standard finite element methods in 3D and 1D, the energies $E^{\vare,h}(\omega)$ and $E^\vare(\omega)$ are approximated by Monte Carlo simulations where the computation is performed on $4$ cores of an \textbf{Intel Core i7-4770S CPU} @ $3.1$ GHz.  The usage of the surrogate model in fact leads to a reduction of the required degrees of freedom in the computation of the discretized problem and thus reduces the computation time markedly, see Fig.~\ref{fig:Fig02}, ranging from $<10$ seconds for the surrogate model to $\sim 10^{4}$ seconds in the three-dimensional case.

Moreover, we find that the one-dimensional surrogate $E^\vare$ estimates the fluctuation $E^{\vare,h}(\omega) -\langle E^{\vare,h}(\omega) \rangle$ around the expected value and therefore, leads to
\begin{align}\label{num_approx}
E^{\vare,h}(\omega) \approx E^\vare(\omega) +\bg(\langle E^{\vare,h}(\omega)\rangle- \langle E^\vare(\omega)\rangle\bg).
\end{align}
The convergence in distribution of the fluctuations $E^{\vare,h}(\omega)- \langle E^{\vare,h}(\omega) \rangle$ is depicted in Fig.~\ref{fig:Fig02}.

Furthermore, we perform the same simulation for different choices of the standard deviations $\sigma_{1,2}$ and the correlation length $\vare$ and evaluate the systematic error
\begin{align}\label{systematic_error}
\left|\langle E^\vare(\omega)\rangle - \langle E^{\vare,h}(\omega)\rangle \right|
\end{align}
in~\eqref{num_approx}. First, we consider fixed standard deviations $\sigma_{1}=\sigma_{2}=0.00375$ and compute~\eqref{systematic_error} for different correlations lengths
\[
\vare\in \{0.03,0.05,0.07\}
\]
and different radii $h$, see~(\textbf{a}) in Fig.~\ref{fig:trend_thesis5}.~Following that, we fix the correlation length $\vare=0.05$ and compute~\eqref{systematic_error} for different standard deviations
\[
(\sigma_1=\sigma_2)\in \{0.0025,0.00375,0.005\}
\]
and different radii $h$, see~(\textbf{b}) in Fig.~\ref{fig:trend_thesis5}.~The results then suggest a quantitative estimate for the systematic error~\eqref{systematic_error} of the form
\begin{align*}
 |\langle E^{\vare,h}(\omega)\rangle-\langle E^\vare(\omega)\rangle| \leq C h^{\alpha}
\end{align*}
for $\alpha \in (0,1)$ and a constant $C$ which is independent of $h$ but monotone (increasing) in $\vare^{-1}\max\{\sigma_{1},\sigma_{2}\}$. By~\eqref{num_approx} we thus obtain that a good approximation of the effective energy $E^{\vare,h}(\omega)$ by the surrogate substitute $E^{\vare}(\omega)$ can be expected in the case where $h \ll \vare$, i.e., when the thickness of the rod is small compared to the correlation length of the geometric perturbation, or in the case where the amplitudes of the geometric perturbations which are determined by the standard deviations $\sigma_{1}$ and $\sigma_{2}$ are small and comparable to $\vare$. A mathematically rigorous proof of the estimate, however, remains open up to this point and gives an opportunity for further research.

In the opposite regime where the systematic error is large the surrogate model $E^{\vare}(\omega)$ can be used to estimate the fluctuation $E^{\vare,h}-\langle E^{\vare,h}\rangle$ and the random variable $E^{\vare,h}(\omega)$ can be approximated by means of \eqref{num_approx} where the systematic error $\langle E^{\vare,h}-E^{\vare}\rangle$ is inferred by computing $E^{\vare,h}$ and $E^{\vare}$ for the same sample of perturbations $\Phi$. Through this we obtain a multi-fidelity approach which combines the accuracy of the three-dimensional (high-fidelity) model with the efficiency of the one-dimensional (low-fidelity) surrogate model. This leads to decent approximations of the effective energy $E^{\vare,h}$ with computation time comparable to that of a sole usage of the one-dimensional surrogate model, see Fig.~\ref{fig:large_systematic}.
\begin{figure}[htp!]
\centering
  \qquad
  \subfigure[Cumulative distribution functions $F_X$.]{\includegraphics[width=0.4\textwidth]{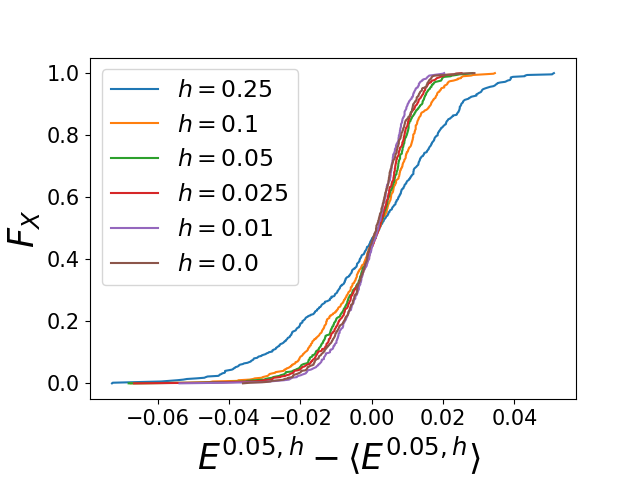}}\\
   \subfigure[$h=0.25$] {\includegraphics[width=0.29\textwidth]{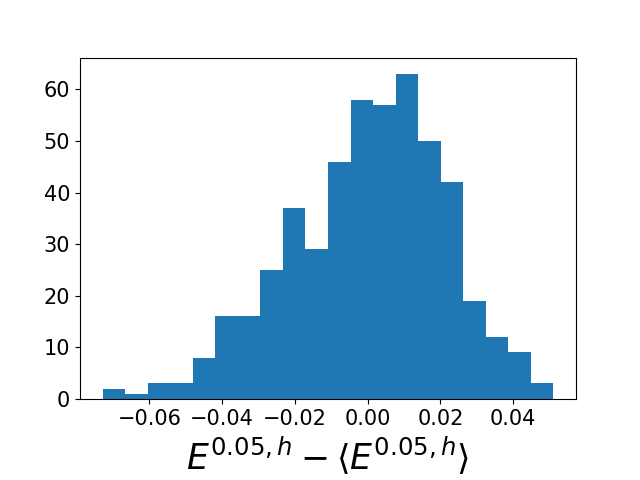}}\hspace{2em}
   \subfigure[$h=0.1$] {\includegraphics[width=0.29\textwidth]{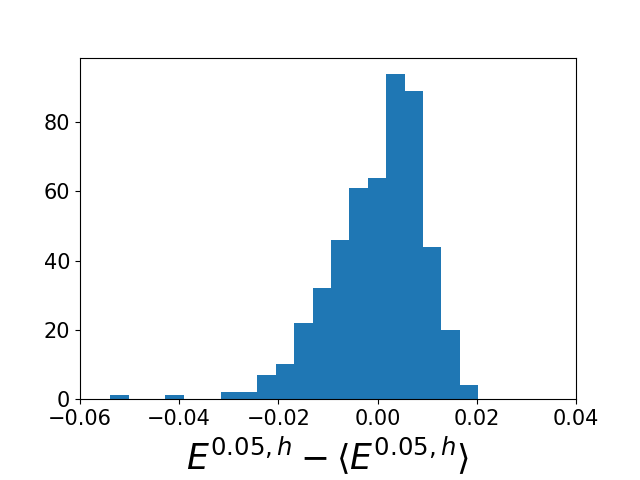}}\hspace{2em}
   \subfigure[$h=0.05$] {\includegraphics[width=0.29\textwidth]{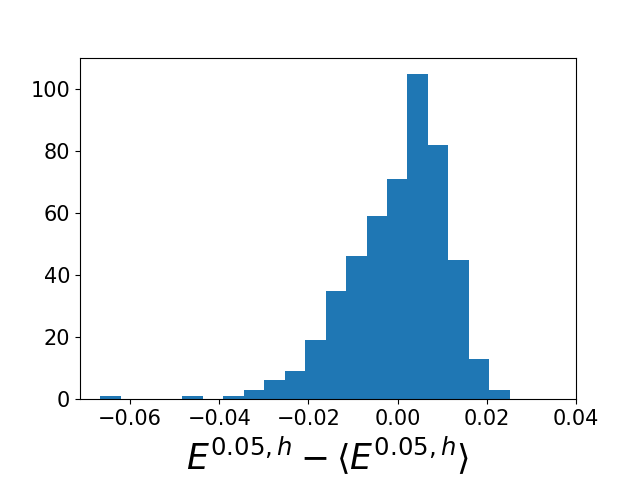}}\\
   \subfigure[$h=0.025$] {\includegraphics[width=0.29\textwidth]{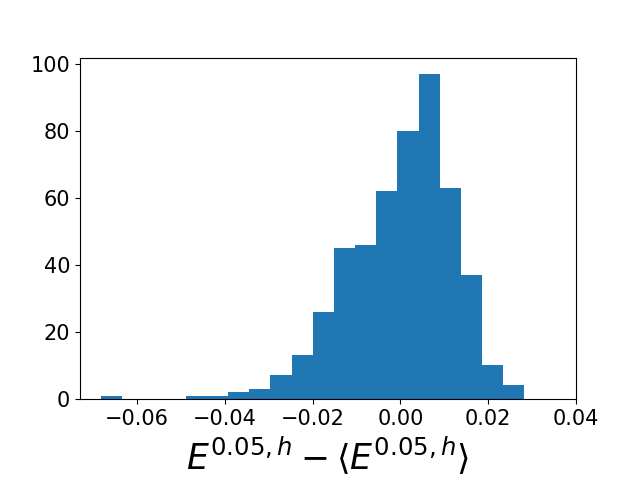}}\hspace{2em}
   \subfigure[$h=0.01$] {\includegraphics[width=0.29\textwidth]{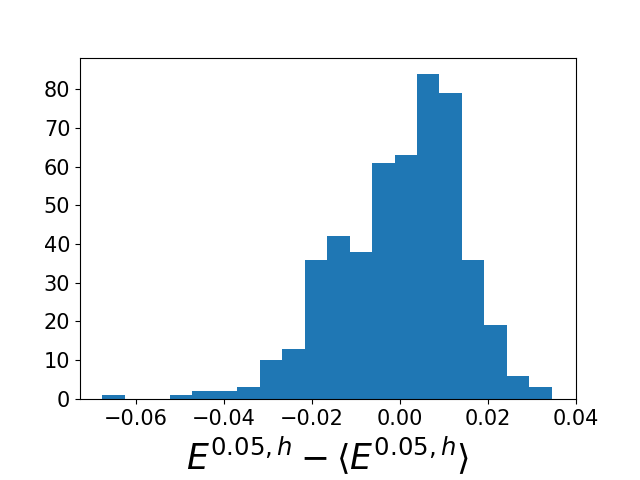}}\hspace{2em}
   \subfigure[$h=0.0$] {\includegraphics[width=0.29\textwidth]{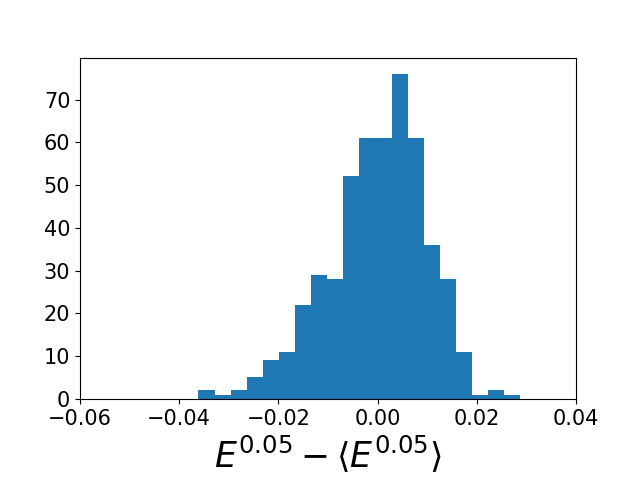}}
 \caption[]{\textbf{Numerical simulations for cylinder geometry.} Cumulative distribution functions $F_{X}$ of the fluctuations $E^{\vare,h}- \langle E^{\vare,h} \rangle$ for different radii $h$ (\textbf{a}). Distribution for 3D (\textbf{b})-(\textbf{f}) and 1D (\textbf{g}) finite element approximation for circular cross-section $S$ with $\sigma_{1}=\sigma_{2}=0.3$ and boundary conditions $t_1=-0.5L$ and $k_0=k_L=0$. The effective elastic energy of the unperturbed rod is $E^{\rm opt}=31.439$. The computation of $ E^{\vare,h}$ and $E^\vare$ is performed for $\varepsilon=0.05$ and $500$ samples of perturbations. The one dimensional surrogate model is indicated by $h=0.0$. The computation time in seconds $s$ is $\sim 10^4 s$ in (\textbf{b})-(\textbf{f}) and $\sim 1s$ in (\textbf{g}).}
 \label{fig:Fig02}
 \end{figure}

\begin{figure}[htp!]
\centering
  \qquad
\subfigure[]{\includegraphics[width=0.4\textwidth]{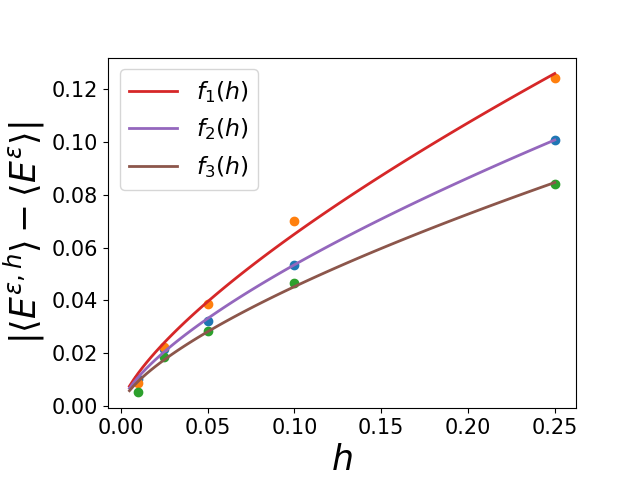}} \hspace{2em}
\subfigure[]{\includegraphics[width=0.4\textwidth]{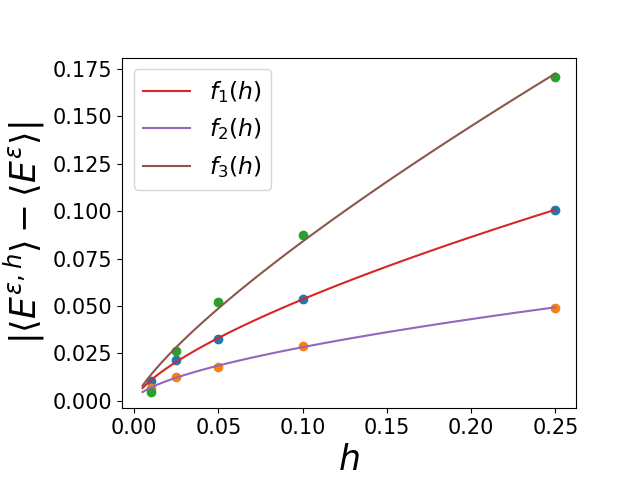}}
 \caption[]{Trends of the systematic errors $\left|\langle E^{\vare,h}\rangle-\langle E^{\vare}\rangle\right|$ for correlation lengths $\vare \in \{0.03,0.05,0.07\}$ and fixed standard deviations $\sigma_{1}=\sigma_{2}=0.3$ in (\textbf{a}) and for standard deviations $(\sigma_{1}=\sigma_2)\in \{0.0025,0.00375,0.005\}$ and fixed correlation length $\vare=0.05$ in (\textbf{b}).~The computation of $E^{\vare,h}$ and $E^{\vare}$ is performed for $250$ samples of perturbations using a Monte Carlo simulation approach.~Further, the expected errors are interpolated by the functions $f_j(h)=a_jh^{t_j}$ resulting in $a_j\in \{0.34,0.26,0.22 \}$ and $t_j\in\{0.72,0.69,0.68\}$ in (\textbf{a}) and in $a_j\in \{0.114,0.26,0.51\} $ and $t_j\in \{0.6,0.69,0.78\}$ in (\textbf{b}), for $j=1,2,3$.}
 \label{fig:trend_thesis5}
 \end{figure}

 \begin{figure}[htp!]
\centering
\subfigure[Cumulative distribution functions $F_X$]{\includegraphics[width=0.59\textwidth]{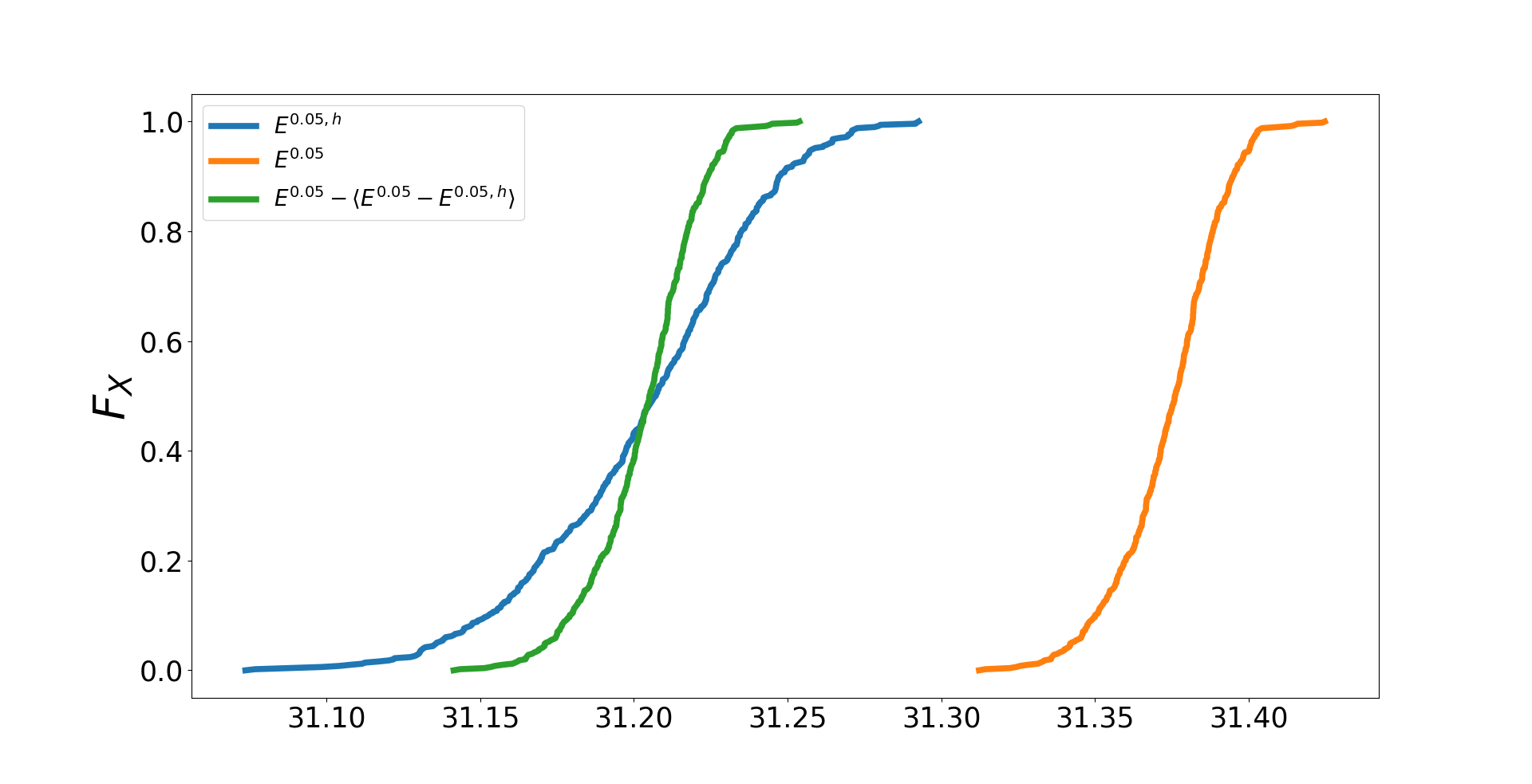}} \\
\subfigure[Computation time $\sim 10^4 s$]{\includegraphics[width=0.29\textwidth]{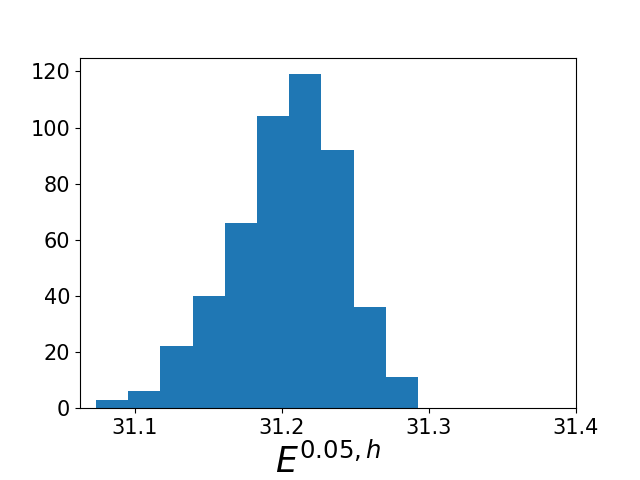}} \hspace{1em}
\subfigure[Computation time $\sim 10^0 s$]{\includegraphics[width=0.29\textwidth]{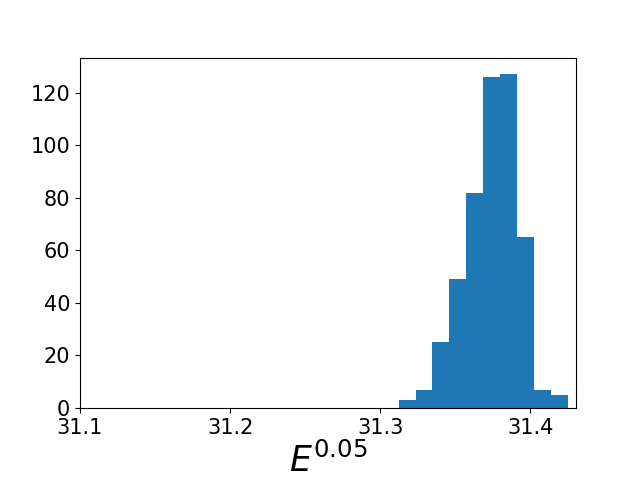}}\hspace{1em}
\subfigure[Computation time $\sim 10^1 s$]{\includegraphics[width=0.29\textwidth]{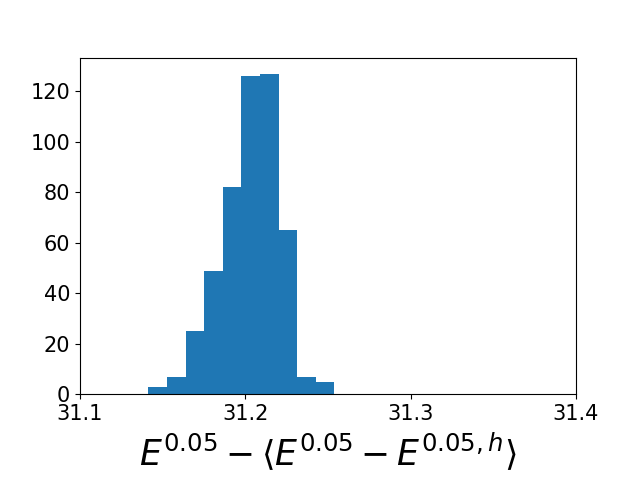}}
 \caption[]{Cumulative distribution functions $F_X$ of energies $E^{\vare,h}$, $E^{\vare}$ and $E^{\vare}+\langle E^{\vare,h}-E^{\vare}\rangle$ for circular cross-section $S$ with $\sigma_{1}=\sigma_{2}=0.4$, $\varepsilon=0.05$ and $h=0.25$ (\textbf{a}). Distribution for 3D (\textbf{b}) and 1D (\textbf{c}) finite element approximations of effective energies $E^{\vare,h}$ and $E^{\vare}$. Subfigure (\textbf{d}) shows the distribution for an approximation of $E^{\vare,h}$  by~\eqref{num_approx}. The computation of $ E^{\vare,h}$ and $E^\vare$ is performed for $500$ samples of perturbations.}
 \label{fig:large_systematic}
 \end{figure}


 \subsection{Conclusion}
Including the one-dimensional surrogate model in the approximation of mechanical properties we can significantly reduced the computational effort that is required for a finite element analysis of mechanical tests in three space dimensions. In particular, the surrogate model provides an estimation of the fluctuation of an effective mechanical property around its mean value which is fairly accurate for a wide range of randomly perturbed rod-shaped structures where perturbations are caused by correlated in-plane shifts of layers. From this we obtain a new multi-fidelity (Monte Carlo) approach, whereas the systematic error has to be inferred from a comparison of the three-dimensional (high-fidelity) model with the one-dimensional (low-fidelity) surrogate model for a single sample of perturbations. This leads to a marked reduction in computational effort compared to the sole usage of a three-dimensional model where on the other hand accuracy of the approximation is preserved compared to a sole usage of the one-dimensional surrogate model. However, the accuracy of the approximation depends strongly on the magnitudes $\sigma_1$ and $\sigma_2$ of the random shifts as well as on the ratio $\frac{h}{\vare}$ between the diameter $h$ of the rod and the correlation length $\vare$. This leads to limitations of the approach as the factor $\frac{h \max\{\sigma_1,\sigma_2\}}{\vare}$ increases and a higher number of evaluations of the three-dimensional model might be required in this case. 
 
\section{Analytical results}\label{section:main results}
In this subsection we state our main analytical results. Recall that we consider the energy functional
\begin{align}\label{def of energy2}
 \E^{\vare,h}(\omega,\vv)\coloneqq \frac{1}{h^{2}}\frac1L\int\limits_{O^{\vare,h}(\omega)}Q(\tau_{\frac{x_1}{\vare}}\omega,\nabla \vv(x))\,dx,
\end{align}
and the minimization problem
\begin{align}\label{intr:eq2}
E^{\vare,h}(\omega):=\inf_{\vv \in H_{\mathcal{B}}^1(O^{\vare,h}(\omega);\R^3)} \E^{\vare,h}(\omega,\vv),
\end{align}
defined by \eqref{def of energy2a} and \eqref{intr:eq2a} respectively. We prove rigorously, via the framework of $\Gamma$-convergence theory, that the limit of $E^{\varepsilon,h}(\omega)$ can be identified as minima of a limiting functional which precise form depends on the relative scaling $\gamma\sim\frac{h}{\vare}$. In the present paper we consider the cases
\begin{align}\label{scheme}
h\ll \vare\text{ (the case $\gamma=0$) and }\vare\ll h\text{ (the case $\gamma=\infty$).}
\end{align}
The case $\gamma\sim\frac{h}{\vare}\in(0,\infty)$ will be treated in a fourthcoming paper.
The effective elastic energies are given by the following functionals:
\begin{subequations}\label{functionals}
\begin{align}
  \E^{{\rm rod},\vare}(\omega,\bar{u},\ro)&:=\fint_0^L
                                    Q^{\rm rod}\Big(\tau_{\frac{x_1}{\vare}}\omega,
                                    \begin{pmatrix}
                                      \pt_1\bar u+\tfrac1L(\ro_3\Phi_1(\tau_{\frac{x_1}{\vare}}\omega)-\ro_2\Phi_2(\tau_{\frac{x_1}{\vare}}\omega))\\
                                      \pt_1\ro
                                    \end{pmatrix}
  \Big)\,dx_1,\label{eq:1Deps}\\
\widehat{\E}^{{\rm hom},h}(\uu)&:=\frac1L\int_O
                                    Q^{\rm hom}\Big(\scaled \uu(\id+\tfrac{h}{L}\la\B\ra)\Big)\,dx,\label{eq:3Dh}\\
\E^\gamma(\bar{u},\ro)&:=\fint_0^L
                          Q^{\gamma}\Big(
                          \begin{pmatrix}
                            \pt_1\bar u+\tfrac1L(\ro_3\la\Phi_1\ra-\ro_2\la\Phi_2\ra)\\
                            \pt_1\ro
                          \end{pmatrix}
\Big)
\,dx_1,\qquad \gamma\in\{0,\infty\}.\label{eq:1Dgamma}
\end{align}
\end{subequations}
Here, $\B$ is the random matrix field defined by \eqref{the form of B}. Recall that $Q^{\rm rod}$ is defined in \eqref{eq:def of Qel}. The quadratic form $Q^{\rm hom}$ describes the homogenized 3d material and is defined for all $\F\in \R^{3\times 3}$ by
\begin{equation}\label{def Qhom}
  Q^{\rm hom}(\F):=\inf_{\chi\in L^2_0(\Omega;\R^3)}\int_\Omega Q(\omega,\F+\chi\otimes \e_1)\,d\mathbb P(\omega).
\end{equation}
It is the standard homogenization formula of stochastic homogenization in the special case of a random laminate that oscillates in $x_1$-direction.
The quadratic forms $Q^{\gamma}:\R^4\to\R$ describe the effective properties of the homogenized 1D-rod, and are defined as follows:
\begin{subequations}
\begin{align}\label{def:Q0}
  Q^0(\xxi):=&\,\inf_{\chi\in L^2_0(\Omega;\R^4)}\int_\Omega Q^{\rm rod}(\omega,\xxi+\chi)\,d\mathbb P(\omega),\\
  \label{def:Qinfty}
  Q^\infty(\xxi):=&\,\inf_{\varphi\in H^1(S;\R^3)}\int_SQ^{\hom}\Big((\xxi_1\e_1+\xxi_{234}\wedge(0,\bar x))\otimes\e_1+(0,\nabla_{\bar x}\varphi)\Big)\,d\bar x.
\end{align}
\end{subequations}
We see that formulas are obtained by consecutively applying the formulas corresponding to dimension reduction and homogenization.

The following Lemma shows that $Q^\gamma, Q^{\rm hom}$, and $Q^{\rm rod}$ satisfy quadratic growth conditions and are sufficiently regular, so that the functionals $\E^\gamma,\widehat{\E}^{\hom,h}$, and $\E^{\vare,\rm rod}$ are indeed well-defined:

\begin{lemma}[Quadratic growth of $Q^{\rm rod}$ and $Q^{\gamma}$]\label{L:quadr}
  There exist constants $0<\beta_1\leq\beta_2<\infty$ only depending on $\alpha_1,\alpha_2$ and $S$ such that the following statement holds.
  \begin{enumerate}[label=(\alph*)]
  \item\label{lem coer a} $Q^{\rm rod}:\Omega\times\R^4\to\R$ defined by \eqref{eq:def of Qel} is measurable and for $\mathbb P$-a.a.~$\omega\in\Omega$, the function $Q^{\rm rod}(\omega,\cdot)$ is quadratic and satisfies
    \begin{equation*}
      \beta_1|\xxi|^2\leq Q^{\rm rod}(\omega,\xxi)\leq \beta_2|\xxi|^2\text{ for all }\xxi\in\R^4.
    \end{equation*}
  \item\label{lem coer b} For $\gamma\in\{0,\infty\}$ the function $Q^{\gamma}$ is quadratic and satisfies
    \begin{equation*}
      \beta_1|\xxi|^2\leq Q^{\gamma}(\xxi)\leq \beta_2|\xxi|^2\text{ for all }\xxi\in\R^4.
    \end{equation*}
  \item\label{lem coer c} If $Q$ is independent of $\omega$, then we have
    \begin{equation*}
      Q^{\rm rod}=      Q^{0}=      Q^{\infty}.
    \end{equation*}
  \end{enumerate}
\end{lemma}
We refer to Section \ref{sec homogenization} for the proof.

\subsection{Qualitative convergence results for $E^{\vare,h}$}
Our first result proves convergence of the effective modulus $E^{\vare,h}(\omega)$. Depending on the scaling scheme, c.f.~\eqref{scheme}, the limit is given as the minimum of one of the energies \eqref{eq:1Deps} to \eqref{eq:1Dgamma} subject to appropriate boundary conditions, cf.~Fig.~\ref{diagram}. To be precise, we define
\begin{align}
  E^{{\rm rod},\vare}(\omega)&:=\inf\{\E^{{\rm rod},\vare}(\omega,\bar u,\ro)\,:\,(\bar{u},\ro)\in (\bar u_{\mathrm{aff}},\ro_{\mathrm{aff}})+H_0^1(0,L)\times H_{00}^1(0,L;\R^3)
                   \},\label{infimum1}\\
  \nonumber\\
  E^{{\rm hom},h}&:=\inf\Big\{\widehat{\E}^{{\rm hom},h}(\uu)\,:\,\uu\in H^1(O;\R^3)\text{ with }\uu(x_1,\cdot)=\uu^h_{\rm aff}(x_1,\cdot)\text{ for }x_1\in\{0,L\}\,\Big\},\label{infimum2}\\
  \nonumber\\
  E^\gamma&:=\inf\{\E^\gamma(\bar u,\ro):\,(\bar{u},\ro)\in (\bar u_{\mathrm{aff}},\ro_{\mathrm{aff}})+H_0^1(0,L)\times H_{00}^1(0,L;\R^3)
                   \},\qquad \gamma\in\{0,\infty\}\label{infimum3},
\end{align}
where $(\bar u_{\rm aff},\ro_{\rm aff})$ is defined in~\eqref{affine function def}, and $\uu^h_{\rm aff}:O\to\R^3$ by
\begin{align}
  \uu^h_{\rm aff}(x):=(1-\tfrac{x_1}{L})(0,(h\A_0+\K_0)\barx)+\tfrac{x_1}{L}\big(\mathbf{t}+(0,(h\A_L+\K_L)(\bar x+L\la\Phi\ra)\big).\label{3D affine function def}
\end{align}

\begin{theorem}[Convergence of the effective modulus]\label{main theorem qual}
  Let Assumption \ref{assumption of model} be satisfied. Assume that $\Phi$ satisfies the smallness condition \eqref{eq:phic} for some $c_S$ that can be chosen only depending on $S$. Then the following statements hold for $\mathbb P$-a.a.~$\omega\in\Omega$:
\begin{enumerate}[label=(\alph*)]
\item We have
\begin{alignat}{2}
E^{\vare,h}(\omega)&\xrightarrow{h\to 0} E^{{\rm rod},\vare}(\omega)&&\xrightarrow{\vare\to 0}E^0,\label{chain 1}\\
E^{\vare,h}(\omega)&\xrightarrow{\vare\to 0} E^{{\rm hom},h}&&\xrightarrow{h\to 0} E^{\infty}.\label{chain 2}
\end{alignat}
\item\label{main theorem equal 4} Assume additionally that $Q$ is independent of $\omega$. Then for $\gamma\in\{0,\infty\}$ the affine configuration  $(\bar u_{\mathrm{aff}},\ro_{\mathrm{aff}})$ is the unique minimizer of $\E^\gamma$, and
  \begin{equation*}
    E^\gamma=LQ^{0}(\pt_1(\bar u_{\mathrm{aff}},\ro_{\mathrm{aff}})).
  \end{equation*}
\end{enumerate}
\end{theorem}

We present the proof of the theorem at the end of Section~\ref{sec:gamma convergence}.

\begin{remark}[The case $\gamma\in(0,\infty)$]\label{remark forthcoming paper 2}\normalfont
  As already mentioned in the introduction, in a forthcoming paper, which is work in progress, we analyze the simultaneous limits $(h,\vare)\to 0$. More pecisely, we show that if $\vare=\vare(h)$ is a parameter satisfying
  \begin{align*}
    \lim_{h\to 0}\vare(h)=0,\quad\lim_{h\to 0}\frac{h}{\vare(h)}=\gamma\in[0,\infty],
  \end{align*}
  then $E^{\vare(h),h}(\omega)\xrightarrow{h\to 0} E^\gamma$ for $\mathbb P$-a.a.~$\omega\in\Omega$, where $E^\gamma$ is the minimum of the functional $\E^\gamma$ formulating in terms of some density function $Q^\gamma$ for $\gamma\in(0,\infty)$. Heuristically, a density function $Q^\gamma$ can be thought as an interpolation of $Q^0$ and $Q^\infty$. To see this, we will for instance prove that $Q^\gamma$ is continuous in the parameter $\gamma\in[0,\infty]$ and conclude as a consequence that $E^\gamma$ is a continuous function of $\gamma$. Together with Theorem~\ref{main theorem qual}, this establishes the convergence diagram in Fig.~\ref{diagram}.
The analysis for the simultaneous limit builds on the two-scale methods of \cite{NeukammPhd,NeukammARMA2012}, where simultaneous homogenization and dimension reduction for rods in the periodic case is studied.
\end{remark}

\begin{figure}[htp!]
  \centering
  \includegraphics[width=65mm]{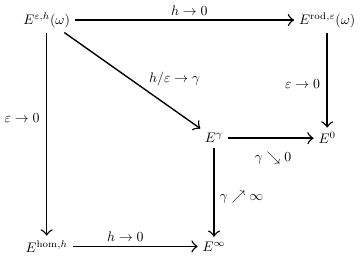}
\caption{A schematic description for the convergence statements from Theorem \ref{main theorem qual}. The convergence in the simultaneous limit $h/\vare\to\gamma$ will be established in a forthcoming paper.}\label{diagram}
\end{figure}

\begin{remark}\normalfont
  From Theorem \ref{main theorem qual} \ref{main theorem equal 4} we see that for spatially homogeneous $Q$, the geometric perturbation is not influencing the limit $E^0=E^\infty$ of the effective modulus. However, the perturbations lead to fluctuations along the limit of $E^{{\rm rod},\vare}\to E^0$.
\end{remark}

\begin{remark}\normalfont
  We note that we prove the dimension reduction result $E^{\vare,h}\to E^{{\rm rod},\vare}$ in fact in a more general (deterministic) framework which is independent of Assumption \ref{assumption of prob}, see Proposition \ref{thm dim reduction} below for details.
\end{remark}

The proof of Theorem~\ref{main theorem qual} follows from several $\Gamma$-convergence results for the energy functionals in \eqref{functionals} and $\E^{\vare,h}$, see Section~\ref{sec:gamma convergence} below. For the proof we appeal to stochastic two-scale convergence methods (see for instance \cite{Allaire1992,NeukammPhd,NeukammARMA2012}). In particular, to relate 3D-displacements and rod-configurations, we make use of a decomposition due to Griso \cite{Griso2004}, see~Definition~\ref{def of g conv} below. Furthermore, we shall see that the energy functional $\E^{\vare,h}$ can be transformed to a functional with a fixed domain. This leads to the appearance of a randomly oscillating prestrain that captures the effect of the geometry perturbations. In our convergence analysis we treat the prestrain following the method in \cite{BNS2020}, where the derivation of rods that feature a micro-heterogeneous prestrain is studied.

\subsection{Quantitative convergence results for $E^{{\rm rod},\vare}$ }
To quantify the speed of convergence of $E^{{\rm rod},\vare}$ as $\vare\to 0$, we need to replace the qualitative ergodicity assumption by a stronger quantitative one. There are different ways to quantifify ergodicity in stochastic homogenization. We use functional inequalities to quantify ergodicity, e.g., see \cite{GNOInvent,GloriaNeukammOttoMIlan,GNO2021}.

For this purpose we assume from now on, in addition to Assumption~\ref{assumption of prob}, that the probability space $\Omega$ consists of $\R^N$-valued (with $N\in\N$ fixed), locally integrable random fields on $\R$, i.e., $\Omega\subset L^1_{\loc}(\R;\R^N)$, and that $\tau$ denotes the \textit{shift}, i.e.,
\begin{equation*}
  \tau_{s}\omega(\cdot):=\omega(\cdot+s)\qquad\text{for all }s\in\R.
\end{equation*}

\begin{assumption}[Spectral gap assumption]\label{sg assumption}
  We assume that there exists a constant $\rho> 0$ such that for any random variable $F : \Omega\to\R$ we have
  \begin{align}
    \mathbb{E}\big[ |F-\mathbb{E}[F]|^{2}\big]\leq \frac{1}{\rho^2}\,\mathbb{E}\bg[\int_{\R}\bg(\int_{s-1}^{s+1}\bg|\frac{\partial F}{\partial \omega}\bg|\,\bg)^2\,ds\bg],
  \end{align}
  where the norm of the functional derivative is defined as
  \begin{align}
    \int_{s-1}^{s+1}\bg|\frac{\partial F}{\partial \omega}\bg|:=\sup\bg\{\limsup_{t\to 0}\frac{F(\omega+t\delta\omega)-F(\omega)}{t}\bg\}
  \end{align}
  and the supremum is taken over all measurable perturbations $\delta\omega:\R\to \R$ with
  $$\mathrm{supp}\,\delta\omega\subset [s-1,s+1]\quad\text{and}\quad\|\delta\omega\|_{L^\infty(s-1,s+1)}\leq 1.$$
\end{assumption}

\begin{remark}\normalfont
An admissible example for which the spectral gap assumption is satisfied can be constructed using the Malliavin calculus as follows: we assume that the random field $\omega:\R\to\R$ is a centered, stationary Gaussian random field such that the covariance function $s\mapsto\mathcal{C}(s):=\mathrm{Cov}[\omega(s),\omega(0)]$ is a bounded function on $\R$ with compact support. Then the probability space $(\Omega,\mathbb{P},\mathcal{F})$ satisfies Assumption \ref{sg assumption}. For a proof, we refer to \cite{DuerinckxOtto2020,DG20b}. We also refer to \cite{pSGDuerGloria,DG20b} for further more general examples for which the spectral gap assumption is satisfied.
\end{remark}

We shall further assume that the random coefficients in our model are \textit{1-local Lipschitz random variables}: We say that a random variable $F:\Omega\to\R$ is 1-local Lipschitz, if there exists a constant $C_F$ such that
\begin{equation}\label{locallip}
  |F(\omega)-F(\omega')|\leq C_F\|\omega-\omega'\|_{L^\infty(-1,1)}\qquad\text{for $\mathbb P$-a.a. }\omega,\omega'\in\Omega.
\end{equation}
In that case we call $C_F$ a Lipschitz constant of $F$.

\begin{remark}
\normalfont
  For a 1-local Lipschitz random variable $F$ one can easily check that
  \begin{equation*}
    \int_{s-1}^{s+1}\bg|\frac{\partial F}{\partial \omega}\bg|\leq C_F.
  \end{equation*}
  We also note that it is easy to construct 1-local Lipschitz random variables: If $\Lambda:\R\to\R$ is a Lipschitz function with Lipschitz constant $C_{\Lambda}$, and $\mu$ a measure on $(-1,1)$, then
  \begin{equation*}
    F(\omega):=\Lambda\Big(\int_{(-1,1)}\omega(s)d\mu(s)\Big)
  \end{equation*}
  is 1-local Lipschitz where the Lipshitz constant is given by $C_{\Lambda}\int_{(-1,1)}\,d\mu$.
\end{remark}

By making use of the two-scale expansion of the minimizer $(\bar u^\vare,\ro^\vare)$ we are able to prove the following quantitative convergence result:
\begin{theorem}[Convergence rate of effective energy]\label{thm conv rate}
  Let Assumption \ref{assumption of model} be satisfied and assume that $\Phi$ satisfies the smallness condition \eqref{eq:phic} with $c_S\leq\frac12$. Suppose also that Assumption~\ref{sg assumption} holds. Let $\Ab:\Omega\to\R^{4\times 4}_{\rm sym}$ be defined by the identity
    \begin{equation*}
      \Ab(\omega)\xxi\cdot\xxi:=Q^{\rm rod}(\omega,\xxi)\text{ for all }\xxi\in\R^4.
    \end{equation*}
    Assume that $\Ab$ and $\Phi$ are 1-local Lipschitz random variables with  Lipschitz constant $C_{\Ab,\Phi}$. Let $(\bar u^\vare,\ro^\vare)$ and $(\bar u^0,\ro^0)$ denote the minimizers of $\E^{{\rm rod},\vare}$ and $\E^{0}$, respectively. Then $(\bar u^0,\ro^0)=(\bar u_{\rm aff}, \ro_{\rm aff})$, and for all $L\geq1$ and $0<\vare\leq L$ we have $\mathbb P$-a.s.,
    \begin{align}\label{conv rate function}
      \fint_0^L|(\bar u^\vare,\ro^\vare)-(\bar u_{\rm aff},\ro_{\rm aff})|^2\,dx_1\leq(|{\bf{t}}_1|^2+|k_0|^2+|k_L|^2)\mathscr C^2C_{\Ab,\Phi}^2\frac{\vare}{L}\\\label{conv rate energy}
      |E^{\vare,\rm rod}-E^{0}|\leq\frac{|{\bf{t}}_1|^2+|k_0|^2+|k_L|^2}{L^2}(\mathscr C^2 C_{\Ab,\Phi}^2(\tfrac{\vare}{L})^\frac12+\mathscr C C_{\Ab,\Phi})\big(\frac{\vare}{L})^{1/2}
    \end{align}
    Above, $\mathscr C$ denotes a random variable satisfying
    \begin{equation}
      \mathbb E\Big[\exp\big(\frac{\mathscr C}{C}\big)\Big]\leq 2,\label{upper bound random constant}
    \end{equation}
    where $C$ only depends on $\rho$, $\alpha_1$ and $\alpha_2$.
    Furthermore, in the constant coefficient case, i.e., when $\Ab$ is independent of $\omega$, we have the improved estimate
    \begin{align}
      |E^{\vare,\rm rod}-E^{0}|\leq\frac{|{\bf{t}}_1|^2+|k_0|^2+|k_L|^2}{L^2}\mathscr C^2 C_{\Ab,\Phi}^2\frac{\vare}{L}.\label{conv rate energy2}
    \end{align}
  \end{theorem}

\begin{remark}
\normalfont
We point out that in the case of small thickness $h\ll \vare$, a convergence rate of \eqref{eq:uq_est_intr} for a given $a>0$ follows already from Theorem~\ref{thm conv rate}, assuming that we have a good understanding in the convergence rate of the dimension reduction procedures. Indeed, using triangular inequality, Markov's inequality, \eqref{upper bound random constant} and \eqref{conv rate energy2} we obtain
\begin{align}
&\,\mathbb{P}\left(|E(O_{h})-E(O^{\vare,h}(\omega))| \geq a \right)\nonumber\\
\leq&\,
\mathbb{P}\left(|E(O_{h})-E^0| \geq \frac{a}{3} \right)
+\mathbb{P}\left(|E^{\vare,{\rm rod}}-E(O^{\vare,h}(\omega))| \geq \frac{a}{3} \right)
+\mathbb{P}\left(|E^{\vare,{\rm rod}}-E^0| \geq \frac{a}{3} \right) \nonumber\\
\leq&\,
\mathbb{P}\left(|E(O_{h})-E^0| \geq \frac{a}{3} \right)
+\mathbb{P}\left(|E^{\vare,{\rm rod}}-E(O^{\vare,h}(\omega))| \geq \frac{a}{3} \right)
+\frac{9\la|E^{\vare,{\rm rod}}-E^0|\ra}{a^2}\nonumber\\
\leq&\,\mathbb{P}\left(|E(O_{h})-E^0| \geq \frac{a}{3} \right)
+\mathbb{P}\left(|E^{\vare,{\rm rod}}-E(O^{\vare,h}(\omega))| \geq \frac{a}{3} \right)
+\frac{C\vare}{L^{3}a^2}.\label{3.20}
\end{align}
In the case $h\ll\vare$ we expect that the first two terms in \eqref{3.20} are also small (in comparison to the last term of order $\vare$), which in turn implies a convergence rate (in $\vare$) for \eqref{eq:uq_est_intr}. The smallness of the first two terms in \eqref{3.20} corresponds to the convergence rate of dimension reduction. Several related numerical simulations have been implemented in Section \ref{sec rate dimension reduction}, which particularly suggest that a convergence rate should be of polynomial growth $h^\alpha$ with $\alpha\in(0,1)$. A rigorously analytical proof nevertheless remains open. We plan therefore to tackle this problem in a forthcoming paper.
\end{remark}

 \subsection{\texorpdfstring{$\Gamma$}{Gamma}-convergence of energies and Proof of Theorem~\ref{main theorem qual}}\label{sec:gamma convergence}
 As already indicated, Theorem \ref{main theorem qual} rather directly follows from a set of $\Gamma$-convergence results for $\E^{\vare,h}(\omega)$, which we shall discuss in this section and which are of independent interest.
 Specifically, we study the following limits:
\begin{itemize}
\item (dimension reduction): $\E^{\vare,h}(\omega)$ to $\E^{\varepsilon,{\rm rod}}(\omega)$ and $\widehat{\E}^{{\rm hom},h}$ to $\E^\infty$,
\item (homogenization):  $\E^{\varepsilon,{\rm rod}}(\omega)$ to $\E^0$ and $\E^{\vare,h}(\omega)$ to $\widehat{\E}^{{\rm hom},h}$.
\end{itemize}
The starting point of our analysis is the following transformed energy, which in contrast to the original energy functional has an unperturbed and upscaled domain:
\begin{equation*}
  H^1(O;\R^3)\ni \uu\mapsto \widehat{\E}^{\vare,h}(\omega,\uu):=\E^{\vare,h}(\omega,\uu\circ\Psi_{\varepsilon,h}^{-1}).
\end{equation*}
By direct calculation one easily sees that $\widehat{\E}^{\vare,h}$ is a quadratic integral functional on the unperturbed domain $O$ and can be written in the form
\begin{align}\label{def of hat energy}
\widehat{\E}^{\vare,h}(\omega,\uu)=\frac1L\int_{O}Q\bg(\tau_{\tfrac{x_1}{\vare}}\omega,\scaled \uu\big(\id+\tfrac{h}{L}\B(\tau_{\frac{x_1}{\vare}}\omega)\big)\bg)\,dx,
\end{align}
where $\B:\Omega \to\R^{3\times 3}$ is defined as
\begin{equation}\label{the form of B}
  \B(\omega)=
  \begin{pmatrix}
    0 & 0 &0\\
    -\Phi_1(\omega) & 0 &0\\
    -\Phi_2(\omega) & 0 &0
  \end{pmatrix}.
\end{equation}
Moreover, we also see that $\vv=\uu\circ\Psi_{\varepsilon,h}^{-1}$ satisfies the boundary condition \eqref{bc} if and only if
\begin{equation}\label{bc u}
  \uu(0,\cdot)=\uu^{\vare,h}_{\rm aff}(0,\cdot)\text{ and }  \uu(L,\cdot)=\uu^{\vare,h}_{\rm aff}(L,\cdot)\text{ on }S,
\end{equation}
where
\begin{equation}\label{def aff vare h}
  \uu^{\vare,h}_{\rm aff}(x):=(1-\tfrac{x_1}{L})(0,(h\A_0+\K_0)\barx)+
                              \tfrac{x_1}{L}\big(\mathbf{t}+(0,(h\A_L+\K_L)(\bar x+\fint_0^L\Phi(\tau_{\frac{t}{\vare}}\omega)\,dt)\big).
\end{equation}
In order to formulate the $\Gamma$-convergence result for $\mE^{\vare,h}$, we need to fix a suitable notion of convergence:
\begin{definition}[Decomposition]\label{def of g conv}
For $\uu\in H^1(O;\R^3)$, we define the function $(\buu,\ro)$ by
\begin{subequations}
\begin{align}
\buu(x_1)&=\fint_{S}\uu(x_1,\barx)\,d\barx,\label{decomp 1}\\
\ro_1(x_1)&=\frac{1}{I_2+I_3}\int_{S}(0,\barx)\wedge \uu(x_1,\barx)\cdot \e_1\,d\barx,\label{decomp 2}\\
\ro_2(x_1)&=\frac{1}{I_3}\int_{S}(0,\barx)\wedge \uu(x_1,\barx)\cdot \e_2\,d\barx,\label{decomp 3}\\
\ro_3(x_1)&=\frac{1}{I_2}\int_{S}(0,\barx)\wedge \uu(x_1,\barx)\cdot \e_3\,d\barx,\label{decomp 4}
\end{align}
\end{subequations}
where $I_j=\int_{S}x_j^2\,d\barx$ for $j=1,2$. Moreover, the operator $\Pi:H^1(O;\R^3)\to H^1(0,L)\times H^1(0,L;\R^3)$ is defined by
\begin{align}\label{def of pai}
\Pi(\uu):=(\buu_1,\ro).
\end{align}
We say that a sequence $(\uu^h)_h\subset H^1(O;\R^3)$ converges to a rod configuration  $(\bar u,\ro)\in H^1(0,L)\times H^1(0,L;\R^3)$ if
\begin{align*}
  \Pi(\uu^h)\to (\bar u,\ro)\quad\text{strongly in $L^2(0,L)\times L^2(0,L;\R^3)$ as $h\to 0$}.
\end{align*}
In that case we simply write $\uu^h\stackrel{\Pi}{\to} (\bar u,\ro)$.
\end{definition}
The next proposition establishes sequential compactness for sequences with finite energy. The argument exploits the fact the matrix-field $\B$ appearing in \eqref{def of hat energy} is not arbitrary, but takes only values in
\begin{equation}\label{def of R0}
  \R^{3\times 3}_{0}:=\{\F\in \R^{3\times 3}\,:\,\F_{1j}=0\text{ for }j=1,2,3\,\}.
\end{equation}

\begin{proposition}[Compactness]\label{prop compact}
  There exists a constant $c_S>0$ only depending on $S$ such that the following holds:
  \begin{enumerate}[label=(\alph*)]
  \item \label{prop compact:a}(Compactness for $h\to 0$). Consider sequences $(\widehat\B^h)_h\subset L^\infty(O;\R^{3\times 3}_0)$ and  $(\uu^{h})_h\subset H^1(O;\R^3)$. Assume that
    \begin{align}\label{smallness of B}
      &\limsup\limits_{h\to 0}\|\widehat\B^h\|_{L^\infty(O)}\leq c_S,\\
      &\limsup\limits_{h\to 0}\,\Big(\int_O|\sym\big(\scaled \uu^h(\id+\tfrac{h}{L}\widehat\B^h)\big)|^2\,dx+|\ro^h(0)|^2)<\infty,\label{finite assump}
    \end{align}
    where $(\bar\uu^h,\ro^h)$ denotes the decomposition of $\uu^h$ according to Definition~\ref{def of g conv}.
    Then
    \begin{align}
      \limsup_{h\to0}\,(\|\partial_1\bar\uu^h_1\|_{L^2(0,L)}+\|\ro^h\|_{H^1(0,L)}+\|\sym\scaled\uu^h\|_{L^2(O)})<\infty,\label{finiteness of uh}
    \end{align}
    and there exists $(\bar u,\ro)\in H^1(0,L;\R)\times H^1(0,L;\R^3)$ such that
  \begin{equation}\label{conv of projection minus mean}
    \uu^h-\bg(\fint_0^L\bar\uu^h_1\,dx_1\bg)\e_1\stackrel{\Pi}{\to}(\bar u,\ro)\qquad\text{for a subsequence.}
  \end{equation}
  Moreover, there exists some $\mathbf{z}=(\mathbf{z}_1,\mathbf{z}_2)\in H^1(0,L;\R^2)$ such that
  \begin{align}\label{conv of hu23}
    h(\bar \uu^h_{2},\bar \uu^h_3)-\fint_0^Lh(\bar \uu^h_{2},\bar \uu^h_3)(x_1)\,dx_1\rightarrow \mathbf{z}\quad\text{strongly in $L^2(0,L;\R^2)$}
  \end{align}
as $h\to 0$ and the limit $\ro$ from \eqref{conv of projection minus mean} satisfies
\begin{align}\label{trace of r23}
\fint_0^L(\ro_2,\ro_3)(x_1)\,dx_1=(\mathbf{z}_{2}(0)-\mathbf{z}_{2}(L),\mathbf{z}_{1}(L)-\mathbf{z}_{1}(0))^T.
\end{align}
\item \label{prop compact:b}(Compactness for $\vare\to 0$). Let $0<h\leq 1$ be fixed. Consider sequences $(\widehat\B^\vare)_\vare\subseteq L^\infty(O;\R^{3\times 3}_0)$ and $(\uu^{\vare})_\vare\subset H^1(O;\R^3)$. Assume that
    \begin{align}\label{smallness of B2}
      &\limsup\limits_{\vare\to 0}\|\widehat\B^\vare\|_{L^\infty(O)}\leq c_S,\\
      &\limsup\limits_{\vare\to 0}\,\Big(\int_O|\sym\big(\scaled \uu^\vare(\id+\tfrac{h}{L}\widehat\B^\vare)\big)|^2\,dx+\|\uu^\vare(0,\cdot)\|_{H^\frac12(S)}\Big)<\infty.\label{finite assump hom2}
    \end{align}
Then
\begin{align}
\limsup_{\vare\to 0}\|\uu^\vare\|_{H^1(O)}<\infty,\label{finite assump hom 2}
\end{align}
and there exist a subsequence of $(\uu^\vare)_\vare$ and $\uu\in H^1(O;\R^3)$ such that
\begin{align}
\uu^\vare\to \uu\quad\text{strongly in $L^2(O;\R^3)$ as $\vare\to 0$}.\label{3D hom limit}
\end{align}
\end{enumerate}
\end{proposition}

\begin{remark}
\normalfont
We briefly explain the necessity of the smallness of $c_S$, which is twofold: On the one hand, the smallness of $c_S$ guarantees that the term $\|\sym(\tfrac{h}{L}\scaled \uu^h \widehat\B)\|_{L^2(O)}$ can be absorbed by $\|\sym\scaled\uu^h\|_{L^2(O)}$; on the other hand, as will be clear in the proof of the $\Gamma$-convergence results, not only $\partial_1\ro$ (see Definition \ref{def of g conv}) but also $\ro$ itself appears in the energy functional. We then apply the Poincar\'e's inequality to control $\ro$ with help of $\partial_1\ro$ (and boundary conditions for $\ro(0)$), where the smallness of $c_S$ is invoked. Moreover, the application of the Poincar\'e's inequality also explains the appearance of the factor $L^{-1}$ in front of $\widehat\B$.
\end{remark}

Next, since the boundary condition \eqref{bc u} is involved with $\vare$, $h$ and $\omega$, it is not \textit{a priori} clear whether the equi-bounded energy conditions \eqref{finite assump} and \eqref{finite assump hom2} given in Proposition \ref{prop compact} are fulfilled for functions satisfying \eqref{bc u}. We show that this is indeed the case under the Assumption \ref{assumption of model} and the smallness condition \eqref{eq:phic}.

\begin{lemma}[Existence of equi-bounded energies]\label{prop finite}
Assume Assumption \ref{assumption of model} and the smallness condition \eqref{eq:phic} with $c_S$ as in Proposition~\ref{prop compact}. Then there exists a constant $C$ (only depending on $\alpha_1,\alpha_2$ and $O$) such that the following holds for $\mathbb P$-a.a.~$\omega\in\Omega$:
\begin{enumerate}[label=(\alph*)]
\item \label{prop finite a}For all $\vare>0$ there exists a sequence $(\uu^h)_h\subset H^1(O;\R^3)$ such that the boundary condition \eqref{bc u} is satisfied and
\begin{equation}\label{uniform norm indep of eps}
\limsup\limits_{h\to 0}\,(\widehat\E^{\vare,h}(\omega,\uu^h)+|\ro^h(0)|)\leq C.
\end{equation}
\item \label{prop finite b}For all $0<h\leq1$ there exists a sequence $(\uu^\vare)_\vare\subset H^1(O;\R^3)$ such that the boundary condition \eqref{bc u} is satisfied and
\begin{equation}\label{uniform norm indep of h}
\limsup\limits_{\vare\to 0}\,(\widehat\E^{\vare,h}(\omega,\uu^\vare)
+\|\uu^\vare(0,\cdot)\|_{H^1(S)})\leq C,
\end{equation}
\end{enumerate}
\end{lemma}

We now state the $\Gamma$-convergence results for the transformed energy functional $\widehat{\E}_{\vare, h}$. We start with the $\Gamma$-convergence result for dimension reduction ($h\to 0$) with $\vare>0$ being fixed. Similar to the results given in Proposition \ref{prop compact}, the $\Gamma$-convergence result for dimension reduction does not depend on a specific $\omega\in\Omega$ and will be formulated in a deterministic framework, where we replace
the random matrix field $\B(\tau_{\frac{x_1}{\varepsilon}}\omega)$
and the random quadratic form $Q(\tau_{\frac{x_1}{\vare}}\omega,\cdot)$ by a deterministic matrix $\widehat\B(x_1)$ and a deterministic quadratic form $\widehat Q(x_1,\cdot)$, respectively.

\begin{proposition}[Dimension reduction]\label{thm dim reduction}
  Let $\widehat Q:(0,L)\times\R^{3\times 3}\to\R$ be measurable and assume that $\widehat Q(x_1,\cdot)\in\mathcal Q(\alpha_1,\alpha_2)$ for a.e.~$x_1\in(0,L)$. Let $\widehat\B\in L^\infty(0,L;\R^{3\times 3}_0)$ satisfy $\|\widehat\B\|_{L^\infty(0,L)}\leq c_S$ where $c_S$ is chosen as in Proposition~\ref{prop compact}. Let $\widehat \E^h:H^1(O;\R^3)\to\R$ be defined by
  \begin{equation*}
    \widehat \E^h(\uu):=\frac1L\int_{O}\widehat Q(x_1,\scaled \uu(\id+\tfrac hL\widehat\B))\,dx,
  \end{equation*}
  and $\widehat\E^0((\bar u,\ro)):H^1(0,L)\times H^1(0,L;\R^3)$ be defined by
  \begin{align}
    \widehat{\E}^0(\bar u,\ro):=\fint_0^L \widehat Q^{\rm rod}\Big(
    \begin{pmatrix}
      \pt_1\bar u+\tfrac1L(-\ro_3\widehat{\B}_{21}+\ro_2\widehat{\B}_{31})\\
      \pt_1\ro
    \end{pmatrix}\Big)
    \,dx_1,\label{gamma limit dr}
  \end{align}
  where
  \begin{equation}\label{eq:def of Qel hat}
    \widehat Q^{\rm rod}(x_1,\xxi):=\inf_{\varphi\in H^1(S;\R^3)}\int_S\widehat Q\Big(x_1,
    (\xxi_1\e_1+\xxi_{234}\wedge(0,\bar x))\otimes\e_1+(0,\nabla_{\bar x}\varphi)\Big)\,d\bar x.
  \end{equation}
  Then the following $\Gamma$-convergence holds:
  \begin{enumerate}[label=(\alph*)]
  \item \label{thm dim reduction a}(Lower bound). Consider a sequence $(\uu^h)_h\subset H^1(O;\R^3)$ with
    \begin{equation}\label{eq:thm dim red a:ass}
      \limsup_{h\to 0}\,(\widehat{\E}^h(\uu^h)+|\ro^h(0)|)<\infty,
    \end{equation}
    where $\ro^h$ is associated with $\uu^h$ according to Definition~\ref{def of g conv}. Assume that $\uu^h\stackrel{\Pi}{\to} (\bar u,\ro)$ for some $(\bar{u},\ro)\in H^1(0,L)\times H^1(0,L;\R^3)$. Then
    \begin{align}
      \liminf_{h\to 0}\widehat{\E}_{h}(\uu^h)\geq \widehat{\E}_0(\bar u,\ro).
    \end{align}
  \item \label{thm dim reduction b}(Upper bound). For each $(\bar{u},\ro)\in H^1(0,L)\times H^1(0,L;\R^3)$ there exists $(\uu^h)_{h>0}\subset H^1(O;\R^3)$ such that $\uu^h\stackrel{\Pi}{\to} (\bar u,\ro)$ as $h\to 0$ and
    \begin{align}\label{upper bound dim red}
      \lim_{h\to 0}\widehat{\E}_{h}(\uu^h)= \widehat{\E}_0(\bar u,\ro).
    \end{align}
  \item \label{thm dim reduction c}(Compactness and boundary conditions).
    Let $\A_0,\A_L\in\R^{2\times 2}_{\sym}$, $\K_0,\K_L\in \R^{2\times 2}_{\rm skew}$, $\mathbf{t}\in\R^3$, and $\mathbf{c}\in\R^2$. Define $\uu^h_{\rm aff}:O\to\R^3$ by
    \begin{align}\label{3.43}
      \uu^h_{\rm aff}(x):=(1-\tfrac{x_1}{L})(0,(h\A_0+\K_0)\barx)+\tfrac{x_1}{L}\Big(\mathbf{t}+\big(0,(h\A_L+\K_L)(\barx+\mathbf{c})\big)\Big),
    \end{align}
    $(\bar u_{\mathrm{aff}},\ro_{\mathrm{aff}})$ as in \eqref{affine function def} with $k_0=K_0^{21}$ and $k_L=\K_L^{21}$. If $(\uu^h)\subset H^1(O;\R^3)$ satisfies
    $$\limsup_{h\to 0}\widehat{\E}^h(\uu^h)<\infty$$
    and
    \begin{equation}\label{bcgen3d}
      \uu^h=\uu^h_{\rm aff}\text{ on the top and bottom faces } \{0\}\times S\text{ and }\{L\}\times S,
    \end{equation}
    then $\uu^h\stackrel{\Pi}{\to} (\bar u,\ro)$ for a subsequence (not relabeled) and a limit
    \begin{equation}\label{bcgen1d}
      (\bar u,\ro)\in (\bar u_{\mathrm{aff}},\ro_{\mathrm{aff}})+H_0^1(0,L)\times H_{00}^1(0,L;\R^3).
    \end{equation}
    Moreover, for any  $(\bar u,\ro)$ satisfying \eqref{bcgen1d} there exists a recovery sequence $(\uu^h)\subset H^1(O;\R^3)$ satisfying the properties of part \ref{thm dim reduction b} and additionally the boundary condition \eqref{bcgen3d}.
\end{enumerate}
\end{proposition}

We apply this proposition to treat the $\Gamma$-limit $h\to 0$ of the functionals $\E^{\vare,h}$ and $\widehat{\E}^{{\hom}}$:
  \begin{itemize}
  \item To establish the $\Gamma$-convergence of $\E^{\vare,h}(\omega)$ to $\E^{\varepsilon,{\rm rod}}(\omega)$, we shall simply set $\widehat Q(x_1,\F)=Q(\tau_{\tfrac{x_1}{\vare}}\omega_0,\F)$ and $\widehat\B(x_1)=\B(\tau_{\frac{x_1}{\vare}}\omega)$ in order to apply Proposition \ref{thm dim reduction}. Moreover, by taking $\mathbf{c}=\fint_0^L\Phi(\tau_{\frac{t}{\vare}}\omega)\,dt$ in \eqref{3.43} we see that $\uu^h$ satisfies the boundary condition \eqref{bc u}.

  \item In the same manner, the $\Gamma$-convergence of $\widehat{\E}^{{\rm hom},h}$ to $\E^\infty$ is a direct consequence of Proposition \ref{thm dim reduction} by setting $\widehat{Q}(x_1,\F)=Q^{\rm hom}(\F)$, $\widehat{\B}(x_1)=\la \B\ra$ and $\mathbf{c}=\la\Phi\ra$ therein.
  \end{itemize}

Next, we state the result for the limit $\E^{{\rm rod},\vare}(\omega)$ as $\vare\to 0$ which invokes stochastic homogenzation.
\begin{proposition}[Homogenization of $\E^{{\rm rod},\vare}$]\label{thm: 1D hom}
Let the Assumption \ref{assumption of model} and the smallness condition \eqref{eq:phic} be satisfied with $c_S=\frac{1}{2}$. Then for a.a. $\omega\in\Omega$ the following holds:
\begin{enumerate}[label=(\alph*)]
\item (Lower bound). Consider a sequence $(\bar u^\vare,\ro^\vare)_\vare\subset H^1(0,L)\times H^1(0,L;\R^3)$ with
  \begin{align}\label{finite homogenization}
    \limsup_{\vare\to 0}\,(\E^{{\rm rod},\vare}(\omega,(\bar{u}^\vare,\ro^\vare))+|(\ro_2^\vare,\ro_3^\vare)(0)|)<\infty,
  \end{align}
  and assume that
  \begin{align*}
    (\bar u^\vare,\ro^\vare)\to (\bar{u},\ro)\quad\text{strongly in $L^2(0,L)\times L^2(0,L;\R^3)$},
  \end{align*}
  for some $(\bar{u},\ro)\in H^1(0,L)\times H^1(0,L;\R^3)$. Then,
  \begin{align}
    \liminf_{\vare\to 0}\E^{{\rm rod},\vare}(\omega,(\bar{u}^\vare,\ro^\vare))\geq \E^{0}(\bar{u},\ro).
  \end{align}

\item \label{thm: 1D hom b}(Upper bound). For each $(\bar{u},\ro)\in H^1(0,L)\times H^1(0,L;\R^3)$ there exists a sequence $(\bar u^\vare,\ro^\vare)_\vare\subset H^1(0,L)\times H^1(0,L;\R^3)$ such that
\begin{align*}
(\bar{u}^\vare,\ro^\vare)\to (\bar{u},\ro)\quad\text{strongly in $L^2(0,L)\times L^2(0,L;\R^3)$},
\end{align*}
and
\begin{align}
  \lim_{\vare\to 0}\E^{{\rm rod},\vare}(\omega,(\bar{u}^\vare,\ro^\vare))= \E^{0}(\bar{u},\ro).
\end{align}

\item \label{thm: 1D hom c}(Compactness and boundary conditions).
  Let $(\bar u_{\mathrm{aff}},\ro_{\mathrm{aff}})$ be as in \eqref{affine function def}. If $(\bar u^\vare,\ro^\vare)_\vare\subset (\bar u_{\mathrm{aff}},\ro_{\mathrm{aff}})+H^1_0(0,L)\times H^1_{00}(0,L;\R^3)$  satisfies $\limsup_{\vare\to 0}\E^{\vare,\rm rod}(\omega,(\bar u^\vare,\ro^\vare))<\infty$, then $(\bar u^\vare,\ro^\vare)\to (\bar u,\ro)$ for a subsequence (not relabeled) and a limit $(\bar u,\ro)$ satisfying \eqref{bcgen1d}.
  Moreover, for any  $(\bar u,\ro)$ satisfying \eqref{bcgen1d} there exists a recovery sequence $(\bar u^\vare,\ro^\vare)_\vare\subset (\bar u_{\mathrm{aff}},\ro_{\mathrm{aff}})+H^1_0(0,L)\times H^1_{00}(0,L;\R^3)$ satisfying the properties of part \ref{thm: 1D hom b}.
\end{enumerate}
\end{proposition}

The next result establishes stochastic homogenization of the 3D model.

\begin{proposition}[Homogenization of the 3D-model]\label{thm:hom 3D model}
Let the Assumption \ref{assumption of model} and the smallness condition \eqref{eq:phic} be satisfied for $c_S$ as in Proposition~\ref{prop compact}. Then for a.a. $\omega\in\Omega$ and all $0<h\leq1$ (fixed) the following holds:
\begin{enumerate}[label=(\alph*)]
\item  (Lower bound). Consider a sequence $(\uu^\vare)_\vare\subset H^1(O;\R^3)$ with
\begin{align}\label{3D finite homogenization}
\limsup_{\vare\to 0}\,(\widehat{\E}^{\vare,h}(\omega,\uu^\vare)+\|\uu^\vare(0,\cdot)\|_{H^\frac{1}{2}(S)})<\infty,
\end{align}
and assume that
\begin{align*}
  \uu^\vare\to\uu\quad\text{strongly in $L^2(O;\R^3)$ as $\vare\to 0$}
\end{align*}
for some $\uu\in H^1(O;\R^3)$. Then
\begin{align}
\liminf_{\vare\to 0}\widehat{\E}^{\vare,h}(\omega,\uu^\vare)\geq \widehat{\E}^{{\rm hom},h}(\uu).
\end{align}

\item \label{thm:hom 3D model b}(Upper bound). For each $\uu\in H^1(O;\R^3)$ there exists $(\uu^\vare)_\vare\subset H^1(O;\R^3)$ such that
  \begin{align*}
    \uu^\vare\to\uu\quad\text{strongly in $L^2(O;\R^3)$},
  \end{align*}
  and
  \begin{align}
    \lim_{\vare\to 0}\widehat{\E}^{\vare,h}(\omega,\uu^\vare)= \widehat{\E}^{{\rm hom},h}(\uu).
  \end{align}

\item (Compactness and boundary conditions).
  If $(\uu^\vare)_\vare\subset H^1(O;\R^3)$ satisfies the boundary condition \eqref{bc u} and $\limsup_{\vare\to 0}\widehat{\E}^{\vare,h}(\uu^\vare)<\infty$, then for a subsequence (not relabeled) we have $\uu^\vare\to\uu$ in $L^2(O;\R^3)$ where $\uu\in H^1(O;\R^3)$ satisfies
  \begin{equation}\label{eq:bc3dhom}
    \uu(0,\cdot)=\uu^h_{\rm aff}(0,\cdot)\text{ and }\uu(L,\cdot)=\uu^h_{\rm aff}(L,\cdot)\text{ on }S,
  \end{equation}
  where $\uu^h_{\rm aff}$ is defined by \eqref{3D affine function def}.
  Moreover, for every $\uu\in H^1(O;\R^3)$ that satisfies \eqref{eq:bc3dhom} there exists a recovery sequence $(\uu^\vare)\subset H^1(O;\R^3)$ satisfying the boundary condition \eqref{bc u} and the properties of part \ref{thm:hom 3D model b}.
\end{enumerate}
\end{proposition}

The proof of our result on qualitative convergence is now a rather direct consequence of the previous propositions:
\begin{proof}[Proof of Theorem~\ref{main theorem qual}]
   By a scaling arugment, we may assume w.l.o.g.~that $L=1$.
  \begin{enumerate}[label=(\alph*)]
  \item The statements follow by $\Gamma$-convergence of the corresponding energy functionals and compactness properties. In particular, Propositions~\ref{thm: 1D hom} and \ref{thm:hom 3D model} yield the limits for $\vare\to0$, and Proposition~\ref{thm dim reduction} yields the limits for $h\to 0$. We only present more details for limit
    \begin{equation}\label{eq:conv eeee1231}
      \lim\limits_{h\to 0}E^{\vare,h}(\omega)=E^{{\rm rod},\vare}(\omega)\qquad\text{for $\mathbb P$-a.e. $\omega\in\Omega$},
    \end{equation}
    since the other three statements can be shown similarly. The sequence $(E^{\vare,h}(\omega))_h$ is non-negative and (thanks to Lemma~\ref{prop finite}) bounded. Choose $\uu^h\in H^1(O;\R^3)$ that satisfies the boundary condition \eqref{bc u} and $\widehat\E^{\vare, h}(\omega, \uu^h)\leq E^{\vare,h}(\omega)+h$. Since $(E^{\vare,h}(\omega))_h$ is bounded, we may appeal to the compactness and lower bound part of Proposition~\ref{thm dim reduction} and deduce $\liminf_{h\to 0}\widehat\E^{\vare,h}(\omega,\uu^h)\geq E^{{\rm rod},\vare}(\omega)$. We conclude $E^{{\rm rod},\vare}(\omega)\leq \liminf_{h\to 0}E^{\vare, h}(\omega)$. Now, let $(\bar u,\ro)\in (\bar u_{\rm aff},\ro_{\rm aff})+H^1_0(0,1)\times H^1_{00}(0,1;\R^3)$ satisfy $E^{{\rm rod},\vare}(\omega,(\bar u,\ro))=E^{{\rm rod},\vare}(\omega)$. With Proposition~\ref{thm dim reduction} we find a recovery sequenec $\uu^h\in H^1(O;\R^3)$ satisfying \eqref{bc u} and $\lim_{h\to 0}\widehat\E^{\vare,h}(\omega,\uu^h)=\E^{{\rm rod},\vare}(\omega,(\bar u,\ro))=E^{{\rm rod},\vare}(\omega)$. Hence, $\limsup_{h\to 0}E^{\vare,h}(\omega)\leq \limsup_{h\to o}\lim_{h\to 0}\widehat\E^{\vare,h}(\uu^h)=E^{{\rm rod},\vare}(\omega)$. In summary, \eqref{eq:conv eeee1231} follows.
    \item By convexity (which allows us to apply Jensen's inequality) and since $\int_0^1\ro_{23}\,dx_1=0$, we deduce that for all $(\bar u,\ro)\in (\bar u_{\rm aff},\ro_{\rm aff})+H^1_0(0,1)\times H^1_{00}(0,1;\R^3)$,
      \begin{align*}
        \E^{0}(\bar u,\ro)=&\fint_0^1Q^0\Big(                          \begin{pmatrix}
                            \pt_1\bar u+\ro_3\la\Phi_1\ra-\ro_2\la\Phi_2\ra\\
                            \pt_1\ro
                          \end{pmatrix}\Big)\,dx_1
        \geq\,Q^0\Big(\fint_0^1\begin{pmatrix}
                            \pt_1\bar u+\ro_3\la\Phi_1\ra-\ro_2\la\Phi_2\ra\\
                            \pt_1\ro
                          \end{pmatrix}\,dx_1\Big)\\
        = &Q^0\Big(\begin{pmatrix}
          \bar u(1)-\bar u(0)\\
          \ro(1)-\ro(0)
        \end{pmatrix}\Big)
        = Q^0\Big(\begin{pmatrix}
          \bar u_{\rm aff}(1)-\bar u_{\rm aff}(0)\\
          \ro_{\rm aff}(1)-\ro_{\rm aff}(0)
        \end{pmatrix}\Big)\\
        =&\E^0(\bar u_{\rm aff},\ro_{\rm aff}),
      \end{align*}
      where the last identity holds since $(\bar u_{\rm aff},\ro_{\rm aff})$ is affine and $\ro_{{\rm aff},23}=0$.
      Now the claim follows, since $\E^0(\bar u_{\rm aff},\ro_{\rm aff})=Q^{0}(\partial_1(\bar u_{\rm aff},\ro_{\rm aff}))$ and $Q^0=Q^\infty$.
  \end{enumerate}
\end{proof}

\section{Proofs of the analytical results}\label{section:proof of the main resuls}
In this section we give the proofs for our main results.
To ease notation, for an $n$-dimensional vector $\xxi=(\xxi_1,\cdots,\xxi_n)$ we simply write $\xxi_{ij}:=(\xxi_i,\xxi_j)$ for the subvector of $\xxi$ composed by its $i$-th and $j$-th components.
Particularly, for our proofs we will frequently use $\buu_{23}$ and $\ro_{23}$.
For simplicity, we also write $A\lesssim B$ (resp. $B\gtrsim A$) if there exists some $C>0$ such that $A\leq CB$.
The dependence of $C$ on the given parameters (such as $\alpha_1$ and $O$ etc.) will be made precise at the beginning of each proof.

In the proofs we often assume that w.l.o.g.~$L=1$. The general case can then be obtained by the following scaling property: For $L\geq 1$, $\vare>0$, we have
\begin{align*}
  \frac1L\int_{O}Q\bg(\tau_{\tfrac{x_1}{\vare}}\omega,\scaled \uu\big(\id+\tfrac{h}{L}\B(\tau_{\frac{x_1}{\vare}}\omega)\big)\bg)\,dx =  \int_{O_1}Q\bg(\tau_{\tfrac{x_1}{\underline\vare}}\omega,\nabla_{\!\underline h} \uu_1\big(\id+\underline h\B(\tau_{\frac{x_1}{\underline\vare}}\omega)\big)\bg)\,dx
\end{align*}
where
\begin{equation*}
  O_1=(0,1)\times S,\qquad \uu_1(x_1,\bar x)=\frac1L\uu(Lx_1,\bar x)\,\qquad \underline\vare=\vare/L,\qquad \underline h=h/L.
\end{equation*}

\subsection{Compactness: Proofs of Proposition \ref{prop compact} and Lemma \ref{prop finite}}
Our starting point is the following decomposition result due to Griso \cite{Griso2004}:

\begin{theorem}[\cite{Griso2004}, Griso's decomposition]\label{thm:griso decomp}
Let $O=(0,L)\times S$ where $L\geq 1$ and $S$ satisfies \eqref{cancelation}.
Let $\uu\in H^1(O;\R^3)$ and let $(\bar \uu,\ro)$ be associated with $\uu$ according to Definition~\ref{def of g conv}.
Define $\W$ by
\begin{align}\label{griso decomposition of u}
\uu(x)=\buu(x_1)+\ro(x_1)\wedge(0,\barx)+\W(x).
\end{align}
Then
\begin{align}\label{ineq_Griso}
\|\sym\Big(\scaled\buu+\scaled\big(\ro\wedge(0,\barx)\big)\Big) \|^2_{L^2(\OO)}+\frac{1}{h^2}\|\W\|^2_{L^2(\OO)}+\|\scaled \W\|^2_{L^2(\OO)}
\leq C_S\|\sym(\scaled \uu)\|^2_{L^2(\OO)},
\end{align}
where $C_S$ is some positive constant depending only on the cross-section $S$.
\end{theorem}

The following corollary is an immediate consequence of Theorem \ref{thm:griso decomp} and obtained by expanding the first term on the left-hand side of \eqref{ineq_Griso}.
\begin{corollary}\label{lem: griso_decomp}
Let the assumptions and notations of Theorem \ref{thm:griso decomp} be retained. Then there exists a constant $C_S$ (only depending on $S$) such that
\begin{align}
&\,\|\partial_1\bar\uu_1\|^2_{L^2(0,L)}+\|\partial_1 \ro\|^2_{L^2(0,L)}+\|\partial_1\buu_2(x_1)-\frac{1}{h}\ro_3(x_1)\|^2_{L^2(0,L)}\nonumber\\
+&\,\|\partial_1\buu_3(x_1)+\frac{1}{h}\ro_2(x_1)\|^2_{L^2(0,L)}+\frac{1}{h^2}\|\W\|^2_{L^2(\OO)}+\|\scaled \W\|^2_{L^2(\OO)}
\leq C_S\|\sym \scaled \uu\|_{L^2(O)}^2.\label{cor2}
\end{align}
\end{corollary}

With help of Corollary \ref{lem: griso_decomp} we obtain the following refinement of the decomposition \eqref{griso decomposition of u}.
We shall use it when proving the dimension-reduction results.

\begin{lemma}\label{compactness without bc}
Let $O=(0,L)\times S$ and assume \eqref{cancelation}.
Let $(\uu^h)_h$ be a sequence in $H^1(O;\R^3)$ satisfying
\begin{align}\label{finite2}
\limsup_{h\to 0}\|\sym \scaled \uu^h\|_{L^2(O)}<\infty.
\end{align}
Let $(\buu^h,\ro^h)$ be associated with $\uu^h$ according to Definition \ref{def of g conv}.
Then there exist $\VV^h,\oo^h\in H^1(O;\R^3)$ and some constant $C_O>0$ depending on $O$ such that
\begin{align}\label{decomp2}
\uu^h(x)=\begin{pmatrix}\buu^h_1(x_1)+\ro^h_2 x_3-\ro^h_3 x_2
\\ \frac{1}{h}\int_0^{x_1}\ro^h_3(t)\,dt-\ro^h_1(x_1) x_3
\\ -\frac{1}{h}\int_0^{x_1}\ro^h_2(t)\,dt+\ro^h_1(x_1) x_2\end{pmatrix}
+\VV^h(x)+\oo^h(x),
\end{align}
and the following bounds hold for all $0<h\ll 1$:
\begin{subequations}\label{eq:conv1-3}
\begin{align}
\|(\pt_1 \buu_1^h,\pt_1\ro^h)\|_{L^2(0,L)}+\|\scaled \VV^h\|_{L^2(O)}&\leq C_O\|\sym \scaled \uu^h\|_{L^2(O)},\label{conv1}\\
\lim_{h\to 0}\|\VV^h\|_{L^2(O)}&= 0,\label{conv2}\\
\lim_{h\to 0}(\|\sym \scaled \oo^h\|_{L^2(O)}+\|h\scaled \oo^h\|_{L^2(O)})&=0.\label{conv3}
\end{align}
\end{subequations}
\end{lemma}

\begin{proof}
We start with the decomposition $\uu^h(x)=\buu^h(x_1)+\ro^h(x_1)\wedge(0,\barx)+\W^h(x)$ where $\buu^h$ and $\ro^h$ are associated with $\uu^h$ according to Definition~\ref{def of g conv}.
Rewriting this decomposition yields
\begin{align*}
\uu^h(x)=
\begin{pmatrix}\buu^h_1(x_1)+\ro_2^h(x_1) x_3-\ro_3^h(x_1) x_2
\\ \frac{1}{h}\int_0^{x_1}\ro_3^h(t)\,dt-\ro^h_1(x_1) x_3
\\ -\frac{1}{h}\int_0^{x_1}\ro_2^h(t)\,dt+\ro^h_1(x_1) x_2\end{pmatrix}
+\W^h(x)+\RR^h(x_1),
\end{align*}
where
\begin{align*}
\RR^h(x_1)=
\begin{pmatrix}0
\\ \buu_2^h(x_1)-\frac{1}{h}\int_0^{x_1}\ro^h_3(t)\,dt
\\ \buu_3^h(x_1)+\frac{1}{h}\int_0^{x_1}\ro^h_2(t)\,dt\end{pmatrix}.
\end{align*}
In view of the estimate \eqref{cor2},  $(\RR^h-\fint_0^L \RR^h(x_1)\,dx_1)_h$ is bounded in $H^1(0,L;\R^3)$.
By mollifying $\RR^h$ on a scale $\gg h$ we obtain an approximating sequence $(\tkk^h)\subset H^2(0,L;\R^2)$ such that
\begin{gather}
\|\tkk^h\|_{H^1(0,L)}\lesssim\|\RR_{23}^h-\fint_0^L \RR_{23}^h(x_1)\,dx_1\|_{H^1(0,L)},\\
\lim_{h\to 0}\bg(\|\tkk^h-\big(\RR_{23}^h-\fint_0^L \RR_{23}^h(x_1)\,dx_1\big)\|_{L^2(0,L)}+h\|\tkk^h\|_{H^2(0,L)}\bg)=0.\label{conv of K}
\end{gather}
We now define
\begin{align}
\VV^h&:=\W^h
+\begin{pmatrix}
h(x_2\partial_1\tkk^h_1+x_3\partial_1\tkk^h_2)\\
\RR_{23}^h-\fint_0^L \RR_{23}^h(x_1)\,dx_1-\tkk^h
\end{pmatrix},\\
\nonumber\\
\oo^h&:=\begin{pmatrix}
-h(x_2\partial_1\tkk^h_1+x_3\partial_1\tkk^h_2)\\
\tkk^h
\end{pmatrix}+\fint_0^L \RR^h(x_1)\,dx_1.
\end{align}
Then \eqref{conv1} to \eqref{conv3} follow from \eqref{cor2} and \eqref{conv of K}.
\end{proof}

\begin{lemma}[Smallness condition for $\B$]\label{L:smallness}
  There exists a constant $c_S$ only depending on $S$ such that for all measurable and bounded $\widehat\B:O\to\R_0^{3\times 3}$ (recall that the set $\R_0^{3\times 3}$ is defined by \eqref{def of R0}) satisfying the smallness condition
  \begin{equation}\label{cond:smallness}
    |\widehat\B_{ij}|\leq c_S\qquad\text{for all }i,j\text{ and a.e.~in }O,
  \end{equation}
  the following estimate holds: For all $0<h\leq 1$, $L\geq 1$,  and $\uu\in H^1(O;\R^3)$ we have
  \begin{equation*}
    \frac12\|\sym\scaled\uu\|_{L^2(O)}\leq \|\sym(\scaled\uu(\id+\tfrac{h}{L}\widehat\B))\|_{L^2(O)}+L^{-\frac12}|\ro(0)|,
  \end{equation*}
  where $\ro$ is associated with $\uu$ via Definition~\ref{def of g conv}.
\end{lemma}
\begin{proof}
  Let $(\buu,\ro,\W)$ be defined as in \eqref{griso decomposition of u}. Since the first row of $\widehat\B$ vanishes, we have
  \begin{align*}
    \|\sym(h\scaled \uu\widehat\B)\|_{L^2(O)}\,\leq\,&\|\widehat\B\|_{L^\infty(O)}(\|\nabla_{\bar x}(\ro\wedge(0,\bar x))\|_{L^2(O)}+h\|\scaled\WW\|_{L^2(O)})\nonumber\\
    \leq &\,\|\widehat\B\|_{L^\infty(O)}(\sqrt2\sqrt{|S|}\|\ro\|_{L^2(0,L)}+h\|\scaled\WW\|_{L^2(O)}).
  \end{align*}
  Combined with the Poincar\'e's inequality in form of
  \begin{equation*}
    \|\ro\|_{L^2(0,L)}\leq L(\|\partial_{1}\ro\|_{L^2(0,L)}+L^{-\frac12}|\ro(0)|)
  \end{equation*}
  and \eqref{cor2} to estimate $\partial_{1}\ro$ and $\scaled\WW$,
  we obtain
  \begin{align*}
    \|\sym(\tfrac{h}{L}\scaled \uu\widehat\B)\|_{L^2(O)}\,\leq\,C\|\widehat\B\|_{L^\infty(O)}(\|\sym\scaled\uu\|_{L^2(O)}+L^{-\frac12}|\ro(0)|),
  \end{align*}
  where $C$ only depends on $S$. We conclude that
    \begin{align*}
      \big(1-C\|\widehat\B\|_{L^\infty(O)}\big)\|\sym\scaled \uu\|_{L^2(O)}\leq    \|\sym(\scaled\uu(I+\tfrac hL\widehat\B))\|_{L^2(O)} +C\|\widehat\B\|_{L^\infty(O)}L^{-\frac12}|\ro(0)|.
    \end{align*}
    Hence, the claim follows with $c_S:=\frac{1}{2C}$.
\end{proof}

We are now ready to prove Proposition \ref{prop compact}.
\begin{proof}[Proof of Proposition \ref{prop compact}]
  In this proof $\lesssim$ means $\leq$ up to a constant that only depends on $S$. W.l.o.g. we may assume that $L=1$.

\step 1 {\bf Proof of \ref{prop compact:a}}

Let $(\buu^h,\ro^h,\W^h)$ be associated with $\uu^h$ according to Theorem \ref{thm:griso decomp}. By Lemma~\ref{L:smallness}, \eqref{finite assump}, and the boundary condition for $\ro$, we get
\begin{equation*}
  \limsup\limits_{h\to 0}\|\sym(\scaled \uu^h)\|_{L^2(O)}<\infty.
\end{equation*}
Combined with \eqref{cor2} we obtain \eqref{finiteness of uh}, and thus \eqref{conv of projection minus mean} by the compact embedding of $H^1(0,1)$ in $L^2(0,1)$.
Next we show \eqref{conv of hu23} and \eqref{trace of r23}.
We shall only prove \eqref{conv of hu23} for $\mathbf{z}_1$ and \eqref{trace of r23} for $\ro_3$, the statement for $\mathbf{z}_2$ and $\ro_2$ can be shown similarly.
From \eqref{cor2} it follows
\begin{alignat*}{2}
&\partial_1(h\buu^h_2-\fint_{0}^1 h\buu^h_2(x_1)\,dx_1)\rightharpoonup \ro_3&&\quad\text{as $h\to 0$ in $L^2(0,1)$}.
\end{alignat*}
Hence by Poincar\'e's inequality and \eqref{conv of hu23} we know that, up to a subsequence, $h\bar \uu_2^h-\fint_{0}^1 h\buu^h_2(x_1)\,dx_1$ weakly converges to some $\mathbf{z}_{1}\in H^1(0,1)$ with $\partial_{1}\mathbf{z}_1=\ro_3$, which in turn implies \eqref{conv of hu23} and \eqref{trace of r23} by also combining with the compact embedding of $H^1(0,1)$ in $L^2(0,1)$ and integration by parts.

\step 2 {\bf Proof of \ref{prop compact:b}}

Let $(\buu^\vare,\ro^\vare,\W^\vare)$ be associated with $\uu^\vare$ according to Theorem \ref{thm:griso decomp}.
Then by the Sobolev trace theorem we have
\begin{equation*}
  |\ro^\vare(0)|\lesssim \|\uu^\vare(0,\cdot)\|_{L^2(S)}.
\end{equation*}
As in the proof of  \ref{prop compact:a} we obtain for all $0<h\leq 1$ the bound
\begin{equation*}
  \frac12\|\sym(\scaled \uu^\vare)\|_{L^2(O)}\leq \,\|\sym\bg(\scaled \uu^\vare(\id+h\widehat\B^\vare)\bg)\|_{L^2(O)}+\|\uu^\vare(0,\cdot)\|_{L^2(S)}.
\end{equation*}
In view of \eqref{finite assump hom2} and \eqref{cor2}, we thus obtain $\limsup_{\vare\to0}\frac12\|\sym(\scaled\uu^\vare)\|_{L^2(O)}<\infty$.
Consider the rescaled function $\vv^\vare:O_h\to\R^3, \vv^\vare(x):=\uu^\vare(x_1,\frac1h\bar x)$ where $O_h=(0,1)\times (hS)$.
Note that $\vv^\vare\in H^1(O_h;\R^3)$ and $\|\sym(\nabla\vv^\vare)\|_{L^2(O_h)}=h\|\sym(\scaled\uu^\vare)\|_{L^2(O)}$.
Then by appealing to Korn's inequality in form of $$\|\vv^\vare\|_{L^2(O_h)}\leq C(O_h)(\|\sym(\nabla\vv^\vare)\|_{L^2(O_h)}+\|\vv^\vare(0,\cdot)\|_{H^\frac12(hS)}),$$ and scaling back to $\uu^\vare$, we obtain \eqref{finite assump hom 2}.
By compact embedding of $H^1(O;\R^3)$ in $L^2(O;\R^3)$ the convergence \eqref{3D hom limit} follows.
\end{proof}

\begin{proof}[Proof of Lemma~\ref{prop finite}]
  In this proof $\lesssim$ means $\leq$ up to a constant that only depends on $\alpha_1,\alpha_2,O$, and the tensors appearing in the boundary condition \eqref{bc u}.
  W.l.o.g. we may assume that $L=1$.

\step 1 {\bf Proof of \ref{prop finite a}}

Consider
\begin{align}
\uu^h(x)&:=x_1\mathbf{t}+x_1\bg(0,\bg(h\A_1+\K_1\bg)\bg(\int_0^1\Phi(\tau_{\frac{t}{\vare}}\omega)\,dt+\barx\bg)\bg)^T
+\bg(1-x_1\bg)(0,(h\A_0+\K_0)\barx)^T.\label{def of uuh in prop finite}
\end{align}
Clearly, $\uu^h$ satisfies the boundary condition \eqref{bc u}.
Skew-symmetry of $\K_0$ and $\K_1$ and direct calculation yield
\begin{align}
&\,h\scaled \uu^h(x)\nonumber\\
=&\,h\mathbf{t}\otimes\e_1
+\bg(0,h\bg(h\A_1+\K_1\bg)\bg(\int_0^1\Phi(\tau_{\frac{t}{\vare}}\omega)\,dt+\barx\bg)\bg)^T\otimes \e_1
\nonumber\\
-&\,\bg(0,h\bg(h\A_0+\K_0\bg)\barx\bg)^T\otimes\e_1
+x_1\begin{pmatrix}
0&0\\
0&h\A_1+\K_1
\end{pmatrix}
+\bg(1-x_1\bg)\begin{pmatrix}
0&0\\
0&h\A_0+\K_0
\end{pmatrix},\label{long form 2}\\
\nonumber\\
&\,\sym\scaled \uu^h(x)\nonumber\\
=&\,\sym\bg(\mathbf{t}\otimes\e_1\bg)
+\sym\bg[\bg(0,\bg(h\A_1+\K_1\bg)\bg(\int_0^1\Phi(\tau_{\frac{t}{\vare}}\omega)\,dt+\barx\bg)\bg)^T\otimes \e_1\bg]
\nonumber\\
-&
\sym\bg[\bg(0,\bg(h\A_0+\K_0\bg)\barx\bg)^T\otimes\e_1\bg]
+x_1\begin{pmatrix}
0&0\\0&\sym\A_1
\end{pmatrix}+\bg(1-x_1\bg)\begin{pmatrix}
0&0\\0&\sym\A_0
\end{pmatrix}
\label{long form 1}.
\end{align}
Then \eqref{assump on Q}, \eqref{long form 2} and \eqref{long form 1} imply
\begin{align*}
\limsup\limits_{h\to 0}\,\widehat\E^{\vare,h}(\omega,\uu^h)\lesssim&\,\limsup_{h\to 0}\bg\|\sym \bg(\scaled \uu^h\big(\id+h\B(\tau_{\frac{x_1}{\varepsilon}}\omega)\big)\bg)\bg\|^2_{L^2(O)}\nonumber\\
\lesssim&\,\bg(|\mathbf t|+(|\A_0|+|\A_1|)+(1+c)(|\K_0|+|\K_1|)\bg)^2
\end{align*}
and
\begin{align*}
&\,\limsup_{h\to 0}|\ro^h(0)|\nonumber\\
\lesssim&\,\limsup_{h\to 0}\bg|\int_{S}(0,\barx)\wedge \buu^h(0,\barx)\,d\barx\bg|\nonumber\\
=&\,\limsup_{h\to 0}\bg|\int_{S}(0,\barx)\wedge (0,(h\A_0^{11}+\K_0^{11})x_2+(h\A_0^{12}+\K_0^{12})x_3,
(h\A_0^{21}+\K_0^{21})x_2+(h\A_0^{22}+\K_0^{22})x_3)\,d\barx\bg|\nonumber\\
<&\,\infty,
\end{align*}
from which \eqref{uniform norm indep of eps} follows.
Finally, the boundary condition of $(\bar u,\ro)$ follows immediately from the trace theorem.
To see that $\int_0^1\ro_{23}(t)\,dt=0$, direct calculation shows that $\bar\uu_{23}^h(0)$ and $\bar\uu_{23}^h(1)$ are uniformly bounded in $h$, hence
\begin{align}\label{trace to zero}
h\bar \uu^h_{23}(t)-\fint_0^1h\bar \uu^h_{23}(x_1)\,dx_1\to 0\quad\text{as $h\to 0$ for $t\in\{0,1\}$},
\end{align}
and the claim follows from \eqref{trace of r23}.

\step 2 {\bf Proof of \ref{prop finite b}}

The argument is a straightforward modification of the proof of \ref{prop finite a}: We simply replace $\uu^h$ in \eqref{def of uuh in prop finite} by $\uu^\vare$  (and consider the limit $\vare\to 0$ for $0<h\leq 1$ fixed).
The boundary conditions for $\uu^\vare$ then follow with help of the Sobolev trace theorem and Birkhoff's Theorem \ref{thm:ergodic-thm}.
We omit further details.
\end{proof}


\subsection{Dimension reduction: Proof of Proposition~\ref{thm dim reduction}}

\begin{proof}[Proof of Proposition \ref{thm dim reduction}]
  We split the proof into three parts corresponding to the lower bound, upper bound and boundary compatibility statements respectively. W.l.o.g.~we may assume that $L=1$.

\step 1 {\bf Proof of the lower bound}

Thanks to Proposition \ref{prop compact} we may apply Lemma \ref{compactness without bc}, and thus the decomposition \eqref{decomp2} holds with functions $(\buu^h,\ro^h,\VV^h,\oo^h)$ satisfying the bounds \eqref{conv1} to \eqref{conv3}.
Applying $\scaled$ to \eqref{decomp2} we see that
\begin{align}
h\scaled \uu^h
=\begin{pmatrix}
h\pt_1\buu^h_1+h\pt_1\ro^h_2 x_3-h\pt_1\ro^h_3 x_2&-\ro^h_3&\ro^h_2\\
\ro^h_3-h\pt_1\ro^h_1 x_3&0&-\ro^h_1\\
-\ro^h_2+h\pt_1\ro_1^hx_2&\ro^h_1&0
\end{pmatrix}
+h\scaled \VV^h+h\scaled \oo^h,\label{long sum}
\end{align}
and thus
\begin{align}
\sym\scaled\uu^h(x)=\sym(\pt_1\buu^h_1(x_1)\e_1+\pt_1\ro^h(x_1)\wedge(0,\barx))\otimes \e_1+\sym\scaled \VV^h(x)+\sym\scaled\oo^h(x)\label{decomp of sym scaled uh}.
\end{align}
In view of \eqref{conv1} and since $\uu^h\stackrel{\Pi}{\to} (\bar u,\ro)$ by assumption, we deduce that $(\bar{\uu}_1^h,\ro^h)\to(\bar u,\ro)$ weakly in $H^1(0,1)\times H^1(0,1;\R^3)$.
Furthermore, in view of \eqref{conv1} to \eqref{conv3}, and \eqref{long sum} we see that
\begin{align}\label{weak 1}
\sym(h\scaled \uu^h\widehat{\B})\rightharpoonup
\sym\bg(\big(\mathbb{K}\ro\big)\widehat{\B}\bg)\quad\text{in }L^2(O;\R_{\sym}^{3\times 3}),
\end{align}
where $\mathbb{K}\in\rm{Lin}(\R^3;\R_{\rm skw}^{3\times 3})$ is defined by
\begin{align}\label{def of K}
\mathbb{K}\ro =\begin{pmatrix}
0&-\ro_3&\ro_2\\
\ro_3&0&-\ro_1\\
-\ro_2&\ro_1&0
\end{pmatrix}.
\end{align}
Next, we identify the weak limit of \eqref{decomp of sym scaled uh}.
First, we note that
\begin{align*}
\scaled \VV^h=(\pt_1 \VV^h,0,0)+(0,\nabla_{\barx}\ZZ^h),
\end{align*}
where $\ZZ^h(x):=\frac{1}{h}(\VV^h(x)-\fint_{S}\VV^h(x_1,\barx)\,d\barx)$.
By \eqref{conv1} and \eqref{conv2}, $(\ZZ^h)_h$ is bounded in $L^2(0,1;H^1(S;\R^3))$ and $\VV^h\to 0$.
Hence, for a subsequence (not relabeled) we have
\begin{alignat*}{2}
\pt_1\VV^h&\rightharpoonup 0&&\quad\text{in $L^2(O;\R^3)$},\\
\nabla_{\barx}\ZZ^h&\rightharpoonup \nabla_{\barx}\ZZ&&\quad\text{in $L^2(0,1;L^2(S;\R^{3\times 2}))$},
\end{alignat*}
with $\ZZ\in L^2(0,1;H^1(S;\R^3))$.
We thus obtain
\begin{align}\label{weak 2}
\sym\scaled\uu^h(x)\rightharpoonup
\sym((\pt_1\bar{u}(x_1)\e_1+\pt_1\ro(x_1)\wedge(0,\barx))\otimes \e_1)+\sym(0,\nabla_{\barx}\ZZ(x))
\end{align}
in $L^2(0,1;L^2(S;\R^{3\times 3}_\sym))$.

Next, by weak lower semi-continuity of the functional $L^2(O;\R^{3\times 3})\ni\F\mapsto \int_O \widehat Q(x_1,\F)\,dx$ we conclude that
\begin{align}
  \nonumber
\liminf_{h\to 0}\widehat{\E}^{h}(\uu^h)
\,\geq&\,
      \int_O \widehat Q\bg(x_1,\sym((\pt_1\bar{u}(x_1)\e_1+\pt_1\ro(x_1)\wedge(0,\barx))\otimes \e_1)\\
  &\qquad\qquad +
\sym\bg(\big(\mathbb{K}\ro\big)\widehat{\B}(x_1)\bg)+
    \sym(0,\nabla_{\barx}\ZZ(x))\bg)\,dx.
\label{lower bound 1}
\end{align}
Notice that
\begin{equation}\label{eq:st112002}
  \sym\bg(\big(\mathbb{K}\ro\big)\widehat{\B}\bg)=(-\ro_3\widehat{\B}_{21}+\ro_2\widehat{\B}_{31})\e_1\otimes\e_1
+\sym(0,\nabla_{\bar x}\widehat\varphi),
\end{equation}
where
\begin{equation}\label{eq:st:widehatvarphi}
  \widehat\varphi(x):=\begin{pmatrix}
    (-\ro_1\widehat{\B}_{31}-\ro_3\widehat{\B}_{22}+\ro_2\widehat{\B}_{32})x_2+(\ro_1\widehat{\B}_{31}-\ro_3\widehat{\B}_{23}+\ro_2\widehat{\B}_{33})x_3\\
-\ro_1\widehat{\B}_{32}x_2-\ro_1\widehat{\B}_{33}x_3\\
\ro_1\widehat{\B}_{22}x_2+\ro_1\widehat{\B}_{23}x_3
\end{pmatrix}.
\end{equation}
Moreover, $\widehat\varphi(x_1,\cdot)+\ZZ(x_1,\cdot)\in H^1(S;\R^3)$ for a.e.~$x_1$.
Hence, with
\begin{equation*}
  \xxi:=
  \begin{pmatrix}
    \pt_1\bar{u}-\ro_3\widehat{\B}_{21}+\ro_2\widehat{\B}_{31}\\
    \pt_1\ro
  \end{pmatrix},
\end{equation*}
we get
\begin{align*}
  \text{[R.H.S.}&\text{\ of \eqref{lower bound 1}]}=\,\int_O\widehat Q\bg(x_1, (\xxi_1\e_1+\xxi_{234}\wedge(0,\bar x))\otimes\e_1+(0,\nabla_{\bar x}\widehat\varphi+\nabla_{\bar x}\ZZ)\bg)\,dx\\
  \geq&\int_0^1\inf_{\varphi\in H^1(S;\R^3)}\int_S\widehat Q\bg(x_1,(\xxi_1\e_1+\xxi_{234}\wedge(0,\bar x))\otimes\e_1+(0,\nabla_{\bar x}\varphi)\bg)\,d\bar x\,dx_1\\
  =&\int_0^1\widehat Q^{\rm rod}(x_1,\xxi)\,dx_1.
\end{align*}

\step 2 {\bf Proof of the upper bound}

For convenience set
\begin{equation*}
  \xxi:=
  \begin{pmatrix}
    \partial_1\bar u-\ro_3\widehat{\B}_{21}+\ro_2\widehat{\B}_{31}\\
    \partial_1\ro
  \end{pmatrix}.
\end{equation*}
Note that with this notation we have  $\widehat{\E}(\bar u,\ro)=\int_0^1\widehat Q^{\rm rod}(x_1,\xxi)\,dx_1$.
Choose $\varphi\in L^2(0,1;H^1(S;\R^3))$ such that
\begin{equation}\label{eq:12sdsdklk2399d}
  \int_0^1\widehat Q^{\rm rod}(x_1,\xxi)\,dx_1=\int_O\widehat Q\bg(x_1,(\xxi_1\e_1+\xxi_{234}\wedge(0,\bar x))\otimes\e_1+(0,\nabla_{\bar x}\varphi)\bg)\,dx,
  \end{equation}
  and let $\widehat\varphi$ be defined as in \eqref{eq:st:widehatvarphi}.
Now, let $(\VV^{h})_h\subset C_c^\infty(0,1,C^\infty(\overline S;\R^3))$ denote an approximating sequence satisfying
\begin{align*}
\lim_{h\to 0}\Big(\|\nabla_{\bar x}(\VV^{h}-(\varphi-\widehat\varphi))\|_{L^2(O)}+
\|h\pt_1\VV^h\|_{L^2(O)}+\|h\VV^h\|_{L^2(O)}\Big)=0
\end{align*}
and set
\begin{align}
\uu^{h}(x):=\begin{pmatrix}\bar u(x_1)+\ro_2(x_1) x_3-\ro_3(x_1) x_2
\\ \frac{1}{h}\int_0^{x_1}\ro_3(t)\,dt-\ro_1(x_1) x_3
\\ -\frac{1}{h}\int_0^{x_1}\ro_2(t)\,dt+\ro_1(x_1) x_2\end{pmatrix}+h\VV^{h}(x).\label{recovery seq}
\end{align}
Then $\uu^h\stackrel{\Pi}{\to} (\bar u,\ro)$ and a calculation similar to Step~1 shows that
\begin{equation*}
  \sym\big(\scaled \uu^h(\id+h\widehat\B)\big)\to \sym((\partial_1\bar u\e_1+\partial_1\ro\wedge (0,\barx))\otimes\e_1)+\sym\bg(\big(\mathbb{K}\ro\big)\widehat{\B}\bg)+\sym(0,\nabla_{\bar x}\varphi-\nabla_{\bar x}\widehat\varphi)
\end{equation*}
strongly in $L^2(O;\R^{3\times 3})$. In view of the definition of $\varphi$ and $\widehat\varphi$ we deduce that the right-hand side equals
\begin{equation*}
\sym[ (\xxi_1\e_1+\xxi_{234}\wedge(0,\bar x))\otimes\e_1]+\sym(0,\nabla_{\bar x}\varphi).
\end{equation*}
Hence, by continuity of the functional $L^2(O;\R^{3\times 3})\ni\F\mapsto \int_O\widehat Q(x_1,\F)\,dx$ we conclude that
\begin{equation*}
  \lim\limits_{h\to 0}\widehat{\E}^h(\uu^h)=\int_O\widehat Q\bg(x_1,(\xxi_1\e_1+\xxi_{234}\wedge(0,\bar x))\otimes\e_1+(0,\nabla_{\bar x}\varphi)\bg)\,dx.
  \end{equation*}
  In view of \eqref{eq:12sdsdklk2399d} this completes the proof of Step 2.

  \step 3 {\bf Proof of compactness and boundary conditions}

  Let $(\uu^h)_h\subset H^1(O;\R^3)$ satisfy $\limsup_{h\to 0}\widehat{\E}^h(\uu^h)<\infty$ and the boundary condtion \eqref{bcgen3d}.
The latter yields
  \begin{equation}\label{eq:bchh}
    \bar\uu^h(0)=0,\quad \bar\uu^h(1)=\mathbf{t}+(0,(h\A_1+\K_1)\mathbf c),\quad  \ro^h(0)=\ro_{\rm aff}(0)+O(h),\quad\ro^h(1)=\ro_{\rm aff}(1)+O(h).
  \end{equation}
  Hence $\limsup_{h\to 0}|\ro^h(0)|<\infty$ and consequently \eqref{finite assump} follows.
By Proposition~\ref{prop compact} we now get the bound \eqref{finiteness of uh}.
Since $\bar\uu^h(0)=0$, we conclude that $\uu^h\stackrel{\Pi}{\to}(\bar u,\ro)$ for a subsequence, and in view of \eqref{eq:bchh} we find that $(\bar u(x_1),\ro(x_1))=(\bar u_{\rm aff}(x_1),\ro_{\rm aff}(x_1))$ for $x_1\in\{0,1\}$.
In particular, we see that $h(\bar\uu^h_{23}(1)-\bar\uu^h_{23}(0))\to \mathbf{z}(1)-\mathbf{z}(0)$, cf.~\eqref{conv of hu23}.
Combined with \eqref{trace of r23} we get $\int_0^1\ro_{23}\,dx_1=0$.
In summary, we conclude \eqref{bcgen1d} , which completes the proof of the first claim.

  Next, we assume that $(\bar u,\ro)$ satisfies \eqref{bcgen1d}.
  We need to construct a recovery sequence satisfying additionally \eqref{bcgen3d}.
  To that end let $\uu^h$ and $\VV^h$ be as in Step~2.
  Since $\VV^h(x_1,\cdot)$ vanishes for $x_1\in\{0,1\}$ and in view of \eqref{bcgen1d}, we already have $\uu^h(0,\bar x)=(0,\K_0\bar x)^T$ and $\uu^h(1,\bar x)=(\mathbf{t}_1,\K_1\bar x)^T$.
In order to achieve the asserted boundary conditions, we add an appropriate correction to $\uu^h$.
For our purpose we introduce the affine displacement
\begin{equation*}
  \ww^h(x):=(1-x_1)(0,h\A_0\barx)+x_1(0,\mathbf{t}_{23}+h\A_1(\bar x+\mathbf c)+\K_1\mathbf{c}).
\end{equation*}
Note that $\uu^h+\ww^h$ satisfies the boundary conditions \eqref{bcgen3d} and converges to the same limit as $\uu^h$.
A direct calculation shows that
\begin{equation}\label{eq:st:110002}
  \sym\bg(\scaled\ww^h
(\id+h\widehat\B)\bg)\to \sym\begin{pmatrix}
 0&0\\
 (\mathbf{t}_{23}+\K_1\mathbf{c})&\A_0+x_1(\A_1-\A_0)
 \end{pmatrix}=:\sym(0,\nabla_{\bar x}\widetilde\varphi)
\end{equation}
strongly in $L^2(O;\R^{3\times 3})$ for a suitbale function $\widetilde\varphi\in L^2(0,1;H^1(S;\R^3))$.
Analogously to Step~2, let $(\WW^{h})_h\subset C_c^\infty(0,1,C^\infty(\overline S;\R^3))$ denote an approximating sequence satisfying
\begin{align*}
\lim_{h\to 0}\Big(\|\nabla_{\bar x}(\WW^{h}-\widetilde\varphi)\|_{L^2(O)}+
\|h\pt_1\WW^h\|_{L^2(O)}+\|h\WW^h\|_{L^2(O)}\Big)=0,
\end{align*}
and consider
\begin{equation*}
  \tilde\uu^h:=\uu^h+\ww^h-h\WW^h.
\end{equation*}
Then $\tilde\uu^h$ satisfies the claimed boundary conditions \eqref{bcgen3d} and $\tilde\uu^h\stackrel{\Pi}{\to}(\bar u,\ro)$.
Furthermore, by \eqref{eq:st:110002} and the construction of $\WW^h$, we have
\begin{equation*}
  \|  \sym\bg(\scaled\tilde\uu^h
(\id+h\widehat\B)\bg)-  \sym\bg(\scaled\uu^h
(\id+h\widehat\B)\bg)\|_{L^2(O)}=0,
\end{equation*}
and thus (by Step~2),
\begin{equation*}
  \lim\limits_{h\to 0}\widehat{\E}^h(\tilde\uu^h)=  \lim\limits_{h\to 0}\widehat{\E}^h(\uu^h)=\widehat{\E}^0(\bar u,\ro).
\end{equation*}
\end{proof}

\subsection{Homogenization: Proofs of Lemma \ref{L:quadr}, Propositions~\ref{thm: 1D hom} and \ref{thm:hom 3D model}}\label{sec homogenization}
In this section we present the proofs for the stochastic homogenization results.
We appeal to stochastic two-scale convergence, see Appendix~\ref{appendix: two scale conv} for definitions and auxiliary lemmas.

We begin with the proof of the auxiliary Lemma \ref{L:quadr}.

\begin{proof}[Proof of Lemma \ref{L:quadr}]
We split our proof into two steps.

\step 1 {\bf Proofs of \ref{lem coer a} and \ref{lem coer b}}

Since the proofs of \ref{lem coer a} and \ref{lem coer b} are almost identical, we shall give here only the proof of \ref{lem coer a}. Let $\omega\in \Omega$ such that $Q(\omega,\cdot)\in\mathcal Q(\alpha_1,\alpha_2)$ (which holds $\mathbb{P}$-a.s. by assumption). Then we know that for all $\xxi\in\R^4$ and $\varphi\in H^1(S;\R^3)$
\begin{align}
&\,\int_SQ\Big(\omega,
 (\xxi_1\e_1+\xxi_{234}\wedge (0,\barx))\otimes\e_1+(0,\nabla_{\bar x}\varphi)\Big)\,d\bar x \nonumber\\
\geq&\, \alpha_1\|\sym[(\xxi_1\e_1+\xxi_{234}\wedge (0,\barx))\otimes\e_1]+\sym(0,\,\nabla_x\varphi)\|^2_{L^2(S)}.
\end{align}
Hence the lower bound will follow, as long as we can prove that there exists some positive constant $C_S$, depending only on $S$, such that for all $\xxi\in \R^4$ and $\varphi\in H^1(S;\R^3)$ we have
\begin{align}
\|\sym((\xxi_1\e_1+\xxi_{234}\wedge (0,\barx))\otimes\e_1)+\sym(0,\,\nabla_{\bar x}\varphi)\|^2_{L^2(S)}\geq C_S(|\xxi|^2 +\|\sym(0,\,\nabla_{\bar x}\varphi)\|^2_{L^2(S)}).\label{ineq:coercive1}
\end{align}
Assume therefore that \eqref{ineq:coercive1} does not hold. Then we could find a sequence $(\xxi^n,\varphi^n)_{n\in\N}\subset
\R^4\times H^1(S;\R^3)$ such that $|\xxi^n|^2+\|\Theta^n\|^2_{L^2(S)}=1$ with $\Theta^n=\sym(0,\,\nabla_{\bar x}\varphi^n)$ and
\begin{align}\label{ineq:less than 1/n}
\|\sym((\xxi_1^n\e_1+\xxi_{234}^n\wedge (0,\barx))\otimes\e_1)+\Theta^n\|^2_{L^2(S)}\leq n^{-1}.
\end{align}
By direct calculation we already see that \eqref{ineq:less than 1/n} implies
\begin{align}
|\xxi^n_{134}|&\to 0\quad\text{in $\R^3$},\\
\sym\nabla_{\bar x}\varphi^n_{23}&\to 0\quad\text{in $L^2(S;\R_{\sym}^{2\times 2})$},\\
\pt_2\varphi^n_1-\xxi^n_2 x_3,\pt_3\varphi_1^n+\xxi^n_2 x_2&\to 0\quad\text{in $L^2(S)$}\label{ineq:vanishing}
\end{align}
as $n\to \infty$. We now take the distributional derivative $\pt_3$ and $\pt_2$ on $\pt_2\varphi^n_1-\xxi^n_2 x_3$ and $\pt_3\varphi_1^n+\xxi^n_2 x_2$ respectively and then subtract the latter from the former, then \eqref{ineq:vanishing} implies
$$\xxi_2:=\lim_{n\to\infty}\xxi_2^n=0.$$
Combining with \eqref{ineq:vanishing} and the Poincar\'e's inequality we infer that $\nabla_{\bar x}\varphi^n_1\to 0$ in $L^2(S;\R^2)$ and consequently
\begin{align*}
|\xxi^n|^2+\|\Theta^n\|^2_{L^2(S)}\to 0\quad\text{as $n\to \infty$}.
\end{align*}
This contradicts the fact that $|\xxi^n|^2+\|\Theta^n\|^2_{L^2(S)}\equiv 1$ and \eqref{ineq:coercive1} follows. On the other hand, the upper bound follows immediately by setting a test function $\varphi$ equal to zero and the fact that $Q(\omega,\cdot)\in\mathcal Q(\alpha_1,\alpha_2)$. This completes the proof of Step 1.

\step 2 {\bf Proof of \ref{lem coer c}}

It is clear by definition that if $Q$ is independent of $\omega$, then so is $Q^{\rm rod}$. Let now $\mathbb{L}^{\rm rod}$ be the symmetric matrix of the bilinear form associated with $Q^{\rm rod}$. Then
$$Q^{\rm rod}(\xxi+\chi)=Q^{\rm rod}(\xxi)+Q^{\rm rod}(\chi)+2\mathbb{L}^{\rm rod}\xxi\cdot\chi.$$
Since $\chi\in L_0^2(\Omega;\R^4)$, we know that $\int_\Omega \mathbb{L}^{\rm rod}\xxi\cdot\chi\,d\mathbb{P}(\omega)=0$. Thus
\begin{align}
Q^0(\xxi)=\inf_{\chi\in L^2_0(\Omega;\R^4)}\int_\Omega Q^{\rm rod}(\xxi+\chi)\,d\mathbb P(\omega)
=Q^{\rm rod}(\xxi)+\inf_{\chi\in L^2_0(\Omega;\R^4)} Q^{\rm rod}(\chi)=Q^{\rm rod}(\xxi),
\end{align}
from which the first identity in \ref{lem coer c} follows. On the other hand, by definition we have $Q=Q^{\rm hom}$ when $Q$ is independent of $\omega$. Then the second identity in \ref{lem coer c} follows directly from the fact that $Q^0$ and $Q^\infty$ are defined identically.
\end{proof}

We are now ready to prove Proposition \ref{thm: 1D hom}.

\begin{proof}[Proof of Proposition \ref{thm: 1D hom}]
  In this proof $\lesssim$ means $\leq$ up to a constant that only depends on $\alpha_1$ and $O$.
W.l.o.g.~we assume that the set $\sD_\Omega$ in the definition of the space  of two-scale test-functions $\sD$ contains $\Phi_1$ and $\Phi_2$.
Furthermore, we fix $\omega_0\in\Omega_Q$, where $\Omega_Q$ denotes the set of full measure obtained by Lemma~\ref{L:twoscale:quadratic} applied to $Q^{\rm rod}$.
Furthermore, we simply write $\twos$ and $\twoss$ instead of $\twos_{\omega_0}$ and $\twoss_{\omega_0}$ to denote weak and strong two-scale convergence, respectively (since the sample $\omega_0$ under consideration is always chosen fixed).
W.l.o.g. we assume that $L=1$.

  \step 1 {\bf Proof of the lower bound}

Set
\begin{equation*}
  \vv^\vare:=
  \begin{pmatrix}
    \pt_1\bar u^\vare+\ro_3^\vare\Phi_1(\tau_{\frac{x_1}{\vare}}\omega_0)-\ro_2^\vare\Phi_2(\tau_{\frac{x_1}{\vare}}\omega_0)\\
    \pt_1\ro^\vare
  \end{pmatrix},
\end{equation*}
so that
\begin{align*}
\E^{\vare,{\rm rod}}(\omega_0,(\bar{u}^\vare,\ro^\vare))&:=\int_0^1
                                    Q^{\rm rod}\Big(\tau_{\frac{x_1}{\vare}}\omega_0,\vv^\vare\Big)\,dx_1.
\end{align*}
We first identify the weak two-scale limit of $(\vv^\vare)_\vare$.
In view of Lemma~\ref{L:quadr} and with help of the bound $\|\Phi\|_{L^\infty(\Omega)}\leq c_S$, see \eqref{eq:phic}, and Poincar\'e's inequality in form of $\|\ro^\vare_{23}\|_{L^2(0,1)}\leq \|\partial_1\ro^\vare_{23}\|_{L^2(0,1)}+|\ro_{23}(0)|$, we get
\begin{align*}
  \frac{1}{\beta_1}\E^{\vare,{\rm rod}}(\omega,(\bar{u}^\vare,\ro^\vare))\geq& \|\pt_1\ro^\vare\|^2_{L^2(0,1)}+\frac{1}{2}\|\pt_1\bar u^\vare\|^2_{L^2(0,1)}-c_S^2\|\ro^\vare_{23}\|^2_{L^2(0,1)}\\
  \geq& \|\pt_1\ro^\vare\|^2_{L^2(0,1)}+\frac{1}{2}\|\pt_1\bar u^\vare\|^2_{L^2(0,1)}-2c_S^2\|\pt_1\ro^\vare_{23}\|^2_{L^2(0,1)}-2c_S^2|\ro_{23}^\vare(0)|^2.
\end{align*}
Since $c_S^2\leq\frac14$ by assumption, we may absorb the third term on the right-hand side into the first term.
Hence, the assumption  $\limsup_{\vare\to 0} (\E^{\vare,{\rm rod}}(\omega,(\bar{u}^\vare,\ro^\vare))+|\ro_{23}^\vare(0)|^2)<\infty$ yields the bound
\begin{align*}
\limsup_{\vare\to 0} (\|\pt_1(\bar u^\vare,\ro_1^\vare)\|_{L^2(0,1)}+\|\ro^\vare_{23}\|_{H^1(0,1)})<\infty.
\end{align*}
We conclude that $(\bar u^\vare,\ro^\vare)$ weakly converges to $(\bar u,\ro)$ in $H^1(0,1)\times H^1(0,1;\R^3)$ (and not only strongly in $L^2$ as assumed).
Hence, by Lemma \ref{bound lemma} we may pass to a subsequence such that $\pt_1(\bar u^\vare,\ro^\vare)\twos\pt_1(\bar u,\ro)+\chi$ weakly two-scale with $\chi\in L^2(0,1;L^2_{0}(\Omega;\R^4))$.
Furthermore, since $\ro^\vare\to \ro$ strongly in $L^2(0,1;\R^3)$ (and $\Phi_1,\Phi_2\in\sD_\Omega$ by assumption), we have
\begin{equation*}
  \ro^\vare_3(x_1)\Phi_1(\tau_{\frac{x_1}{\vare}}\omega_0)-\ro^\vare_2(x_1)\Phi_2(\tau_{\frac{x_1}{\vare}}\omega_0)\twoss
\ro_3\Phi_1-\ro_2\Phi_2\quad\text{weakly two-scale},
\end{equation*}
see Proposition~\ref{P:twoscale} \ref{P:twoscale:strongimpliestrong}. Thus, we conclude that
\begin{equation*}
  \vv^\vare\twos
  \underbrace{\begin{pmatrix}
    \pt_1\bar u+\ro_3\langle\Phi_1\rangle-\ro_2\langle\Phi_2\rangle\\
    \pt_1\ro
  \end{pmatrix}}_{=:\vv}+\widehat\chi
\end{equation*}
for some $\widehat\chi\in L^2(0,1;L^2_{0}(\Omega;\R^4))$.
Hence, by weak two-scale lower semicontinuity of convex functionals, cf.~Lemma \ref{L:twoscale:quadratic}, we have
\begin{align*}
  \liminf\limits_{\vare\to0}\E^{\vare,{\rm rod}}(\omega_0,(\bar{u}^\vare,\ro^\vare))&=  \liminf\limits_{\vare\to0}\int_0^1
                                                                           Q^{\rm rod}\Big(\tau_{\frac{x_1}{\vare}}\omega_0,\vv^\vare\Big)\,dx_1\,\geq\,  \int_0^1\int_\Omega
         Q^{\rm rod}\Big(\omega',\vv+\widehat\chi\Big)\,d\mathbb P(\omega') dx_1\\
  &\geq \int_0^1\inf_{\chi\in L^2_{\rm pot}(\Omega;\R^4)}\int_\Omega
         Q^{\rm rod}\Big(\omega,\vv+\chi\Big)\,d\mathbb P(\omega)\,dx_1
         = \int_0^1Q^{0}(\vv)\,dx_1\\
  &=\E^0(\bar u,\ro).
\end{align*}
This completes the proof of the lower bound.

\step 2 {\bf Proof of the upper bound}

Set
\begin{equation*}
  \vv:=
  \begin{pmatrix}
    \pt_1\bar u+\ro_3\la\Phi_1\ra-\ro_2\la\Phi_2\ra\\
    \pt_1\ro
  \end{pmatrix},
\end{equation*}
so that $\E^0(\bar u,\ro)=\int_0^1Q^{0}(\vv)\,dx_1$.
Choose $\widehat\chi\in L^2(0,1;L^2_{0}(\Omega;\R^4))$ such that
\begin{equation*}
  \E^0(\bar u,\ro)=\int_0^1\int_\Omega Q^{\rm rod}(\omega,\vv+\widehat\chi\,)\,d\mathbb P(\omega) dx_1.
\end{equation*}
Note that
\begin{equation*}
  \tilde\chi:=\chi+\Big(\ro_3(\Phi_1-\langle\Phi_1\rangle)-\ro_2(\Phi_2-\langle\Phi_2\rangle)\Big)\e_1\in L^2(0,1;L^2_{0}(\Omega;\R^4)).
\end{equation*}
Thanks to Lemma~\ref{bound lemma} (applied with $O=(0,1)$) there exists a sequence $(\varphi^\vare)_\vare\subset C^\infty_c(0,1;\R^4)$ such that $\varphi^\vare\to 0$ uniformly and $\varphi^\vare\twoss \tilde\chi$ strongly two-scale.
Set
\begin{equation*}
  \bar u^\vare:=\bar u+\varphi^\vare_1,\qquad \ro^\vare:=\ro+\varphi_{234}^\vare, \qquad\vv^\vare:=  \begin{pmatrix}
    \pt_1\bar u^\vare+\ro_3^\vare\Phi_1(\tau_{\tfrac{x_1}{\vare}}\omega_0)-\ro_2^\vare\Phi_2(\tau_{\tfrac{x_1}{\vare}}\omega_0)\\
    \pt_1\ro^\vare
  \end{pmatrix}.
\end{equation*}
Then by construction, and since $\Phi_1,\Phi_2\in\sD_\Omega$, we have $\vv^\vare\twoss \vv+\chi$ strongly two-scale.
Hence, by Lemma~\ref{L:twoscale:quadratic} we get
\begin{equation*}
  \lim\limits_{\vare\to 0}\E^{\vare,{\rm rod}}(\omega_0,(\bar u^\vare,\ro^\vare))=  \lim\limits_{\vare\to 0}\int_0^1Q^{\rm rod}(\tau_{\tfrac{x_1}{\vare}}\omega,\vv^\vare)\\
=\int_0^1\int_\Omega Q^{\rm rod}(\omega,\vv+\chi)\,d\mathbb P(\omega)\,dx_1=\E^0(\bar u,\ro).
\end{equation*}

\step 3 {\bf Compactness and boundary conditions}

Consider $(\bar u,\ro)\in(\bar u_{\mathrm{aff}},\ro_{\mathrm{aff}})+H_0^1(0,1)\times H_{00}^1(0,1;\R^3)$ and assume that
$\limsup_{\vare\to 0}\E^{\vare,{\rm rod}}(\omega, (\bar u^\vare,\ro^\vare))<\infty$. Combined with Lemma~\ref{L:quadr}, the boundary condition for $(\bar u^\vare,\ro^\vare)$ and Poincar\'e's inequality we get  $\limsup_{\vare\to0}(\|\bar u^\vare\|_{H^1(0,1)}+\|\ro^\vare\|_{H^1(0,1)})<\infty$. Hence, we can pass to a subsequence that converges claimed.

Next, we argue that we can construct a recovery sequence that additionally satisfies the boundary condition.
In fact this only requires a minor modification of the sequence of Step~2, since the sequence $(\bar u^\vare,\ro^\vare)$  of Step~3 already satisfies all boundary conditions except for the condition $\int_0^1\ro^\vare_{23}=0$.
To correct this let $\theta\in C^\infty_c(0,1)$ satisfy $\int_0^1\theta(t)\,dt=1$.
Then the modified sequence $(\bar u^\vare,\tilde\ro^\vare)_\vare$ with
\begin{align*}
\tilde\ro^\vare=\big(\ro^\vare_1,\ro^\vare_{23}-\theta\int_0^1(0,\ro^\vare_{23}(t))\,dt\big)
\end{align*}
is a recovery sequence and satisfies all conditions.
\end{proof}

\begin{proof}[Proof of Proposition \ref{thm:hom 3D model}]
W.l.o.g.~we assume that the set $\sD_\Omega$ in the definition of the space  of two-scale test-functions $\sD$ contains $\Phi_1$ and $\Phi_2$.
Furthermore, we fix $0<h\leq 1$ and $\omega_0\in\Omega_Q$, where $\Omega_Q$ denotes the set of full measure obtained by Lemma~\ref{L:twoscale:quadratic} applied to $Q$.
Furthermore, we simply write $\twos$ and $\twoss$ instead of $\twos_{\omega_0}$ and $\twoss_{\omega_0}$ to denote weak and strong two-scale convergence, respectively.
W.l.o.g.~we assume that $L=1$.

\step 1 {\bf Proof of the lower bound}

By Proposition \ref{prop compact}, \eqref{finite assump hom 2}, we know that $\uu$ is the weak $H^1$-limit of $(\uu^\vare)_\vare$.
From Lemma~\ref{bound lemma} we thus deduce that for a subsequence and $\tilde\chi\in L^2(O;L^2_0(\Omega;\R^3))$,
\begin{align*}
\scaled \uu^\vare\twos\scaled\uu+\sym(\tilde\chi\otimes\e_1)\quad\text{weakly two-scale in }L^2(\Omega\times O;\R^{3\times 3}_\sym).
\end{align*}
Hence, in view of Proposition~\ref{P:twoscale} \ref{P:twoscale:strongimpliestrong} and the special form of $\B$, cf.~\eqref{the form of B}, we get
\begin{align*}
\sym(\scaled\uu^\vare(\id+h\B(\tau_{\frac{x_1}{\vare}}\omega)))\twos\sym(\scaled\uu(\id+h\B))+\sym(\tilde\chi\otimes\e_1)
\end{align*}
weakly two-scale.
Now, Lemma~\ref{L:twoscale:quadratic} implies
\begin{equation}\label{eq:st1123223299}
  \begin{aligned}
    &\liminf_{\vare\to 0}\widehat{\E}^{\vare,h}(\omega,\uu^\vare)\\
    \geq&
  \int_{\Omega\times O} Q\bg(\omega,\sym\big(\scaled \uu(\id+h\B(\omega))\big)+\sym(\tilde\chi\otimes\e_1)\,\bg)\,d\mathbb P(\omega) dx\\
  \geq&\inf_{\chi\in L^2(O;L^2_0(\Omega;\R^3))}
  \int_{\Omega\times O} Q\bg(\omega,\sym\big(\scaled \uu(\id+h\langle\B\rangle)\big)+\sym(\chi\otimes\e_1)\,\bg)\,d\mathbb P(\omega) dx,
\end{aligned}
\end{equation}
where in the last estimate we used that
\begin{equation}\label{eq:st112322329911}
  \sym(\scaled\uu(\B-\la\B\ra))\in \{\sym(\chi\otimes\e_1)\,:\,\chi\in L^2(O;L^2_0(\Omega;\R^3))\},
\end{equation}
which holds thanks to \eqref{the form of B}.
In view of \eqref{def Qhom}, the right-hand side of \eqref{eq:st1123223299} equals $\widehat\E^{h,\hom}(\uu)$.

\step 2 {\bf Proof of the upper bound}

Choose $\tilde\chi\in L^2(O;L^2_0(\Omega;\R^3))$ such that
\begin{align*}
\widehat\E^{{\rm hom},h}(\uu)=&\int_{\Omega \times O} Q(\omega,\sym(\scaled\uu(\id+h\la \B\ra))+\sym(\tilde\chi\otimes\e_1))\,d\mathbb P(\omega) dx.
\end{align*}
In view of \eqref{eq:st112322329911} and Lemma~\ref{bound lemma} we can find a sequence $(\varphi^\vare)_\vare\subseteq C^\infty_c(O;\R^3)$ such that $\varphi^\vare$ and $\pt_{23}\varphi^\vare$ uniformly converge to $0$, and
\begin{equation*}
  \sym(\nabla_h\varphi^\vare)\twoss \sym(\tilde\chi\otimes\e_1)+h\,\sym(\scaled\uu\,(\la \B\ra-\B))\qquad\text{strongly two-scale.}
\end{equation*}
Now, consider the sequence $\uu^{\vare}:=\uu+\varphi^\vare$.
It converges to $\uu$ strongly in $L^2(O;\R^3)$, and by Lemma \ref{L:twoscale:quadratic} we have
\begin{align*}
  \lim_{\vare\to 0}\E^{\vare,h}(\omega,\uu^{\vare})=\widehat\E^{{\rm hom},h}(\uu).
\end{align*}

\step 3 {\bf Compactness and boundary conditions}

Recall the definitions of $\uu^{\vare,h}_{\rm aff}$ and $\uu^h_{\rm aff}$, see \eqref{def aff vare h} and \eqref{3D affine function def}, and note that
\begin{equation}\label{eq:conv aff}
  (\uu^{\vare,h}_{\rm aff},  \scaled\uu^{\vare,h}_{\rm aff})\to (\uu^h_{\rm aff},\scaled\uu^h_{\rm aff})\text{ uniformly as }\vare\to 0,
\end{equation}
since $\Phi_1,\Phi_2\in\sD_\Omega$.
Let $(\uu^\vare)_\vare\subset H^1(O;\R^3)$ satisfy \eqref{bc u} and $\limsup_{\vare\to 0}\widehat{\E}^{\vare,h}(\uu^\vare)<\infty$. Then $\uu^\vare-\uu^{\vare,h}_{\rm aff}=0$ for $x\in\{0,1\}\times S$ (in the sense of trace). With Proposition~\ref{prop compact} we conclude that $(\uu^\vare)_\vare$ is bounded in $H^1(O;\R^3)$. Thus, for a subsequence we have $\uu^\vare\to\uu$ weakly in $H^1(O;\R^3)$ and strongly in $L^2(O;\R^3)$. From the continuity of traces and \eqref{eq:conv aff}, we conclude that $\uu$ satisfies \eqref{eq:bc3dhom}.

Now let $\uu\in H^1(O;\R^3)$ satisfy \eqref{eq:bc3dhom}, and let $(\uu^\vare)_\vare$ denote the recovery sequence of Step~2. We consider
\begin{align*}
\uu^\vare:=\uu+(\uu^{\vare,h}_{\rm aff}-\uu^h_{\rm aff})+\varphi^\vare,
\end{align*}
where $(\varphi^\vare)_\vare\subset C^\infty_c(O;\R^3)$ is defined as in Step~2.
By construction we have $\uu^\vare=\uu^{\vare,h}_{\rm aff}$ on $\{0,1\}\times S$, and $\uu^\vare\to \uu$ uniformly. Furthermore, $\scaled\uu^\vare -\scaled(\uu+\varphi^\vare)\to 0$ uniformly. Hence, by Step 2 we conclude that  $(\uu^\vare)_\vare$ is a recovery sequence.
This completes the proof of Step 3 and also the desired proof.
\end{proof}

\subsection{Isotropic case: Proof of Proposition~\ref{precise q1D}}\label{sec precise q1D}
\begin{proof}[Proof of Proposition \ref{precise q1D}]
As in Section \ref{sec homogenization}, we simply neglect the dependence of $\mu$ and $\ld$ on $\omega$.
We aim to find $\varphi\in H^1(S;\R^3)$ such that
\begin{equation}\label{eq:form}
Q^{\rm rod}(\xxi)=\int_S Q\bg(\sym[(\xxi_1\e_1+\xxi_{234}\wedge (0,\barx))\otimes \e_1]+\sym(0,\nabla_{\barx}\varphi)\bg)\,d\barx.
\end{equation}
The Euler-Lagrange equation w.r.t.
$\varphi$ reads
\begin{align}
0&=\int_{S}4\mu\,\sym[(\xxi_1\e_1+\xxi_{234}\wedge (0,\barx))\otimes \e_1]:\sym(0,\nabla_{\barx}\bar{\varphi})\nonumber\\
&\quad+2\ld\,\trace\big((\xxi_1\e_1+\xxi_{234}\wedge (0,\barx))\otimes \e_1+(0,\nabla_{\barx}\varphi)\big)\cdot\trace\big((0,\nabla_{\barx}\bar{\varphi})\big)\,d\barx \label{euler lagrange isotropic}
\end{align}
for all $\bar{\varphi}\in H^1(S;\R^3)$.
Direct calculation yields
\begin{align}
&\sym\big((\xxi_1\e_1+\xxi_{234}\wedge (0,\barx))\otimes \e_1
+(0,\nabla_{\barx}\varphi)\big)\nonumber\\
=&\frac{1}{2}\begin{pmatrix}
2(\xxi_1+\xxi_3x_3-\xxi_4x_2)&\pt_2\varphi^1-\xxi_2 x_3&\pt_3\varphi^1+\xxi_2 x_3\\
\pt_2\varphi^1-\xxi_2 x_3&2\pt_2\varphi^2&\pt_2\varphi^3+\pt_3\varphi^2\\
\pt_3\varphi^1+\xxi_2 x_3&\pt_2\varphi^3+\pt_3\varphi^2&2\pt_3\varphi^3\\
\end{pmatrix},\label{calculation 1}\\
\nonumber\\
&\sym(0,\nabla_{\barx}\varphi)=\frac{1}{2}\begin{pmatrix}
0&\pt_2\varphi^1&\pt_3\varphi^1\\
\pt_2\varphi^1&2\pt_2\varphi^2&\pt_2\varphi^3+\pt_3\varphi^2\\
\pt_3\varphi^1&\pt_2\varphi^3+\pt_3\varphi^2&2\pt_3\varphi^3\\
\end{pmatrix}.\label{calculation 2}
\end{align}
Inserting \eqref{calculation 1} and \eqref{calculation 2} into \eqref{euler lagrange isotropic}, we obtain the following PDE:
\begin{itemize}
\item[(i)]For $\varphi^1$ we have
\begin{equation*}
\left\{
\begin{array}{ll}
-\Delta_{S}\varphi^1=0&\text{in $S$},\\
\\
(\pt_2\varphi^1,\pt_3\varphi^1)\cdot\nu=\xxi_2(x_3,-x_2)\cdot\nu&\text{on $\pt S$}.
\end{array}
\right.
\end{equation*}
We infer that $\varphi^1=\xxi_2\varphi_{\mathrm{aff}}$, where $\varphi_{\mathrm{aff}}$ is a representative solution of \eqref{non-normalized} with $\xxi_2=1$.

\item[(ii)]For $(\varphi^2,\varphi^3)$ we have
\begin{equation}
\left\{
\begin{array}{rl}
-\diver[(2\mu+\ld)\pt_2\varphi^2+\ld\pt_3\varphi^3,\mu(\pt_3\varphi^2+\pt_2\varphi^3)]=-\ld\xxi_4&\text{in $S$},\\
\nonumber\\
-\diver[\mu(\pt_3\varphi^2+\pt_2\varphi^3),(2\mu+\ld)\pt_3\varphi^3+\ld\pt_2\varphi^2]=\ld\xxi_3&\text{in $S$},\\
\nonumber\\
\bg(2\mu\pt_2\varphi^2+\ld\big(\xxi_1+\xxi_3x_3-\xxi_4x_2+(\pt_2\varphi^2+\pt_3\varphi^3)\big),\mu(\pt_2\varphi^3+\pt_3\varphi^2)\bg)\cdot\nu=0&\text{on $\pt S$},\\
\nonumber\\
\bg(\mu(\pt_2\varphi^3+\pt_3\varphi^2),2\mu\pt_3\varphi^3+\ld\big(\xxi_1+\xxi_3x_3-\xxi_4x_2+(\pt_2\varphi^2+\pt_3\varphi^3)\big)\bg)\cdot\nu=0&\text{on $\pt S$}.
\end{array}
\right.
\end{equation}
A representative solution is
\begin{align*}
\varphi^2&=-\frac{1}{4}\frac{\ld}{\ld+\mu}(2\xxi_1 x_2-\xxi_4x^2_2+\xxi_4x_3^2+2\xxi_3 x_2x_3),\\
\varphi^3&=-\frac{1}{4}\frac{\ld}{\ld+\mu}(2\xxi_1 x_3-\xxi_3x^2_2+\xxi_3x_3^2-2\xxi_4 x_2x_3).
\end{align*}
\end{itemize}
Simplifying we conclude
\begin{align}
  &\sym\big((\xxi_1\e_1+\xxi_{234}\wedge (0,\barx))\otimes \e_1
+(0,\nabla_{\barx}\varphi)\big)\\
&=\sym[(\xxi_1\e_1+\xxi_{234}\wedge (0,\barx))\otimes \e_1]+\xxi_2\sym(\pt_2\varphi_{\rm aff}\e_1\otimes \e_2)
+\xxi_2\sym(\pt_3\varphi_{\rm aff}\e_1\otimes \e_3)\nonumber\\
&\quad-\frac{1}{2}\frac{\ld}{\ld+\mu}\begin{pmatrix}
0&0\\
0&(\xxi_1-\xxi_4x_2+\xxi_3x_3)\id_2
\end{pmatrix}.\label{calculation 3}
\end{align}
Now \eqref{form of q1D isotropic} follows by inserting \eqref{calculation 3} into \eqref{eq:def of Qel}.
Finally, if $S$ is a disc, by \eqref{cancelation} it must be centered at zero and therefore $(x_3,-x_2)\cdot\nu=0$ on $S$.
In this case, we see that $\varphi_{\rm aff}=0$ is always a solution of \eqref{non-normalized}.
\end{proof}


\subsection{Quantitative homogenization: Proof of Theorem \ref{thm conv rate}}\label{sec:quant}
Without loss of generality we may assume that $L=1$. To shorten the notation we write $\E^\vare(\omega,\cdot)$ instead of $\E^{{\rm rod},\vare}(\omega,\cdot)$.
Moreover, we define $\Ab(\omega),\Ab^0\in\R^{4\times 4}_{\sym}$ via the identities
\begin{equation*}
  \Ab(\omega)\xxi\cdot\xxi:=Q^{\rm rod}(\omega,\xxi)\quad\text{and}\quad\Ab^0\xxi\cdot\xxi:=Q^{0}(\xxi)\qquad\text{for all }\xxi\in\R^4.
\end{equation*}
We set $  \Ab^\vare(\omega,s):=\Ab(\tau_{\frac{s}{\vare}}\omega)$ and
\begin{equation}\label{eq:defBb}
  \Bb^\vare(\omega,s):=(0,0,\Phi_2(\tau_{\frac{s}{\vare}}\omega),-\Phi_1(\tau_{\frac{s}{\vare}}\omega))\otimes\e_1,\qquad
  \Bb^0:=(0,0,\langle\Phi_2\rangle,-\langle\Phi_1\rangle)\otimes\e_1.
\end{equation}
With the above notation we have for all $\vv=(\bar u,\ro)\in H^1(0,1;\R^4)$,
\begin{align*}
  \E^\vare(\omega,\vv)=&\fint_0^1\Ab^\vare(\partial_s+\Bb^\vare)\vv\cdot(\partial_s+\Bb^\vare)\vv\,ds,\\
  \E^0(\vv)=&\fint_0^1\Ab^0(\partial_s+\Bb^0)\vv\cdot(\partial_s+\Bb^0)\vv\,ds.
\end{align*}
For the proof it is convenient to introduce the functional $\overline{\E}^\vare(\omega,\cdot):H^1(0,1;\R^4)\to\R$,
\begin{equation*}
  \overline{\E}^\vare(\omega,\vv):=\fint_0^1\overline\Ab^\vare(\omega)(\partial_s+\overline\Bb^\vare(\omega))\vv\cdot(\partial_s+\overline\Bb^\vare(\omega))\vv\,ds.\label{eq:proxy}
\end{equation*}
Here, $\overline\Ab^\vare(\omega),\overline\Bb^\vare(\omega)\in\R^{4\times 4}_{\sym}$ are representative volume elemenet (RVE) approximations of the homogenized coefficients $\A^0$ and $\B^0$. They are defined as follows:
\begin{align*}
  \overline\Bb^\vare(\omega):=&\fint_0^1\Bb^\vare(\omega,s)\,ds,\\
  \overline\Ab^\vare(\omega)\xxi:=&\fint_0^1\Ab^\vare(\omega,s)(\xxi+\partial_s\phi_{\xxi}^\vare(\omega,s))\,ds,\qquad\text{for all }\xxi\in\R^4,
\end{align*}
where $\phi_{\xxi}^\vare(\omega,\cdot)\in H^1_0(0,1;\R^4)$ is the unique, weak solution to
\begin{equation*}
-\partial_s\Big(\Ab^\vare(\omega,s)(\xxi+\partial_s\phi_{\xxi}^\vare(\omega,s))\Big)=0\qquad\text{in }(0,1),
\end{equation*}
and has the meaning of a Dirichlet corrector.
Note that by construction we have
\begin{equation*}
  \overline\Ab^\vare(\omega)\xxi\cdot\xxi=\min_{\varphi\in H^1_0(0,L;\R^4)}\fint_0^1\Ab^\vare(\omega,s)(\xxi+\partial_s\varphi(s))
  \cdot (\xxi+\partial_s\varphi(s))\,ds
\end{equation*}
and
\begin{equation}\label{eqdefbeta}
  \beta_1|\xxi|^2\leq\left\{
  \begin{aligned}
    \Ab^\vare(\omega)\xxi\cdot \xxi\\
    \overline\Ab^\vare(\omega)\xxi\cdot \xxi\\
    \Ab^0\xxi\cdot \xxi
  \end{aligned}\right\}
  \leq \beta_2|\xxi|^2\qquad\text{for all }\xxi\in\R^4
\end{equation}
with some constants $0<\beta_1\leq\beta_2$ that only depend on $\alpha_1,\alpha_2$ and $S$, cf. Lemma~\ref{L:quadr}.
Since we are in the one-dimensional case, a direct calculation yields
\begin{align}\label{harmonic_mean}
  \overline\Ab^\vare(\omega)=&\left(\fint_0^1(\Ab^\vare(\omega,s))^{-1}\,ds\right)^{-1},\\
  \label{eq:dircorr}
  \phi^\vare_\xxi(\omega,\cdot)=&\Phib^\vare(\omega,s)\xxi,\qquad \Phib^\vare(\omega,s):=\int_0^s\big((\Ab^\vare(\omega,s))^{-1}\overline\Ab^\vare(\omega)-\id\big)\,ds,
\end{align}
where $\id$ denotes the idenity matrix in $\R^{4\times 4}$. In view of the explicit formulas for $\overline\Ab^\vare$ and $\overline\Bb^\vare$, we see that ergodicity directly implies
\begin{equation}\label{eq:convergence}
  \overline\Ab^\vare(\omega)\to \Ab^0\text{ and }\overline\Bb^\vare(\omega)\to\Bb^0\text{ for $\mathbb P$-a.a.~$\omega\in\Omega$}.
\end{equation}
This means that $\overline\E^\vare$ is an approximation of $\E^0$ (in a sense that can be made precise via $\Gamma$-convergence).
For the upcoming argument it is usefull to represent $\Bb^\vare(\omega,\cdot)-\overline\Bb^\vare(\omega)$ with help of an auxiliary corrector
\begin{equation}\label{eq:auxcorr}
  \Psib^\vare(\omega,\cdot):(0,1)\to\R^4,\qquad\Psib^\vare(s):=\int_0^s\Bb^\vare(t)-\overline\Bb^\vare\,dt.
\end{equation}
\medskip

From now on we drop the dependence on $\omega$ in our notation if there is no danger of confusion. We tacitly assume that $\omega$ is chosen such that the maps $s\mapsto\Ab(\tau_s\omega)$ and $s\mapsto \Bb(\tau_s\omega)$ are measurable on $\R$, and that \eqref{eq:convergence} holds; note that this is true $\mathbb P$-a.s.
\medskip

To conveniently describe the boundary conditions we set $\vv_{\rm bd}=(\bar u_{\rm aff},\ro_{\rm aff})$ (cf.~\ref{affine function def}), and note that
\begin{equation}\label{estim:vvbd}
  \|\vv_{\rm bd}\|_{W^{1,\infty}(0,1)}\leq 2({\bf{t}}_1|+|k_0|+|k_L|).
\end{equation}
We seek minimizers in the space $\vv_{\rm bd}+\mathcal{H}$, where
\begin{equation}\label{defH}
  \mathcal{H}=H^1_0(0,1)\times H^1_{00}(0,1;\R^3).
\end{equation}
More precisely, let $\vv^\vare$, $\bar\vv^\vare$ and $\vv^0$ denote the minimizers in $\vv_{\rm bd}+\mathcal{H}$ of $\E^\vare$, $\overline\E^\vare$ and $\E^0$, respectively. Now, the idea of the proof is to split the estimate for $|\E^\vare(\vv^\vare)-\E^0(\vv^0)|$ into two parts:
\begin{equation}\label{error:decomp}
  |\E^\vare(\vv^\vare)-\E^0(\vv^0)|\leq |\E^\vare(\vv^\vare)-\overline\E^\vare(\bar\vv^\vare)|+|\overline\E^\vare(\bar\vv^\vare)-\E^0(\vv^0)|.
\end{equation}
To estimate the first term on the right-hand side, we appeal to a two-scale expansion $\hat\vv^\vare$ for the minimizer of $\overline\E^\vare$. It invokes the Dirichlet corrector $\Phib^\vare$ and leads to an error of the order $\sqrt{\vare}$, which is mainly due to the scaling of the Dirichlet corrector. As a side product we also obtain an $H^1$-estimate for the error of the two-scale expansion.
To estimate the second term on the right-hand side of \eqref{error:decomp} we  quantify the rate of convergence in \eqref{eq:convergence} with help of the spectral gap assumption. This error scales as $\sqrt{\vare}$ and is determined by the speed of convergence of spatial averages.

Before we present the proof of Theorem~\ref{thm conv rate} in detail, we state estimates on the rate of convergence of spatial averages. These results determine the scaling w.r.t.~$\vare$ and are the only places where the spectral gap assumption is used. The proofs of the following three results are postponed to the end of this section.

\begin{lemma}[Rate of convergence of spatial averages]\label{L:scaling1}
  Let $F\in L^1(\Omega)$ be a 1-local Lipschitz random variable in the sense of \eqref{locallip}. For $\ell>0$ consider the random variable
  \begin{equation*}
    \mathscr G_\ell(\omega):=\fint_0^\ell F(\tau_t\omega)\,dt-\mathbb E[F].
  \end{equation*}
  Then the spectral estimate with constant $\rho$ of Assumption~\ref{sg assumption} implies that
  \begin{equation*}
    |\mathscr G_\ell|\leq \mathscr C_\ell\,C_{F}(\ell+1)^{-\frac12},
  \end{equation*}
  where $\mathscr C_\ell$ denotes a random variable satisfying
  \begin{equation*}
    \mathbb E\Big[\exp\Big(\frac{\mathscr C_\ell}{C}\Big)\Big]\leq 2,
  \end{equation*}
  for some $C$ only depending on $\rho$.
\end{lemma}

\begin{corollary}[Rate of convergence of RVE-approximation]\label{C:RVE}
  We have
  \begin{equation*}
    |\overline\A^\vare-\A^0|+    |\overline\B^\vare-\B^0|\leq \mathscr C(C_{\A}+C_{\B})\,\sqrt\vare
  \end{equation*}
  where $\mathscr C$ denotes a random variable satisfying
  \begin{equation*}
    \mathbb E\Big[\exp\Big(\frac{\mathscr C}{C}\Big)\Big]\leq 2,
  \end{equation*}
  for some $C$ only depending on $\rho$, $\beta_1$ and $\beta_2$ (cf.~\eqref{eqdefbeta}).
\end{corollary}

\begin{corollary}[Scaling of the correctors]\label{scaling2}
  Let $\Phib^\vare$ and $\Psib^\vare$ be defined by \eqref{eq:dircorr} and \eqref{eq:auxcorr}, respectively.
  Then $\Phib^\vare,\Psib^\vare\in W^{1,\infty}_0((0,1);\R^4)$ with
\begin{equation*}
  \|(\Phib^\vare,\Psib^\vare)\|_{W^{1,\infty}(0,1)}\leq C,
\end{equation*}
where $C$ only depends on $\beta_1$ and $\beta_2$. Moreover, we have
\begin{equation*}
  \left(\fint_0^1|(\Psib^\vare(s),\Phib^\vare(s))|^2\,ds\right)^\frac12\leq \mathscr C (C_{\A}+C_{\B})\sqrt{\vare}\qquad\mathbb P\text{-a.s.},
\end{equation*}
where $\mathscr C$ denotes a random variable satisfying
\begin{equation*}
  \mathbb E\Big[\exp\Big(\frac{\mathscr C}{C}\Big)\Big]\leq 2,
\end{equation*}
where $C$ only depends on $\rho$, $\beta_1$ and $\beta_2$.
\end{corollary}

We are now ready to prove Theorem~\ref{thm conv rate}.
\begin{proof}[Proof of Theorem~\ref{thm conv rate}]
In the following we write $a\lesssim b$, if $a\leq Cb$ for a constant $C$ that only depends on $\rho$, $\alpha_1$ and $\alpha_2$. We split our proof into six steps.
\step 1 {\bf Identification of the minimizers $\bar\vv^\vare$ and $\vv^0$}

We claim that
\begin{equation*}
  \vv^0=\bar\vv^\vare=\vv_{\rm bd}.
\end{equation*}
For the argument note that we have
\begin{equation}\label{eq:specialproperty}
  \Bb^0\vv_{\rm bd}=  \Bb^\vare\vv_{\rm bd}=0\quad\text{and}\quad\Psib^\vare\vv_{\rm bd}=0,
\end{equation}
which follows from the structural properties of $\Bb^\vare$ (cf.~\eqref{eq:defBb}) and $\vv_{\rm bd}$ (cf.~\eqref{affine function def}), and the definition of $\Psib^\vare$.
Furthermore, since $\vv_{\rm bd}$ is affine, we have for all $\varphi\in \mathcal{H}$ (cf.~\eqref{defH}),
\begin{equation*}
  \fint_0^1\Ab^0(\partial_s+\Bb^0)\vv_{\rm bd}\cdot(\partial_s+\Bb^0)\varphi\,ds=
  \Ab^0\partial_s\vv_{\rm bd}\cdot  \fint_0^1(\partial_s+\Bb^0)\varphi\,ds=0.
\end{equation*}
We conclude that $\vv_{\rm bd}$ minimizes $\E^0$ and thus $\vv^0=\vv_{\rm bd}$. By the same argument we see that $\bar\vv^\vare=\vv_{\rm bd}$.

\step 2 {\bf Definition of the two-scale expansion}

Define
\begin{equation*}
  \hat\vv^\vare:=\vv_{\rm bd}+\Phib^\vare\partial_s\vv_{\rm bd}-\eta c^\vare,
\end{equation*}
where $\eta(s):=s(1-s)/(\int_0^1t(1-t)\,dt)$ and $c^\vare\in\R^4$ is chosen such that $\fint_0^L(\hat\vv^\vare-\vv_{\rm bd})_{34}=0$.
Then it is easy to check that $\hat\vv^\vare\in \vv_{\rm bd}+\mathcal{H}$. Moreover, a direct calculation that exploits \eqref{eq:dircorr} and \eqref{eq:specialproperty} shows that
\begin{equation}\label{flux}
  \Ab^\vare(\partial_s+\Bb^\vare)\hat\vv^\vare= \overline\Ab^\vare(\partial_s+\overline\Bb^\vare)\vv_{\rm bd}+\rho^\vare,\qquad \rho^\vare=\Ab^\vare\Bb^\vare(\Phib^\vare\partial_s\vv_{\rm bd}-\eta c^\vare)-\Ab^\vare\partial_s\eta c^\vare.
\end{equation}
We note that
\begin{equation}\label{estimrho}
  \|\rho^\vare\|_{L^2(0,1)}\lesssim \|\Phib^\vare\|_{L^2(0,1)}\|\partial_s\vv_{\rm bd}\|_{L^\infty(0,1)}.
\end{equation}

\step 3 {\bf Estimate of the first two-scale expansion error}

We claim that
\begin{equation}\label{eq:two-scale-expansion}
  \|\partial_s(\vv^\vare-\hat\vv^\vare)\|_{L^2(0,1)}^2\lesssim |\E^\vare(\hat\vv^\vare)-\E^\vare(\vv^\vare)|\lesssim \|\partial_s\vv_{\rm bd}\|^2_{L^\infty(0,1)}\|(\Phib^\vare,\Psib^\vare)\|_{L^2(0,1)}^2.
\end{equation}
For the argument note that we have for all $\varphi\in \mathcal{H}$,
\begin{equation*}
  \int\Ab^\vare(\partial_s+\Bb^\vare)\vv^\vare\cdot(\partial_s+\Bb^\vare)\varphi=0,\qquad
  \int\overline\Ab^\vare(\partial_s+\overline\Bb^\vare)\vv_{\rm bd}\cdot(\partial_s+\overline\Bb^\vare)\varphi=0,
\end{equation*}
since $\vv^\vare$ and $\vv_{\rm bd}$ are minimizers of $\E^\vare$ and $\overline\E^\vare$.
By expanding squares we thus deduce that
\begin{align*}
  &\E^\vare(\hat\vv^\vare-\vv^\vare)=  \E^\vare(\hat\vv^\vare)-\E^\vare(\vv^\vare)\\
  =&\int \Ab^\vare(\partial_s+\Bb^\vare)\hat\vv^\vare\cdot(\partial_s+\Bb^\vare)\hat\vv^\vare-\int\Ab^\vare(\partial_s+\Bb^\vare)\vv^\vare\cdot(\partial_s+\Bb^\vare)\vv^\vare\\
  =&\int\Ab^\vare(\partial_s+\Bb^\vare)\hat\vv^\vare\cdot (\partial_s+\Bb^\vare)(\hat\vv^\vare-\vv^\vare)+\underbrace{\int \Ab^\vare(\partial_s+\Bb^\vare)\vv^\vare\cdot(\partial_s+\Bb^\vare)(\hat\vv^\vare-\vv^\vare)}_{=0}\\
  =&\int\overline\Ab^\vare(\partial_s+\overline\B^\vare)\vv_{\rm bd}\cdot (\partial_s+\Bb^\vare)(\hat\vv^\vare-\vv^\vare)+\int\rho^\vare\cdot (\partial_s+\Bb^\vare)(\hat\vv^\vare-\vv^\vare),
\end{align*}
where the last identity holds thanks to Step~2. By \eqref{eq:specialproperty} and the fact that $\overline\Ab^\vare\partial_s\vv_{\rm bd}$ is a constant vector, we have
\begin{align*}
  &\int\overline\Ab^\vare(\partial_s+\overline\Bb^\vare)\vv_{\rm bd}\cdot (\partial_s+\Bb^\vare)(\hat\vv^\vare-\vv^\vare) = \int\overline\Ab^\vare\partial_s \vv_{\rm bd}\cdot \partial_s\Psib^\vare(\hat\vv^\vare-\vv^\vare) =-\int\overline\Ab^\vare\partial_s \vv_{\rm bd}\cdot \Psib^\vare\partial_s(\hat\vv^\vare-\vv^\vare)
\end{align*}
From the previous two estimates and \eqref{estimrho} we deduce that
\begin{align*}
  &\,\|\partial_s(\vv^\vare-\hat\vv^\vare)\|_{L^2(0,1)}^2\nonumber\\
  \lesssim&\,  \|(\partial_s+\Bb^\vare)(\vv^\vare-\hat\vv^\vare)\|_{L^2(0,1)}^2
  \lesssim\E^\vare(\hat\vv^\vare-\vv^\vare)=\E^\vare(\hat\vv^\vare)-\E^\vare(\vv^\vare)\\
\lesssim &\, \|\Phib^\vare\|_{L^2(0,1)}\|\partial_s\vv_{\rm bd}\|_{L^\infty(0,1)}  \|(\partial_s+\Bb^\vare)(\vv^\vare-\hat\vv^\vare)\|_{L^2(0,1)}\nonumber\\
&\quad+\|\partial_s\vv_{\rm bd}\|_{L^\infty(0,1)}\|\Psib^\vare\|_{L^2(0,1)}  \|\partial_s(\vv^\vare-\hat\vv^\vare)\|_{L^2(0,1)}.
\end{align*}
and thus \eqref{eq:two-scale-expansion} follows.

\step 4 {\bf Estimate of the second two-scale expansion error}

We claim that
\begin{equation}\label{eq:two-scale-expansion2}
  |\E^\vare(\hat\vv^\vare)-\overline\E^\vare(\vv_{\rm bd})|\lesssim
  \|(\Phib^\vare,\Psib^\vare)\|_{L^2(0,1)}\|\vv_{\rm bd}\|_{W^{1,\infty}(0,1)}^2.
\end{equation}
Indeed, thanks to \eqref{flux}, we have
\begin{align*}
  &\E^\vare(\hat\vv^\vare)-\overline\E^\vare(\vv_{\rm bd})\\
  =&\int\overline\Ab^\vare(\partial_s+\overline\Bb^\vare)\vv_{\rm bd}\cdot(\partial_s+\Bb^\vare)\hat\vv^\vare
     +\int\rho^\vare\cdot(\partial_s+\Bb^\vare)\hat\vv^\vare
     -\int\overline\Ab^\vare(\partial_s+\overline\Bb^\vare)\vv_{\rm bd}\cdot (\partial_s+\overline\Bb^\vare)\vv_{\rm bd}\\
  =&\underbrace{\int\overline\Ab^\vare(\partial_s+\overline\Bb^\vare)\vv_{\rm bd}\cdot(\partial_s+\overline\Bb^\vare)(\hat\vv^\vare-\vv_{\rm bd})}_{=0}
     +\int\overline\Ab^\vare(\partial_s+\overline\Bb^\vare)\vv_{\rm bd}\cdot\partial_s\Psib^\vare(\hat\vv^\vare-\vv_{\rm bd})
  +\int\rho^\vare\cdot(\partial_s+\Bb^\vare)\hat\vv^\vare.
\end{align*}
Note that
\begin{align*}
  \int\overline\Ab^\vare(\partial_s+\overline\Bb^\vare)\vv_{\rm bd}\cdot\partial_s\Psib^\vare(\hat\vv^\vare-\vv_{\rm bd})=-\int\overline\Ab^\vare\partial_s\vv_{\rm bd}\cdot\Psib^\vare\partial_s\Phib^\vare\partial_s\vv_{\rm bd}
\end{align*}
We conclude that
\begin{equation*}
  |\E^\vare(\hat\vv^\vare)-\overline\E^\vare(\vv_{\rm bd})|\lesssim
  \|(\Phib^\vare,\Psib^\vare)\|_{L^2(0,1)}\|\vv_{\rm bd}\|_{W^{1,\infty}(0,1)}^2.
\end{equation*}

\step 5 {\bf Proof of \eqref{conv rate energy} and \eqref{conv rate function}}

Note that
\begin{equation*}
  |\overline\E^\vare(\vv_{\rm bd})-\E^0(\vv_{\rm bd})|\leq |\overline\Ab^\vare-\Ab^0|\|\partial_s\vv_{\rm bd}\|_{L^\infty(0,1)}^2.
\end{equation*}
Since $\vv^0=\bar\vv^\vare=\vv_{\rm bd}$, we conclude from the previous steps, Corollaries~\ref{C:RVE} and \ref{scaling2} that

\begin{align*}
  |\E^\vare(\vv^\vare)-\E^0(\vv^0)|&\leq   |\E^\vare(\vv^\vare)-\E^\vare(\hat\vv^\vare)|+|\E^\vare(\hat\vv^\vare)-\overline{\E}^\vare(\vv_{\rm bd})|+|\overline{\E}^\vare(\vv_{\rm bd})-\E^0(\vv_{\rm bd})|\\
  &\leq\Big(\mathscr C^2(C_{\Ab}+C_{\Bb})^2\vare+2\mathscr C(C_{\Ab}+C_{\Bb})\sqrt{\vare}\Big)\|\vv_{\rm bd}\|_{W^{1,\infty}(0,1)}^2.
\end{align*}
Combined with \eqref{estim:vvbd}, the claimed estimate \eqref{conv rate energy} follows.

To prove \eqref{conv rate function}, we first note that
\begin{equation*}
  \|\hat\vv^\vare-\vv_{\rm bd}\|_{L^2(0,1)}\leq\|\Phib^\vare\partial_s\vv_{\rm bd}-\eta c^\vare\|_{L^2(0,1)}\lesssim \|\partial_s\vv_{\rm bd}\|_{L^\infty(0,1)}\|\Phib^\vare\|_{L^2(0,1)}.
\end{equation*}
On the other hand, by \eqref{eq:two-scale-expansion}
\begin{equation*}
  \|\vv^\vare-\hat\vv^\vare\|_{L^2(0,1)}\leq\|\partial_s(\vv^\vare-\hat\vv^\vare)\|_{L^2(0,1)}\leq \|\partial_s\vv_{\rm bd}\|_{L^\infty(0,1)}\|(\Phib^\vare,\Psib^\vare)\|_{L^2(0,1)}
\end{equation*}
and thus by the triangle inequality and Corollary~\ref{scaling2},
\begin{equation*}
  \|\vv^\vare-\vv_{\rm bd}\|_{L^2}\leq \|\partial_s\vv_{\rm bd}\|_{L^\infty}\|(\Phib^\vare,\Psib^\vare)\|_{L^2}\lesssim \mathscr C(C_{\Ab}+C_{\Bb})\sqrt{\vare}\|\partial_s\vv_{\rm bd}\|_{L^\infty},
\end{equation*}
as claimed.

\step 6 {\bf The constant coefficient case}

Assume the $Q^{\rm rod}$ is independent of $\omega$. Then $\Ab^\vare=\bar\Ab^\vare=\Ab^0$ and $\Phib^\vare=0$. As a consequence, the two-scale expansion simplifies and we conclude that $\hat\vv^\vare=\vv_{\rm bd}$.
In view of \eqref{eq:specialproperty} we further have
\begin{align*}
  &\,\E^\vare(\hat\vv^\vare)=  \E^\vare(\vv_{\rm bd})\nonumber\\
  =&\,\int\Ab^0(\partial_s+\Bb^\vare)\vv_{\rm bd}\cdot(\partial_s+\Bb^\vare)\vv_{\rm bd}\,ds\nonumber\\
  =&\,\int\Ab^0(\partial_s+\Bb^0)\vv_{\rm bd}\cdot(\partial_s+\Bb^0)\vv_{\rm bd}\,ds=\E^0(\vv_{\rm bd}).
\end{align*}
We thus conclude from Step~3 that
\begin{equation*}
  |\E^\vare(\vv^\vare)-\E^0(\vv^0)|=  |\E^\vare(\vv^\vare)-\E^\vare(\hat\vv^\vare)|\lesssim\mathscr C^2 C_{\Bb}^2\vare\|\partial_s\vv_{\rm bd}\|_{L^\infty(0,1)}^2,
\end{equation*}
and the desired estimate follows.
\end{proof}

It remains to prove Lemma~\ref{L:scaling1}, Corollaries~\ref{C:RVE} and \ref{scaling2}.
To that end, we also need the following $p$-version of the spectral gap estimate. We refer to \cite{pSGDuerGloria} for a proof.

\begin{lemma}[$p$-spectral gap, \cite{pSGDuerGloria}]\label{lem p sg}
Suppose that the probability space $(\Omega,\mathbb{P},\mathcal{F})$ satisfies Assumption \ref{sg assumption}. Then there exists some $C=C(\rho)>0$ such that for any random variable $F : \Omega\to\R$ and all $p\in[1,\infty)$ we have
\begin{align}
\mathbb{E}\big[ |F-\mathbb{E}[F]|^{2p}\big]^{\frac{1}{2p}}\leq Cp\,\mathbb{E}\bg[\,\bg|\int_{\R}\bg(\int_{s-1}^{s+1}\bg|\frac{\partial F}{\partial \omega}\bg|\,\bg)^{2}\,ds\bg|^p\bg]^{\frac{1}{2p}}.
\end{align}
\end{lemma}

\begin{proof}[Proof of Lemma~\ref{L:scaling1}]
  It suffices to show that that there exists a constant $c'$ only depending on $\rho$ such that for all $\ell>0$ and $p\geq 1$ we have
 \begin{equation*}
   \mathbb E[|\mathscr G^{2p}_\ell|]^{\frac{1}{2p}}\leq c'C_{F}p(\ell+1)^{-\frac12}.
 \end{equation*}
 Since $F$ is a 1-local Lipschitz random field, we find that for any perturbation $\delta\omega$ with support in $s+(-1,1)$ and $|\delta\omega|\leq 1$, we have
  \begin{equation*}
    |F(\tau_t(\omega+\delta\omega))-F(\tau_t\omega)|\leq C_{F}\|\delta\omega\|_{L^\infty(t-1,t+1)}\leq C_{F}\mathbf 1(\{t\in(s-2,s+2)\}).
  \end{equation*}
  Thus
  \begin{equation*}
    |\mathscr G_\ell(\omega+\delta\omega)-\mathscr G_\ell(\omega)|\leq C_F\ell^{-1}|(0,\ell)\cap(s-2,s+2)|\leq C_F\ell^{-1}
    \begin{cases}
      \min\{4,\ell\}&\text{if }s\in(-2,\ell+2)\\
      0&\text{else.}
    \end{cases}
  \end{equation*}
  We conclude that for some universal constant $c\geq 1$ and $\mathbb P$-a.a.~$\omega\in\Omega$ we have
  \begin{equation*}
    \int_\R(\int_{s-1}^{s+1}|\frac{\partial\mathscr G_\ell(\omega)}{\partial\omega}|\,dt)^2\,ds\leq cC_{F}^2\frac{1}{\ell+1}.
  \end{equation*}
  Hence, the claim follows from the $p$-version of the spectral gap estimate, cf.~Lemma~\ref{lem p sg}.
\end{proof}

\begin{proof}[Proof of Corollary~\ref{C:RVE}]
  The argument for $\overline{\B}^\vare$ is immediate. For $\overline\A^\vare$ consider the random variable
  \begin{equation*}
    \mathscr A_\ell(\omega):=\fint_0^\ell\A^{-1}(\tau_t\omega)\,dt-\mathbb E[\A^{-1}].
  \end{equation*}
  Since  $\A^{-1}(\omega)-\A^{-1}(\omega')=\A(\omega)^{-1}(\A(\omega')-\A(\omega))\A(\omega')^{-1}$, we have
  \begin{equation*}
    |\A^{-1}(\omega)-\A^{-1}(\omega')|\leq \frac{C_{\A}}{\beta_2^2}\|\omega'-\omega\|_{L^\infty(-1,1)},
  \end{equation*}
  and thus $\A^{-1}$ is 1-local and Lipschitz. With Lemma~\ref{L:scaling1} we conclude that
  \begin{equation*}
    |\mathscr A_\ell|\leq \mathscr C\frac{C_{\A}}{\beta_1^2}(\ell+1)^{-\frac12}.
  \end{equation*}
  By a direct calculation we have $\overline\A^\vare-\A^0=-\A^0\mathscr A_{L/\vare}\overline\A^\vare$, and thus the claimed bound follows.
\end{proof}

\begin{proof}[Proof of Corollary~\ref{scaling2}]
We split the proof into two steps.
  \step 1 {\bf Estimate of  $\Psib^\vare$}
  For $\ell$ consider the mean free random variable
  \begin{equation*}
    \mathscr G_\ell(\omega):=\fint_0^\ell\B(\tau_t\omega)\,dt-\mathbb E[\B].
  \end{equation*}
  Note that
  \begin{equation*}
    \fint_0^1 |\Psib^\vare(s)|^2\,ds=\fint_0^1 s^2|\mathscr G_{s/\vare}-\mathscr G_{1/\vare}|^2\,ds.
  \end{equation*}
  Hence, by Jensen's inequality and Lemma~\ref{L:scaling1}, there exists $c'$ (only depending on the constant $\rho$ of the spectral gap inequality) such that for any $p\geq 1$,
  \begin{align*}
    \mathbb E\Big[\big(\fint_0^1 |\Psib^\vare(s)|^2\,ds\big)^{p}\Big]^{\frac1{p}}
    &\leq 2\fint_0^1 s^2\mathbb E[|\mathscr G_{s/\vare}|^{2p}]^{1/p}+s^2\mathbb E[|\mathscr G_{1/\vare}|^{2p}]^{1/p}\,ds\\
    &\leq 2\fint_0^1 s^2\mathbb E[|\mathscr G_{s/\vare}|^{2p}]^{1/p}\,ds+\frac{2}{3}\mathbb E[|\mathscr G_{1/\vare}|^{2p}]^{1/p}\\
    &\leq 2c'^2C_{\B}^2p^2\Big(\fint_0^1 s^2(s/\vare+1)^{-1}\,ds+\frac{1}{3}(\frac{1}{\vare}+1)^{-1}\Big)\\
    &\leq 2c'^2C_{\B}^2p^2\,\vare,
  \end{align*}
  and thus the exponential moment bound for $\|\Psib^\vare\|_{L^2(0,1)}$ follows.

  \step 2 {\bf Estimate of  $\Phib^\vare$}

  Consider
  \begin{equation*}
    \mathscr A_\ell(\omega):=\fint_0^\ell\A^{-1}(\tau_t\omega)\,dt-\mathbb E[\A^{-1}].
  \end{equation*}
  Since  $\A^{-1}(\omega)-\A^{-1}(\omega')=\A(\omega)^{-1}(\A(\omega')-\A(\omega))\A(\omega')^{-1}$, we have
  \begin{equation*}
    |\A^{-1}(\omega)-\A^{-1}(\omega')|\leq \frac{C_{\A}}{\beta_2^2}\|\omega'-\omega\|_{L^\infty(-1,1)},
  \end{equation*}
  and thus the argument of Step~1 applied to $\mathscr A_{\ell}$ yields
   \begin{equation*}
    \mathbb E[|\mathscr A^{2p}_\ell|]^{\frac{1}{2p}}\leq \frac{c'}{\beta^2_2}C_{\A}p(\ell+1)^{-\frac12}.
  \end{equation*}
  Note that
  \begin{align*}
    \fint_0^s\partial_t\Phib^\vare(t)\,dt=&\big(\fint_0^s(\A^\vare)^{-1}\,dt-\fint_0^1(\A^{\vare})^{-1}\,dt\big)\overline\A^\vare\\
    =&(\mathscr A_{s/\vare}-\mathscr A_{1/\vare})\overline\A^\vare.
  \end{align*}
  Therefore, we conclude that
  \begin{align*}
    \int |\Phib^\vare|^2\,ds=    \int s^2\left|\fint_0^s\partial_t\Phib^\vare\,dt\right|^2\,ds\\
    \leq |\overline\A^\vare|^2   \int s^2|\mathscr A_{s/\vare}-\mathscr A_{1/\vare}|^2)\,ds.
  \end{align*}
  Since $|\overline\A^\vare|$ is bounded by a constant only depending on $\beta_1,\beta_2$, the claimed estimate follows as in Step~2.
\end{proof}


\section*{Acknowledgements}
We gratefully acknowledge the support of the Deutsche Forschungsgemeinschaft (German Research Foundation) as part of the Priority Program SPP1886 ``Polymorphic uncertainty modelling for the numerical design of structures'' under project 428470437.

\appendix

\section{Appendix: Stochastic two-scale convergence}\label{appendix: two scale conv}
In this section, we recall the concept of stochastic two-scale convergence in the quenched sense as introduced and discussed in~\cite{Zhikov2006,heida2011extension,HNV22}. We present the notion in a form adapted to our needs, namely, for homogenization problems with coefficients that only feature random oscillations in the $x_1$-direction. In the literature, there are various, slightly different notions of stochastic two-scale convergence. In the following, we give a self-contained introduction closely following \cite{HNV22}.
\bigskip

Throughout this section, we assume that $(\Omega,\mathcal F,\mathbb P,\tau)$ satisfies Assumption~\ref{assumption of prob}. Moreover, we assume that $O\subset\R^d$, $d\geq 1$ is an open and bounded Lipschitz domain. As in the periodic case, stochastic two-scale convergence is based on oscillatory test-functions. In the stochastic case the construction of the oscillatory test-functions invokes the stationary extension:
\begin{lemma}[Stationary extension, {see \cite[Lemma~2.2]{HNV22}}]\label{L:stat}
  Let $\varphi:\Omega\to\R$ be $\mathcal F$-measurable. Let $I\subset \R$ be open and denote by $\mathcal{L}(I)$ the corresponding Lebesgue $\sigma$-algebra. Then $S\varphi:\Omega\times I\to\R$, $S\varphi(\omega,x_1):=\varphi(\tau_{x_1}\omega)$ defines a $\mathcal F\otimes\mathcal L(I)$-measurable function -- called the stationary extension of $\varphi$. Moreover, if $I$ is bounded, then for all $1\leq p<\infty$ the map $S:L^p(\Omega)\to L^p(\Omega\times I)$ is a linear injection satisfying
  \begin{equation*}
    \|S\varphi\|_{L^p(\Omega\times I)}=|I|^\frac{1}{p}\|\varphi\|_{L^p(\Omega)}.
  \end{equation*}
\end{lemma}
Another key ingredient of the quenched stochastic two-scale convergence is  Birkhoff's ergodic theorem:
\begin{theorem}[{Birkhoff's ergodic theorem \cite[Theorem 10.2.II]{Daley1988}}]
\label{thm:ergodic-thm} Let Assumption~\ref{assumption of prob} be satisfied and let $\varphi\in L^1(\Omega)$. Then the following holds for $\mathbb P$-a.a.~$\omega\in\Omega$: $ S\varphi(\omega,\cdot)$ is locally integrable and for all open, bounded intervals $I\subset\R$ we have
\begin{equation}
  \lim_{\vare\rightarrow0}\int_{I}S\varphi(\omega,\tfrac{x}{\vare})\,dx=|I|\int_\Omega\varphi\,d\mathbb P(\omega)\,.\label{eq:ergodic-thm}
\end{equation}
\end{theorem}
As a rather direct consequence of Theorem \ref{thm:ergodic-thm} we obtain:
\begin{corollary}\label{C:weak11}
  Let $\omega_0\in\Omega$ and $\varphi:\Omega\to\R$ be measurable and essentially bounded. Assume that for the given sample $\omega_0$ and function $\varphi$, \eqref{eq:ergodic-thm} holds for any open, bounded interval $I\subset\R$. Then for any open, bounded set $O\subset\R^d$ and $u\in L^1(O)$, we have
  \begin{equation}\label{eq:ergodic-thm2}
    \lim\limits_{\vare\to 0}\int_O u(x)S\varphi(\omega_0,\tfrac{x_1}{\vare})\,dx\to \int_\Omega\int_O u(x)\varphi(\omega)\,dx\,d\mathbb P(\omega).
  \end{equation}
\end{corollary}
Another ingredient that we need, in particular for analyzing the two-scale limits of gradients, is the \textit{stochastic derivative}. To that end we note that $\{U_{x_1}\}_{x_1\in\R}$, $U_{x_1}:L^2(\Omega)\to L^2(\Omega)$, $U_{x_1}\varphi(\omega):=\varphi(\tau_{x_1}\omega)$ defines a strongly continuous group of unitary operators. We denote by ${\mathscr H^1}(\Omega)$ the space of functions $\varphi\in L^2(\Omega)$ for which the limit
\begin{align}\label{limit in L2}
  \pt_{\omega}\varphi(\omega):=\lim_{h\to 0}\frac{\varphi(\tau_{h}\omega)-\varphi(\omega)}{h}
\end{align}
exists in $L^2(\Omega)$. For $\varphi\in{\mathscr H^1}(\Omega)$ we call $\pt_\omega\varphi$ the stochastic derivative of $\varphi$. Note that $\pt_\omega$ is the generator of the group $\{U_{x_1}\}_{x_1\in\R}$. It is a closed operator, and thus ${\mathscr H^1}(\Omega)$ with the norm
\begin{equation*}
  \|\varphi\|_{{\mathscr H^1}(\Omega)}:=\left(\int_\Omega |\varphi|^2+|\partial_\omega\varphi|^2\,d\mathbb{P}(\omega)\right)^\frac12
\end{equation*}
is a Hilbert space. By ergodicity we have (e.g.~see \cite{PapaVaradhan})
\begin{equation*}
  L^2_0(\Omega)=\operatorname{closure}\{\partial_\omega\varphi\,:\,\varphi\in{\mathscr H^1}(\Omega)\},
\end{equation*}
where the closure is taken in $L^2(\Omega)$, and $L^2_0(\Omega)$ denotes the space of functions in $L^2(\Omega)$ with mean zero. Note that we have this simple characterization, since we are in the one-dimensional case (i.e., $\{U_{x_1}\}_{x_1\in\R}$ is a one-parameter semigroup).
\medskip

\paragraph{Definition of stochastic two-scale convergence and two-scale test-functions.}
For the definition of two-scale convergence we need to specify a set of test-functions $\sD$ that is dense in $L^2(\Omega\times O)$ and that is the span of a countable set.
We use the countability to fix a common set $\Omega_0$ with $\mathbb P(\Omega_0)=1$ of samples $\omega_0$, for which the two-scale convergence and compactness results apply.

\begin{remark}
\normalfont
In the special case where $\Omega$ is a compact metric space, different constructions are possible and the space of test-functions can be extended.
\end{remark}

As we shall see, it is convenient to consider random variables $\varphi$ on $\Omega$, whose stationary extension $S\varphi(\omega_0,x_1)$ is smooth in $x_1$:
\begin{lemma}\label{L:sDinfty}
  There exists a countable set $\sD_\Omega^\infty$ consisting of bounded, measurable functions $\varphi:\Omega\to\R$ such that $\sD_\Omega^\infty$ is dense in $L^2(\Omega)$. In addition, for all $\varphi\in\sD^\infty_\Omega$ and $\mathbb P$-a.a.~$\omega_0\in\Omega$ we have
  \begin{equation}\label{eq:st110202002003433}
    S\varphi(\omega_0,\cdot)\in C^\infty(\R^d),\qquad \mathop{\operatorname{ess\,sup}}_{\omega_0\in\Omega}\sup_{x_1\in\R}(|\partial_{x_1}^nS\varphi(\omega_0,x_1)|)<\infty\text{ for all }n=0,1,2,\ldots.
  \end{equation}
  Furthermore,   $\sD^\infty_\Omega$ is also dense in ${\mathscr H^1}(\Omega)$ and for all $\varphi\in\sD_\Omega^\infty$ we have $\partial_\omega\varphi=\partial_{x_1}S\varphi(\cdot,0)$ $\mathbb P$-a.s..
\end{lemma}
\begin{proof}
  Let $\eta\in C^\infty_c(\R)$ denote the standard mollifier, i.e.,
  \begin{equation*}
    \eta(t):=
    \begin{cases}
      C\exp(1/(t^2-1))&|t|<1,\\
      0&\text{else},
    \end{cases}
  \end{equation*}
  where $C$ is chosen such that $\int_\R\eta\,dt=1$. For each $k\in\N$ set $\eta_k(t):=k\eta(kt)$.

  \begin{enumerate}
  \item Let $\varphi\in L^2(\Omega)$. By Theorem~\ref{thm:ergodic-thm} there exists $\Omega_\varphi\subset\Omega$ with $\mathbb P(\Omega_\varphi)=1$ such that for all $\omega\in\Omega_{\varphi}$, the function $S\varphi(\omega_0,\cdot):\R\to\R$ is locally integrable. Hence, for any $\psi\in C^\infty_c(\R)$ the convolution
  \begin{equation*}
    \varphi*\psi:\Omega\to\R,\qquad     \varphi*\psi(\omega):=\begin{cases}
      \int_\R S\varphi(\omega,t)\psi(-t)\,dt&\text{if }\omega\in\Omega_\varphi,
      \\0&\text{else},
    \end{cases}
  \end{equation*}
  is well-defined and defines a measurable function with the property $S(\varphi*\psi)(\omega,\cdot)\in C^\infty(\R)$ for all $\omega\in\Omega_\varphi$.   Moreover, if $\varphi$ is bounded, then $\varphi*\psi$ satisfies \eqref{eq:st110202002003433}.

  We have $\partial_\omega(\varphi*\psi)=\varphi*\partial_{x_1}\psi$ on $\Omega_\varphi$, and thus
  \begin{equation*}
        \|\varphi*\psi\|_{L^2(\Omega)}\leq\|\psi\|_{L^1(\R)}\|\varphi\|_{L^2(\Omega)},\qquad         \|\partial_\omega(\varphi*\psi)\|_{L^2(\Omega)}\leq\|\partial_{x_1}\psi\|_{L^1(\R)}\|\varphi\|_{L^2(\Omega)}.
  \end{equation*}
This also implies that $\varphi*\psi\in {\mathscr H^1} (\Omega)$. Note that with $\psi=\eta_k$ we have
  \begin{equation}\label{eq:st11023398712}
    \|\varphi*\eta_k-\varphi\|_{L^2(\Omega)}^2\leq\int_\R\int_\Omega|\varphi(\tau_{x_1}\omega)-\varphi(\omega)|^2d\mathbb P(\omega)\eta_k(-x_1)\,dx_1.
  \end{equation}
  By the continuity of the shift on $L^2(\Omega)$ and since $(\eta_k)_k$ is a sequence of mollifiers, we deduce that $\varphi*\eta_k\to\varphi$ in $L^2(\Omega)$.

\item In this step we construct the set $\sD^\infty_\Omega$. Since $L^2(\Omega)$ is separable, there exist countably many bounded and measurable functions $\varphi^j:\Omega\to\R$, $j\in\N$ which form a dense subset of $L^2(\Omega)$. By mollifying each of these functions as described above, we obtain the countable family $\sD^\infty_\Omega:=\{\varphi^{j}_k:=\varphi^j*\eta_k\,:\,j,k\in\N\}$. By construction $\sD^\infty_\Omega$ is dense in $L^2(\Omega)$ and each $\varphi^j_k$ satisfies \eqref{eq:st110202002003433}.
\item We argue that $\sD^\infty_\Omega$ is dense in ${\mathscr H^1} (\Omega)$.   To that end let $\varphi\in {\mathscr H^1} (\Omega)$ and $\delta>0$. Choose $k\in\N$ large enough such that
  \begin{equation*}
    \|\varphi-\varphi*\eta_k\|_{L^2(\Omega)}+\|\partial_\omega\varphi-\partial_\omega(\varphi*\eta_k)\|_{L^2(\Omega)}\leq\delta/2.
  \end{equation*}
  Note that
  \begin{equation*}
    \begin{aligned}
      &\|\varphi-\varphi^j_k\|_{L^2(\Omega)}+    \|\partial_\omega\varphi-\partial_\omega\varphi^j_k\|_{L^2(\Omega)}\\
      \leq\,&     \|\varphi-\varphi*\eta_k\|_{L^2(\Omega)}+\|(\varphi-\varphi^j)*\eta_k\|_{L^2(\Omega)}\\
        &+\|\partial_\omega\varphi-\partial_\omega(\varphi*\eta_k)\|_{L^2(\Omega)}+
        \|\partial_\omega(\varphi*\eta_k)-\partial_\omega\varphi^j_k\|_{L^2(\Omega)}\\
        \leq\,&\delta/2+\|\varphi-\varphi^j\|_{L^2(\Omega)}(\|\eta_k\|_{L^1(\R)}+\|\partial_{x_1}\eta_k\|_{L^1(\R)}).
    \end{aligned}
  \end{equation*}
  Since $\{\varphi^j\}_{j\in\N}$ is dense in $L^2(\Omega)$, there exists $j\in\N$ such that the right-hand side is smaller then $\delta$. We conclude that $\sD^\infty_\Omega$ is dense in ${\mathscr H^1} (\Omega)$.
  \end{enumerate}
\end{proof}

For the definition of $\sD$ we introduce the sets $\sD_\Omega$ and $\sD_O$ with the following properties:
\begin{itemize}
\item $\sD_\Omega$ is a countable set of bounded, measurable functions on $(\Omega,\mathcal F,\mathbb P)$ that is dense in $L^2(\Omega)$.
\item $\sD_O\subset C(\overline O)$ is a countable set such that $\sD_O\cap C_c^\infty(O)$ is dense in $L^2(O)$ and $\sD_O$ contains the identity $\mathbf 1_O\equiv 1$.
\end{itemize}
We now define the set $\sD$ as the span of the $\mathbb{Q}$-linear span of simple tensor products of functions in $\sD_\Omega$ and $\sD_O$, i.e.,
\begin{gather*}
  \sD:=\operatorname{span}\sD_0=\operatorname{span}\sA,\qquad\text{where}\\
\sA:=\Big\{\varphi(\omega,x)=\varphi_\Omega(\omega)\varphi_O(x):\varphi_\Omega\in\sD_\Omega,\,\varphi^j_O\in\sD_O\Big\},\qquad\sD_0:=\Big\{\sum_{j=1}^m \lambda_j \varphi_j,\,\lambda_j\in \mathbb{Q},\,\varphi_j\in\sA\,\Big\}.
\end{gather*}
We note that by construction, $\sD$ is a dense subset of $L^2(\Omega\times O)$. We use $\sD$  as the space of two-scale test functions in our definition of stochastic two-scale convergence. In particular, we note that for any $\varphi\in \sD$ the oscillatory functions \def\unf{\mathcal T_{\vare}}
\begin{align*}
  (\unf^*\varphi)(\omega,x):=\varphi(\tau_{\frac{x_1}{\vare}}\omega,x),
\end{align*}
is measurable on $\Omega\times O$.

A slightly delicate point in stochastic two-scale convergence is the construction of a set $\Omega_0\subset\Omega$ with $\mathbb P(\Omega_0)=1$ on which the two-scale statements hold. In particular, we require that for all $\varphi\in\sD$ and $\omega_0\in\Omega_0$ the oscillatory function $\unf^*\varphi(\omega_0,\cdot)$ is well-defined and weakly convergent. To achieve this, we define  $\Omega_0$ according to the following lemma:
\begin{lemma}[The set $\Omega_0$]\label{L:defomega_0}
  There exists a measurable set $\Omega_0\subset\Omega$ with $\mathbb P(\Omega_0)=1$ s.t. for all $\omega_0\in\Omega_0$, all open, bounded intervalls $I\subset\R$, all $\varphi^1,\ldots,\varphi^N\in\sD_\Omega$ with $N\in\N$ we have
  \begin{equation}\label{eq:ergodic44}
    \lim\limits_{\vare\to 0}\int_IS(\varphi^1\varphi^2\cdots\varphi^N)(\omega_0,\tfrac{x_1}{\vare})\,dx_1=\int_\Omega\varphi^1\varphi^2\cdots\varphi^N\,d\mathbb P.
  \end{equation}
\end{lemma}
\begin{proof}
  Since $\sD_\Omega$ is countable, there are only countably many function of the from $\varphi=\varphi^1\varphi^2\cdots\varphi^N$ with $\varphi^i\in\sD_\Omega$. For each such $\varphi$ the limt \eqref{eq:ergodic44} holds for all $\omega\in\Omega\setminus E_\varphi$ where $E_\varphi$ is a null-set. Since the countable union of null-sets is again a null-set, the statement follows.
\end{proof}
\begin{lemma}\label{L:admis}
  Let $\Omega_0$ be as in Lemma~\ref{L:defomega_0}. Let $\varphi,\varphi'\in\sD$. Then for all $\omega_0\in\Omega_0$ we have
  \begin{align*}
    \lim\limits_{\vare\to 0}\int_O\unf^*(\varphi\varphi')(\omega_0,x)dx&=\int_\Omega\int_O(\varphi\varphi')(\omega,x)\,dx\,d\mathbb P(\omega).
  \end{align*}
\end{lemma}
\begin{proof}
  By definition of $\sD$ there exists $N\in\N$, $c_j,c_j'\in\R$, $\varphi_{\Omega,j},\varphi_{\Omega',j}\in\sD_\Omega$, $\varphi_{O,j},\varphi_{O,j}'\in\sD_O$ such that
  \begin{equation*}
    \varphi=\sum_{j=1}^Nc_j\varphi_{\Omega,j}\varphi_{O,j}\text{ and }\varphi'=\sum_{j=1}^Nc_j'\varphi_{\Omega,j}'\varphi_{O,j}',
  \end{equation*}
  and thus
  \begin{align*}
    \int_O\unf^*(\varphi\varphi')(\omega_0,x)dx
    =
    \sum_{j,j'=1}^Nc_jc_{j'}'\int_O\varphi_{O,j}(x)\varphi_{O,j'}(x)S(\varphi_{\Omega,j}\varphi_{\Omega,j'}')(\omega_0,\tfrac{x_1}{\vare})\,dx
  \end{align*}
  By Lemma~\ref{L:defomega_0} and Corollary~\ref{C:weak11}, each of the finitely many integrals converge, i.e.,
  \begin{equation*}
    \int_O\varphi_{O,j}(x)\varphi_{O,j'}(x)S(\varphi_{\Omega,j}\varphi_{\Omega,j'}')(\omega_0,\tfrac{x_1}{\vare})\,dx\to
    \int_\Omega\int_O\varphi_{O,j}(x)\varphi_{O,j'}(x)\varphi_{\Omega,j}(\omega)\varphi_{\Omega,j'}'(\omega)\,dx\,d\mathbb P(\omega),
  \end{equation*}
  and thus the claim follows.
\end{proof}

\begin{definition}[Stochastic two-scale convergence, cf.~\cite{Zhikov2006,Heida2017b} and {\cite[Definition 3.6]{HNV22}}]
\label{def:two-scale-conv}
Let $(u^\vare)_\vare$ be a sequence in $L^{2}(O)$, and let $\omega_0\in\Omega_0$ be fixed.  We say that $u^\vare$ converges weakly $\omega_0$-two-scale to $u\in L^2(\Omega\times O)$, and write
\begin{equation*}
  u^\vare\twos_{\omega_0}u\text{ in }L^2(\Omega\times O),
\end{equation*}
if the sequence $u^\vare$ is bounded in $L^2(O)$,
and for all $\varphi\in \sD$ we have
\begin{equation}
  \lim_{\vare\to0}\int_{O}u^\vare(x)(\unf^*\varphi)(\omega_0,x)\,dx=\int_\Omega\int_{O}u(x,\omega)\varphi(\omega,x)\,dx\,d\mathbb P(\omega).\label{eq:def-quenched-two-scale}
\end{equation}
Furthermore, we say that $u^\vare$ converges strongly $\omega_0$-two-scale to $u\in L^2(\Omega\times O)$, and write
write
\begin{equation*}
  u^\vare\twoss_{\omega_0}u\text{ in }L^2(\Omega\times O),
\end{equation*}
if $u^\vare\twos_{\omega_0} u$ in $L^2(\Omega\times O)$ and
\begin{equation}
  \lim_{\vare\to0}\int_{O}u^\vare\varphi^\vare\,dx=\int_\Omega\int_{O}u(x,\omega)\varphi(\omega,x)\,dx\,d\mathbb P(\omega),\label{eq:def-quenched-two-scale-strong}
\end{equation}
for any sequence $(\varphi^\vare)_\vare\subset L^2(\Omega\times O)$ with $\varphi^\vare\twos_{\omega_0}\varphi$ in $L^2(\Omega\times O)$. For sequences of functions with values in $\R^n$ we define weak and strong two-scale convergence componentwise.
\end{definition}

\begin{proposition}\label{P:twoscale}
  The following holds for all $\omega_0\in\Omega_0$.
  \begin{enumerate}[label=(\alph*)]
  \item \label{P:twoscale:comp} (Compactness). Let $(u^\vare)_\vare$ be a bounded sequence in $L^2(O)$. Then there exists a subsequence (still denoted by $\vare$) and $u\in L^2(\Omega\times O)$ such that $u^\vare\twos_{\omega_0}u$ and
    \begin{equation}
      \| u\|_{L^2(\Omega\times O)}\leq \liminf_{\vare\to0}\|u^\vare\|_{L^{2}(O)}.\,\label{eq:two-scale-limit-estimate}
    \end{equation}
  \item (Oscillating test-functions strongly two-scale converge). Let $\varphi\in\sD$. Then the sequence $(\varphi^\vare)_\vare$ with $\varphi^\vare(x):=\unf^*\varphi(\omega_0,x)$ strongly two-scale converges to $\varphi$ in $L^2(\Omega\times O)$ and $\|\varphi^\vare\|_{L^2(O)}\to\|\varphi\|_{L^2(\Omega\times\Omega)}$.
  \item \label{P:twoscale:weakmean} (Weak two-scale convergence implies weak convergence). $u^\vare\twos_{\omega_0}u$ two-scale in $L^2(\Omega\times O)$ implies $u^\vare\rightharpoonup \int_\Omega u(\omega,\cdot)\,d\mathbb P(\omega)$ weakly in $L^2(O)$.
  \item \label{P:twoscale:strong}(Characterization of strong two-scale). $u^\vare\twoss_{\omega_0}u$ strongly two-scale in $L^2(\Omega\times O)$ holds, if and only if $u^\vare\twos_{\omega_0}u$ two-scale in $L^2(\Omega\times O)$ and $\|u^\vare\|_{L^2(O)}\to \|u\|_{L^2(\Omega\times O)}$.
  \item \label{P:twoscale:strongimpliestrong}(Strong convergence implies strong two-scale convergence). Let $u^\vare\to u$ strongly in $L^2(O)$ and let $\varphi\in\sD$. Set $v^\vare(x):=u^\vare(x)\unf^*\varphi(\omega_0,x)$. Then $v^\vare\twoss_{\omega_0}u\varphi$ strongly two-scale in $L^2(\Omega\times O)$.
  \end{enumerate}
\end{proposition}

\begin{proof}
  \begin{enumerate}[label=(\alph*)]
  \item See \cite[Lemma 3.7]{HNV22}.
  \item This directly follows from Lemma~\ref{L:admis} and the definition of two-scale convergence.
  \item Since $(u^\vare)_\vare$ is bounded in $L^2(O)$ and $\sD_O\subset L^2(O)$ is dense, it suffices to show that $\int_Ou^\vare\varphi\,dx\to \int_O\big(\int_\Omega u(\omega,x)\,d\mathbb P(\omega)\big)\varphi(x)\,dx$ for all $\varphi\in\sD_O$. The latter follows since any $\varphi\in\sD_O$ is also an element of $\sD$ (the identity function is assumed to belong to $\sD_\Omega$) and satisfies $\varphi=\unf^*\varphi$.
  \item The direction ``$\Rightarrow$'' is trivial. The argument for ``$\Leftarrow$'' is as follows:    By density there exists $u^\delta\in\sD$ such that $\|u^\delta-u\|_{L^2(\Omega\times O)}\leq\delta$. Set $u^{\delta,\vare}(x):=\unf^*u^\delta(\omega_0,\cdot)$. Then $u^{\delta,\vare}$ strongly two-scale converges to $u^\delta$. Moreover, Lemma~\ref{L:admis} implies that $\|u^{\delta,\vare}\|_{L^2(O)}\to \|u^\delta\|_{L^2(\Omega\times O)}$. Let $v^\vare$ converge weakly two-scale to $v$. Then
    \begin{align*}
      \int_O u^\vare v^\vare\,dx=      &\int_O u^{\delta,\vare} v^\vare-\int_O(u^{\delta,\vare}-u^\vare)v^\vare\\
      =&\int_O v^\vare\unf^* u^\delta(\omega_0,x)\,dx-\int_O(u^{\delta,\vare}-u^\vare)v^\vare.
    \end{align*}
    As $\vare\to0$, the first term on the right-hand side converges to $\int_\Omega\int_O vu^\delta\,dx\,d\mathbb P(\omega)$. Since $(v^\vare)_\vare$ is bounded in $L^2(O)$, we conclude that for some $C>0$ we have
    \begin{equation}\label{eq:st10043488}
      \limsup\limits_{\vare\to 0}\big|\int_O u^\vare v^\vare\,dx-\int_\Omega\int_O uv\,dx\,d\mathbb P|\leq C      \limsup\limits_{\vare\to 0}\|u^{\delta,\vare}-u^\vare\|_{L^2(O)}.
    \end{equation}
    By expanding the square we have
    \begin{equation*}
      \|u^{\delta,\vare}-u^\vare\|_{L^2(O)}^2=\|u^{\delta,\vare}\|_{L^2(O)}^2+\|u^\vare\|_{L^2(O)}^2-2\int_O u^{\delta,\vare} u^\vare\,dx.
    \end{equation*}
    In all three terms we can pass to the limit $\vare\to 0$ and obtain
    \begin{equation*}
      \lim\limits_{\vare\to 0}\|u^{\delta,\vare}-u^\vare\|_{L^2(O)}=\|u^{\delta}-u\|_{L^2(\Omega\times O)}\leq \delta.
    \end{equation*}
    We may combine this with \eqref{eq:st10043488} and pass to the limit $\delta\to 0$. The claim follows.
  \item First note that $v^\vare\twos_{\omega_0} v:=u\varphi$ weakly two-scale. To prove strong two-scale convergence, we argue that $\|v^\vare\|_{L^2(O)}\to \|v\|_{L^2(\Omega\times O)}$. Note that
    \begin{equation*}
      \|v^\vare\|_{L^2(O)}^2=\int_O |u^\vare|^2\unf^*(\varphi^2)(\omega_0,x)\,dx=\int_O |u|^2\unf^*(\varphi^2)(\omega_0,x)\,dx+\int_O(|u^\vare|^2-|u|^2)\unf^*(\varphi^2)(\omega_0,x)\,dx.
    \end{equation*}
    Since $|u|^2\in L^1(O)$, the first term on the right-hand side converges to $\int_\Omega\int_O|u\varphi|^2\,dx\,d\mathbb P(\omega)$ thanks to Corollary~\ref{C:weak11}, Lemma~\ref{L:defomega_0}, and Lemma~\ref{L:admis}. On the other hand the we have (since $\varphi$ is bounded)
    \begin{equation*}
      |\int_O(|u^\vare|^2-|u|^2)\unf^*(\varphi^2)(\omega_0,x)\,dx|\leq \|\varphi^2\|_{L^\infty(\Omega\times O)}\|u^\vare+u\|_{L^2(O)}\|u^\vare-u\|_{L^2(O)}\to 0.
    \end{equation*}
  \end{enumerate}
\end{proof}

\begin{lemma}[Approximation w.r.t.~strong two-scale convergence]\label{L:approximation}
  For all $\omega_0\in\Omega_0$ and every $u\in L^2(\Omega\times O)$ there exists a sequence $(u^\vare)_\vare\subset C^\infty_c(O)$ such that $u^\vare\twoss_{\omega_0}u$ strongly two-scale in $L^2(\Omega\times O)$.
\end{lemma}
\begin{proof}
  For all $\delta>0$ choose $u^\delta$ in the span of $\{\varphi_\Omega\varphi_O\,:\,\varphi_\Omega\in\sD^\infty_\Omega,\varphi_O\in\sD_O\cap C^\infty_c(O)\}$ with $\|u^\delta-u\|_{L^2(\Omega\times O)}<\delta$.
  This is possible, since $\sD^\infty_\Omega$ and $\sD_O\cap C^\infty_c(O)$ are dense in $L^2(\Omega)$ and $L^2(O)$, respectively. Set $u^{\delta,\vare}(x):=\unf^* u^\delta(\omega_0,x)$ and note that $u^{\delta,\vare}\in C^\infty_c(O)$. Then for all $\delta>0$ we have $u^{\delta,\vare}\twoss_{\omega_0} u^\delta$ as $\vare\to 0$.

  In the following we shall deduce the existence of $(u^\vare)_\vare=(u^{\delta(\vare),\vare})_\vare$ by a diagonal sequence argument. In order to do so, we recall the metric characterization of weak two-scale convergence from \cite[Lemma 3.8]{HNV22}: Consider $\sD$ as a normed vector space with norm $\|\cdot\|_{L^2(\Omega\times O)}$ and denote by $\sD^*$ its dual. Note that the operators
  \begin{align*}
    &J^\vare_{\omega_0}:L^2(O)\to \sD^*,\qquad (J^\vare_{\omega_0}v)(\varphi):=\int_O v(x)\unf^*\varphi(\omega_0,x)\,dx\\
    &J^0:L^2(\Omega\times O)\to \sD^*,\qquad (J^0v)(\varphi):=\int_\Omega\int_O v(\omega,x)\varphi(\omega,x)\,dx\,d\mathbb P(\omega)
  \end{align*}
  are linear, bounded and injective. We observe that a bounded sequence $(v^\vare)_\vare\subset L^2(O)$ weakly two-scale converges to $v$ if and only if $J^\vare_{\omega_0}v^\vare\to J_0v$ pointwise. Let $(\varphi_j)_{j\in\N}$ be an enumeration of the countable set $\{\frac{\varphi}{\|\varphi\|_{L^2(\Omega\times O)}}\,:\,\varphi\in\sD_0\}$ and define for $U,V\in\sD^*$ the metric
  \begin{equation*}
    {\mathrm d}(U,V):=\sum_{j\in\N}2^{-j}\frac{|U(\varphi_j)-V(\varphi_j)|}{|U(\varphi_j)-V(\varphi_j)|+1}.
  \end{equation*}
  Then we see that for any bounded sequence $(v^\vare)_\vare\subset L^2(O)$ and $v\in L^2(\Omega\times O)$ we have
  \begin{equation*}
    v^\vare\twos_{\omega_0}v\qquad\Leftrightarrow\qquad {\mathrm d}(J^\vare_{\omega_0}v^\vare, J^0v)\to 0.
  \end{equation*}
  Furthermore, for any $v,v'\in L^2(\Omega\times O)$ we have
  \begin{equation*}
    {\mathrm d}(J^0v, J^0v')\leq 2\|v-v'\|_{L^2(\Omega\times O)}.
  \end{equation*}
  After these preparations we may consider
  \begin{equation*}
    c^{\delta,\vare}:=\big|\|u^{\delta,\vare}\|_{L^2(O)}-\|u\|_{L^2(\Omega\times O)}\big|+{\mathrm d}(J^\vare_{\omega_0}u^{\delta,\vare}, J^0u).
  \end{equation*}
  Then we have $      \limsup\limits_{\vare\to 0}c^{\delta,\vare}\leq\big|\|u^{\delta}\|_{L^2(\Omega\times O)}-\|u\|_{L^2(\Omega\times O)}\big|+{\mathrm d}(J^0u^{\delta}, J^0u)$, and thus
  \begin{equation*}
    \limsup\limits_{\delta\to 0}\limsup\limits_{\vare\to 0}c^{\delta,\vare}=0.
  \end{equation*}
  By a standard diagonalization argument there exists $\delta(\vare)$ with $\lim_{\vare\to 0}\delta(\vare)=0$ such that $c^{\delta(\vare),\vare}\to 0$. Thus the sequence $u^\vare:=u^{\delta(\vare),\vare}$ satisfies
  \begin{equation*}
    \|u^\vare\|_{L^2(O)}\to \|u\|_{L^2(\Omega\times O)},\qquad u^\vare\twos_{\omega_0}u.
  \end{equation*}
  In view of Proposition~\ref{P:twoscale} \ref{P:twoscale:strong} we conclude that we even have strong two-scale convergence and the proof is complete.
\end{proof}

\begin{lemma}[Continuity and lower semicontinuity of quadratic, convex functionals]\label{L:twoscale:quadratic}
Let $Q:\Omega\times\R^n\to\R$ be measurable and assume that for all $\omega\in\Omega$ the map $\xxi\mapsto Q(\omega,\xxi)$ is quadratic and satisfies
  \begin{equation*}
    0\leq Q(\omega,\xxi)\leq C|\xxi|^2\qquad\text{for all }\xxi\in\R^n,
  \end{equation*}
  where $C$ is some positive constant independent of $\omega$. Then there exists $\Omega_Q\subset\Omega_0$ with $\mathbb P(\Omega_Q)=1$ such that for all $\omega_0\in\Omega_Q$ the following holds:
  \begin{enumerate}[label=(\alph*)]
  \item Suppose that $(\xxi^\vare)_\vare\subset L^2(O;\R^n)$ weakly two-scale converges to $\xxi\in L^2(\Omega\times O;\R^n)$. Then
    \begin{equation*}
      \liminf\limits_{\vare\to 0}\int_OQ(\tau_{\frac{x_1}{\vare}}\omega_0,\xxi^\vare(x))\,dx\geq \int_\Omega\int_O Q(\omega,\xxi(\omega,x))\,dx\,d\mathbb P(\omega).
    \end{equation*}
  \item Suppose that $(\xxi^\vare)_\vare\subset L^2(O;\R^n)$ strongly two-scale converges to $\xxi\in L^2(\Omega\times O;\R^n)$. Then
    \begin{equation*}
      \lim\limits_{\vare\to 0}\int_OQ(\tau_{\frac{x_1}{\vare}}\omega_0,\xxi^\vare(x))\,dx=\int_\Omega\int_O Q(\omega,\xxi(\omega,x))\,dx\,d\mathbb P(\omega).
    \end{equation*}
  \end{enumerate}
\end{lemma}
\begin{proof}
  Define $\mathbb L:\Omega\to\R^{n\times n}_{\sym}$ by the identity $Q(\omega,\xxi)=\mathbb L(\omega)\xxi\cdot \xxi$ and note that $\mathbb L_{ij}$ are essentially bounded. Thanks to Theorem \ref{thm:ergodic-thm} we can find a set of full-measure $\Omega_Q\subset\Omega_0$ such that for all $i,j,k,l=1,\ldots,n$ and all $\varphi,\varphi'\in\sD$ we have
  \begin{equation}\label{eq:st00227474}
    \int_O\unf^*(\mathbb L_{ij}\mathbb L_{kl}\varphi\varphi')(\omega_0,x)\,dx\to \int_\Omega\int_O\mathbb L_{ij}(\omega)\mathbb L_{kl}(\omega)\varphi(\omega,x)\varphi'(\omega,x)\,dxd\mathbb P(\omega)
  \end{equation}
  for all $\omega_0\in\Omega_Q$. For the rest of the proof we assume that $\omega_0\in\Omega_Q$.

  As a preliminary step we claim the following: Let $(\xxi^\vare)_\vare,(\hat\xxi^\vare)_\vare\in L^2(O;\R^n)$ and assume that $\xxi^\vare\twos_{\omega_0}\xxi$ and $\hat\xxi^\vare\twoss_{\omega_0}\hat\xxi$ weakly and strongly two-scale, respectively. Then
  \begin{equation}\label{eq:st:st:123}
    \int_O\mathbb L(\tau_{\frac{x_1}{\vare}}\omega_0)\xxi^\vare(x)\cdot \hat\xxi^\vare(x)\,dx\to\int_\Omega\int_O\mathbb L(\omega)\xxi\cdot \hat \xxi\,dx\,d\mathbb P(\omega).
  \end{equation}
  To see this, we first note that by \eqref{eq:st00227474} for all $\varphi\in\sD$ the sequence $\unf^*(\mathbb L_{ij}\varphi)(\omega_0,\cdot)$ strongly two-scale converges to $\mathbb L_{ij}\varphi$. Hence, since $\eta^\vare(x):=\mathbb L(\tau_{\frac{x_1}{\vare}}\omega_0)\xxi^\vare(x)$ is bounded in $L^2(O;\R^n)$, we conclude that $\eta^\vare$ weakly two-scale converges to $\mathbb L\xxi$. Now the claim follows from the definition of strong two-scale convergence.

  Note that \eqref{eq:st:st:123} directly implies part (b) of the lemma. To prove part (a), we proceed as follows: By Lemma~\ref{L:approximation} we can find a sequence $(\hat\xxi^\vare)_\vare$ that strongly two-scale converges to $\xxi$. By expanding the square, we see that
    \begin{align*}
      \int_OQ(\tau_{\frac{x_1}{\vare}}\omega_0,\xxi^\vare)\,dx-      \int_OQ(\tau_{\frac{x_1}{\vare}}\omega_0,\hat\xxi^\vare)\,dx
      =&\int_OQ(\tau_{\frac{x_1}{\vare}}\omega_0,\xxi^\vare-\hat\xxi^\vare)\,dx+2\int_O\mathbb L(\tau_{\frac{x_1}{\vare}}\omega_0)(\xxi^\vare-\hat\xxi^\vare)\cdot\hat \xxi^\vare\,dx\\
                                                                                                                                \geq&2\int_O\mathbb L(\tau_{\frac{x_1}{\vare}}\omega_0)(\xxi^\vare-\hat\xxi^\vare)\cdot\hat \xxi^\vare\,dx
    \end{align*}
  In view of \eqref{eq:st:st:123} we can pass to the limit on the right-hand side. Since $\xxi^\vare-\hat\xxi^\vare$ weakly two-scale converges to $0$, we deduce that
  \begin{equation*}
    \liminf\limits_{\vare\to 0}\int_OQ(\tau_{\frac{x_1}{\vare}}\omega_0,\xxi^\vare)\,dx\geq\limsup_{\vare\to 0}\int_OQ(\tau_{\frac{x_1}{\vare}}\omega_0,\hat\xxi^\vare)\,dx=\int_\Omega\int_OQ(\omega,\xxi)\,dxd\mathbb P(\omega),
  \end{equation*}
  where we also used part (b) of the lemma in the last step .
\end{proof}

\begin{lemma}[Two-scale limits of gradients]\label{bound lemma}
  For all $\omega_0\in\Omega_0$ the following holds:
  \begin{enumerate}[label=(\alph*)]
  \item   Let $(u^\vare)_\vare$ be a sequence that weakly converges in $H^1(O)$ to a limit $u\in H^1(O)$. Then there exists $\chi\in L^2(O;L^2_0(\Omega))$ and a subsequence of $(u^\vare)_\vare$ (still denoted by $\vare$) such that
    \begin{equation*}
      \partial_j u^\vare\twos_{\omega_0}
      \begin{cases}
        \partial_1u+\chi&\text{if }i=1,\\
        \partial_ju&\text{if }j\in\{2,\ldots,d\}
      \end{cases}\qquad\text{weakly two-scale in }L^2(\Omega\times O).
    \end{equation*}
  \item For any $\chi\in L^2(O;L^2_0(\Omega))$ there exists a sequence $(\varphi^\vare)_\vare\subset C^\infty_c(O)$ such that
    \begin{align*}
      &|\varphi^\vare|+\sum_{j=2}^d|\partial_j\varphi^\vare|\to 0\text{ uniformly, and}\\
      &\partial_1\varphi^\vare\twoss_{\omega_0} \chi\text{ strongly two-scale in $L^2(\Omega\times O)$.}
    \end{align*}
  \end{enumerate}
\end{lemma}
\begin{proof}
  For the proof it is convenient to define
  \begin{equation*}
    \sD^\infty_{\Omega,0}:=\{\varphi-\int_\Omega\varphi\,d\mathbb P\,:\,\varphi\in\sD^\infty_{\Omega}\}.
  \end{equation*}
  Note that $\sD^\infty_{\Omega,0}$ is dense in $L^2_0(\Omega)$ and contained in $\sD_\Omega$ (since $\boldsymbol{1}_\Omega\in\sD_\Omega$ by assumption).
  Hence, the set
  \begin{equation*}
    \mathring{\sD}:=\operatorname{span}\big\{\varphi=\varphi_\Omega\varphi_O\,:\,\varphi_\Omega\in \sD^\infty_{\Omega,0},\,\varphi_O\in C^\infty_c(O)\,\Big\}
  \end{equation*}
  is dense in $L^2_0(\Omega)\otimes L^2(O)$.

  \begin{enumerate}[label=(\alph*)]
  \item By compact embedding we have $u^\vare\to u$ strongly in $L^2(O)$ and thus also strongly two-scale. By Proposition~\ref{P:twoscale} \ref{P:twoscale:comp} and Proposition~\ref{P:twoscale} \ref{P:twoscale:weakmean}, we may pass to a subsequence and find $\xxi\in L^2(O;L^2_0(\Omega;\R^d))$ such that $\partial_j u^\vare\twos_{\omega_0}\partial_j u+\xxi_j$ for $j=1,\ldots,d$. This already proves the claim for $j=1$. In the following we prove the claim in the case $j=2,\ldots,d$. Let $\varphi\in\mathring\sD$ and consider $\varphi^\vare(x):=\unf^*\varphi(\omega,x)$. Then $\varphi^\vare\twoss \varphi$ strongly two-scale (by Proposition~\ref{P:twoscale} \ref{P:twoscale:strongimpliestrong}) and thus
    \begin{align}
      \int_O\partial_ju^\vare\varphi^\vare\,dx\to &\int_\Omega\int_O(\partial_ju+\xxi_j)\varphi\,dx\,d\mathbb P(\omega)=\int_\Omega\int_O\xxi_j\varphi\,dx\,d\mathbb P(\omega).\label{eq:st:11883362626}
    \end{align}
    On the other hand, by construction we have $\varphi^\vare\in C^\infty_c(O)$ with $\partial_j\varphi^\vare(x)=\unf^*(\partial_j\varphi)(\omega_0,x)$ and $\partial_j\varphi^\vare\twoss \partial_j\varphi$. Hence,
    \begin{align*}
      \int_O\partial_ju^\vare\varphi^\vare\,dx = &-\int_Ou^\vare(x)\partial_j\varphi^\vare\,dx\,\to\,-\int_\Omega\int_Ou\partial_j\varphi_O\varphi_\Omega\,dx\,d\mathbb P(\omega)=0,
    \end{align*}
    where in the last identity we used the fact that $\int_\Omega\partial_j\varphi\,d\mathbb P=0$. We conclude that the right-hand side of \eqref{eq:st:11883362626} is zero. Since $\mathring{\sD}$ is dense in $L^2_0(\Omega)\otimes L^2(O)$ and since $\xxi_j$ belongs to the latter space, we conclude that $\xxi_j=0$.
  \item  Since $\sD^\infty_\Omega$ and $\sD_O\cap C^\infty_c(O)$ are dense in ${\mathscr H^1}(\Omega)$ and $L^2(O)$ respectively, for all $\delta>0$ we can find a function of the form
    \begin{equation*}
      \varphi^\delta(\omega,x)=\sum_{j=1}^Nc_j\varphi_{\Omega,j}(\omega)\varphi_{O,j}(x),\qquad N\in\N,\,c_j\in\R,\,\varphi_{\Omega,j}\in\sD^\infty_\Omega, \varphi_{O,j}\in\sD_O\cap C^\infty_c(O),
    \end{equation*}
    such that $\|\partial_\omega\varphi^\delta-\chi\|_{L^2(\Omega\times O)}\leq\delta$. Consider
    \begin{equation*}
      \varphi^{\vare,\delta}(x):=\vare\unf^*\varphi^\delta(\omega_0,x).
    \end{equation*}
    Then $\varphi^\vare\in C^\infty_c(O)$ and
    \begin{equation*}
      \partial_j\varphi^{\vare,\delta}(x)=
      \begin{cases}
        \unf^*(\partial_\omega\varphi^\delta)(\omega_0,x)+\vare\unf^*(\partial_j\varphi^\delta)(\omega_0,x)&\text{if }j=1,\\
        \vare\unf^*(\partial_j\varphi^\delta)(\omega_0,x)&\text{if }j=2,\ldots,d.
      \end{cases}
    \end{equation*}
    We conclude that
    \begin{gather*}
        |\varphi^{\vare,\delta}|+\sum_{j=2}^d|\partial_j\varphi^{\vare,\delta}|\to 0\text{ uniformly and }\\
        \partial_1\varphi^{\vare,\delta}\twoss \partial_\omega\varphi^\delta\text{ strongly two-scale as $\vare\to 0$.}
    \end{gather*}
    By passing to a diagonal sequence as in the proof of Lemma~\ref{L:approximation}, we obtain a sequence $(\varphi^\vare)_\vare\subset C^\infty_c(O)$ satisfying
    \begin{gather*}
      |\varphi^{\vare}|+\sum_{j=2}^d|\partial_j\varphi^{\vare}|\to 0\text{ uniformly and }\\
      \partial_1\varphi^{\vare}\twoss \chi\qquad\text{strongly two-scale as $\vare\to 0$.}
    \end{gather*}
  \end{enumerate}
\end{proof}


\addcontentsline{toc}{section}{References}
\bibliography{hom_quant}
\bibliographystyle{acm}
\end{document}